\documentclass[11 pt]{amsart}
\usepackage{ color, graphicx}
\usepackage{amssymb,amsmath,amsthm}
\usepackage{mathrsfs, dsfont,latexsym}
\usepackage{paralist}

\usepackage{hyperref}

\theoremstyle{plain}

\newtheorem{proposition}{Proposition}[section]
\newtheorem{claim}[proposition]{Claim}
\newtheorem{lemma}[proposition]{Lemma}

\newtheorem{remark}[proposition]{Remark}
\numberwithin{equation}{section}
\newtheorem{cor}{Corollary}[section]

\newcommand{\NN}{\mathbb{N}}

\newcommand{\RR}{\mathbb{R}}

\newcommand{\p}{\partial}

\newcommand{\De}{\Delta}

\newcommand{\SCB}{\mathscr{B}}
\newcommand{\SCF}{\mathscr{F}}
\newcommand{\SE}{\mathscr{E}}

\newcommand{\CA}{\mathcal{A}}

\newcommand{\CD}{\mathcal{D}}
\newcommand{\CE}{\mathcal{E}}
\newcommand{\CH}{\mathcal{H}}

\newcommand{\CF}{\mathcal{F}}

\newcommand{\CR}{\mathcal{R}}

\newcommand{\half}{\frac{1}{2}}
\newcommand{\D}{{D}}

\newcommand{\CN}{\mathcal{N}}
\newcommand{\SD}{\mathscr{D}}

\newcommand{\CK}{\mathcal{K}}

\newcommand{\CX}{\mathcal{X}}

\newcommand{\CU}{\mathcal{U}}

\newcommand{\CS}{\mathcal{S}}
\newcommand{\TOM}{\Omega}
\newcommand{\TCS}{\mathcal{S}}

\newcommand{\W}{\tilde{W}}

\newcommand{\intl}{\int\limits}
\newcommand{\lf}{\left}
\newcommand{\rt}{\right}
\newcommand\<{\langle}\renewcommand\>{\rangle}

\newcommand\nc{\newcommand}
\nc\hd{\widehat{D}}
\nc\be{\begin{equation}}
\nc\ee{\end{equation}}
\nc\kp{\kappa}
\nc\8{\infty}

\DeclareMathSymbol{\Gamma}{\mathalpha}{letters}{"00}
\DeclareMathSymbol{\Theta}{\mathalpha}{letters}{"02}
\DeclareMathSymbol{\Lambda}{\mathalpha}{letters}{"03}
\DeclareMathSymbol{\Omega}{\mathalpha}{letters}{"0A}
\definecolor{cr}{rgb}{1,0,0}

\newtheorem{thm}{Theorem}[section]
\newtheorem{lem}[thm]{Lemma}
\newtheorem{prop}[thm]{Proposition}

\newtheorem{df}[thm]{Definition}
\newtheorem{rem}[thm]{Remark}

\numberwithin{equation}{section}

\title{Global existence for capillary water waves}
\author{P. Germain, N. Masmoudi, J. Shatah}
\email{pgermain@cims.nyu.edu, masmoudi@cims.nyu.edu, shatah@cims.nyu.edu}
\thanks{P. Germain is partially supported by NSF grant DMS-1101269 and a start-up grant from the Courant Institute.}
\thanks{N. Masmoudi is partially supported by NSF grant DMS-1211806 grant.}
\thanks{J. Shatah is partially supported by NSF grant DMS-1001674 grant.}

\begin{document}

\maketitle


\newcommand\eps{\varepsilon}

\section*{Abstract}
Consider the capillary water waves equations, set in the whole space with infinite depth, 
and consider small data (i.e. sufficiently close to zero velocity, and constant height of the water).
We prove global existence and scattering. The proof  combines  in a novel way the energy method
with a cascade of energy estimates, the space-time resonance method and commuting  vector fields. 
 
\section{Introduction}

In this manuscript we consider the global existence and asymptotic behavior of surface waves for an 
irrotational, incompressible, and inviscid  fluid in the presence of surface tension.  The fluid velocity is given by Euler's equation in a domain $\CU$:
 \[
\CU=\bigcup_t\SD_t, \quad   \SD = \SD_t  = \{ (x,z) = (x_1,x_2,z) \in \mathbb{R}^3 ,  z\leq h(x,t) \},  
\]
and the free boundary of the fluid at time $t$
\[
\SCB=\SCB_t= \{ (x,h(x,t)),  x \in \mathbb{R}^2\}=\p\SD
\]
moves by the normal velocity  of the fluid.  The surface tension is assumed to be proportional (by the coefficient $c$) to the mean curvature $\kappa$ of $\SCB$ and we neglect the presence of gravity.
In this setting the Euler equation for the fluid velocity $v$,  and the  boundary conditions are  given by


\begin{subequations}\label{CW}
\begin{align}
&\begin{cases}
\D_tv\stackrel{def}{=}\p_t v + \nabla_v v = -\nabla p  \quad  (x, z) \in \SD,\\
\nabla \cdot v =0 \quad (x, z) \in \SD,\\
\end{cases}\\
&
\begin{cases}
\p_t h+ \nabla_v (h-z)=0 \quad (x,z)\in \SCB,\\
p= c\kappa,  \quad (x,z)\in \SCB.
\end{cases}
\end{align}
\end{subequations}
where $\nabla =(\p_1, \p_2,\p_z)$.  Since the flow is assumed to be irrotational, the Euler equation can be reduced to an equation  on the boundary and thus the system of equations (E--BC) reduces to a system 
defined on $\SCB$. This is achieved by introducing the potential $\psi_{\mathcal{H}}$ where $v= \nabla \psi_{\mathcal{H}}$. Denoting the trace of the potential on the free 
boundary by $\psi(x,t) = \psi_{\mathcal{H}} (x,h(x,t),t)$,  the system of equations for $\psi$ and $h$ are (see for instance~\cite{SuSu})
\begin{equation}\label{ww}
\tag{WW}
\left\{
\begin{array}{l} \partial_t h = G(h) \psi \\
\partial_t \psi = c \kappa -\frac{1}{2} |\p \psi|^2 + \frac{1}{2(1+|\p h|^2)} \left(G(h) \psi + \p h \cdot \p \psi \right)^2 \\
(h,\psi)(t=0) = (h_0,\psi_0).  \end{array}
\right.
\end{equation}
where 
$$
G(h) \psi = \sqrt{1 + |\p h|^2} \mathcal{N}(h) \psi \sim \Lambda \psi + \hbox{(quadratic)},
$$
$\mathcal{N}$ being the Dirichlet-Neumann operator associated with $\SD$; $\p= (\p_1,\p_2)$;
$\Lambda = |\p|$; and where  the mean curvature can be expressed as
$$
\kappa = \frac{1}{2} \p \cdot \left( \frac{\p h}{\sqrt{1+|\p h |^2}} \right) \sim 
  \frac12  \Delta h + \hbox{(cubic)}.
$$  
In the sequel we take $c=2$ for simplicity.
The dispersive nature of \eqref{ww} is revealed by writing the linearization of this system around $(h,\psi)=(0,0)$:
\begin{equation} \label{Lin}
\left\{
\begin{array}{l} \partial_t h = \Lambda \psi, \\
\partial_t \psi = \Delta h,
\end{array} \right. 
\end{equation}
where $\mathfrak{R}_i$ are  at least quadratic in $(h,\psi)$. By setting
$$
u \overset{def}{=} \Lambda^{1/2} h + i \psi \;\;\;\mbox{and}\;\;\; u_0 \overset{def}{=} \Lambda^{1/2} h_0 + i \psi_0,
$$
the above system can be written as a single equation
\begin{equation}\label{WWlin}
\left\{
\begin{array}{l}
\partial_t u = - i \Lambda^{3/2} u +  \mathfrak{R} \\
u(t=0) = u_0.
\end{array}
\right.,
\end{equation} 
where $\mathfrak{R} = \Lambda^\half\mathfrak{R_1}+ i \mathfrak{R_2}$

\subsection{Known results}

A great number of works has been devoted to local well-posedness of the water wave problem. These works consider the problem with gravity, capillarity, or both; 
it is furthermore possible to allow domains of finite depth. The small data problem was first addressed by Nalimov~\cite{NA74} 
(see also H.~Yoshihara \cite{YO82}), but the first breakthrough in solving the local well-posedness problem with general data is due to Wu~\cite{WU97,WU99}. 
There are many other works on local well-posedness: we mention in particular Craig \cite{CA85}, 
Christodoulou and Lindblad \cite{CL00}, Lannes \cite{LA05}, Coutand and Shkoller \cite{CS05}, Ambrose and Masmoudi \cite{AM05,AM07,AM09}, Shatah and Zeng \cite{SZ06},
Alzazard, Burq and Zuily~\cite{ABZ}.

In connection with the local regularity problem, Christianson, Hur and Staffilani~\cite{CHS} and Alazard, Burq and Zuily~\cite{ABZ2} have been able to prove recently
a nonlinear smoothing effect.

Fewer works address the global evolution problem; all results are then restricted to small data.
The first progress in this direction is due to Wu~\cite{Wu0}, who proved almost global existence of gravity waves in dimension 2. 
The authors of the present paper~\cite{GMS3} could then prove global existence of gravity water waves in dimension 3. A very similar result was obtained later
by Wu~\cite{Wu}, using a different proof. We shall come back later to the methods of proof employed.

Another line of research has been concerned with traveling waves. We focus on fully localized traveling waves in dimension 3
(that is, two-dimensional interface). Such waves have been constructed recently by Groves and Sun~\cite{GS}, and Buffoni, Groves, Sun and Wahl\'en~\cite{BGSW}. The setting
is that of gravity-capillary waves of finite depth: denote respectively $c$, $g$ and $h$ for the capillarity coefficient, gravity, and the depth of the fluid,
and define the Bond number $\beta = \frac{c}{g h^2}$. The aforementioned works show that if $\beta > \frac{1}{3}$, traveling waves of arbitrarily small size in $L^2$ exist. 
This should be contrasted with our main theorem here: we prove that for $c>0$, $g=0$, $h=\infty$, small data lead to scattering for large time. In our main theorem however, 
smallness is expressed in weighted $L^2$ spaces, and it is not clear to which weighted $L^2$ spaces the traveling waves of Buffoni, Groves, Sun and Wahl\'en belong.  For a more general account of traveling waves, we refer to the recent book of Constantin~\cite{Constantin}.

\subsection{Main theorem}

To state our main result we need to introduce the following notation: let
$$
\p \overset{def}{=} (\partial_1,\partial_2), \;\;\;\;   \Omega \overset{def}{=} x^1 \partial_2 - x^2 \partial_1=\omega^i\p_i,   \;\;\;\; 
\mbox{ and }\;\;\;\; \mathcal{S} \overset{def}{=} \frac{3}{2}t\partial_t + x^i\p_i,
$$
and let $\Gamma$ denote any of the fields $\Gamma= \mathcal{S}, \Omega$, or $\p^3$, where $\p^3=\p_1^{i_1}\p_2^{i_2}$ with $i_1 +i_2 =3$. We adopt the multiindex
notation: if $\gamma = (\gamma_1,\gamma_2,\gamma_3) \in \mathbb{N}^3$,
$$ 
\Gamma^\gamma = \mathcal{S}^{\gamma_1} \Omega^{\gamma_2} \p^{3 \gamma_3}.
$$
We shall also need the spatial part of $\mathcal{S}$
$$
\Sigma \overset{def}{=} x^i\p_i.
$$
Recall the classical Sobolev space $W^{k,p}$, whose norm reads
$$
\left\| u \right\|_{W^{k,p}} \overset{def}{=} \left\| \left( 1 + \Lambda^k \right) u \right\|_p.
$$
For $k$ a positive real number and $\ell$ an integer, define the weighted Sobolev space $W^{k,p}_\ell$ by its norm
$$
\left\| u \right\|_{W^{k,p}_\ell} \overset{def}{=} \sum_{|\gamma|\leq \ell} \left\| \Gamma^\gamma u \right\|_{W^{k,p}}.
$$
Recall that $u = \Lambda^{1/2} h + i \psi$, where $(h,\psi)$ are given by \eqref{ww}.

\begin{thm}
Assume that the initial data  $u_0$ satisfies 
\begin{equation}
\label{conditiondata}
\| \Lambda^{1/2}  u_{0} \|_{W^{9/2,2}_{2K} ( \RR^2  )} 
+ \|  \Lambda^{\alpha_*}   u_0 \|_{W^{0,2}_{2K}( \RR^2  ) } 
<  \epsilon,
\end{equation} 
where $K\geq 10$, and $\epsilon,\alpha_*> 0$ are sufficiently small.  
Then there exists a global solution $u$ of \eqref{WWlin} such that 
$ \| u\|_X \lesssim \eps $ (where the  $ \| \, \|_X$ norm is defined in section  \ref{tech}). 
Furthermore, this solution scatters, i.e. there exists a solution $u_\ell$ of the linearized problem
$$
\partial_t u_\ell = - i \Lambda^{3/2} u_\ell 
$$
such that
$$
\left\| \Lambda^{1/2} \left( u(t) - u_\ell(t) \right) \right\|_2 \rightarrow 0 \qquad \mbox{as $t\rightarrow \infty$}.
$$
\end{thm}

\begin{remark}
Consider $h_0 \in \mathcal{S}(\mathbb{R}^2)$ and $v_0 \in \mathcal{S}(\mathcal{D}_0)$ such that $\operatorname{curl} v = \operatorname{div} v = 0$. 
Then the data $\mu (h_0,v_0)$ satisfy the hypotheses of the theorem if $\mu$ is small enough.  However note that if  we took  $\alpha_* =0$ in the theorem above, namely that 
$ \|    u_0 \|_{W^{0,2}_{2K}( \RR^2  ) } 
<  \epsilon $ 
then  $\mu (h_0,v_0)$ would not satisfy the hypotheses of the theorem unless 
 a moment condition is satisfied by   $v_0$. 
Thus  by requiring  $\alpha_*>0$ we allow  a larger class of data at the expense of complicating  the  proof, since such a condition on the initial data  makes $\p h$ 
decay at a polynomial rate slower than $1/t$ (even at the linear level!).
\end{remark}
 
The proof of this result is based on combining  in a new way  1) the vector field method (see Klainerman \cite{Klainerman}) which is based on invariances of the equation, 2) A cascade of energy 
estimates,  3) dispersive bounds  and  4) 
the space time resonance method \cite{GMS1,GMS3}, which is based on resonant interactions of waves.  Below we give a brief and {\it simplified} description of the ideas of the proof. 

\subsection{Discussion of the difficulties and the method} 

There are several difficulties that distinguish this problem from the gravity water waves problem, namely:

\paragraph{ \it 1) Low frequencies} The group velocity for  capillary waves is given by $\frac32|\xi|^{1/2}$ while  for  gravity water waves it is given by
 $\frac12|\xi|^{-1/2}$.   This makes high frequency capillary waves disperse faster than gravity waves, but it also causes low frequency waves to decay very slowly.  
For the long time behavior the low rate of dispersion of low frequencies is more problematic.  In particular, in our setting, the quantity $\| \nabla h \|_{L^\infty}$ decays 
at a polynomial rate slower than $1/t$. This slow decay combined with the commutation properties of the 
vector fields causes a cascading growth on the weighted energy 
estimates.

\paragraph{\it  2) Quadratic resonances}  A second difficulty  is linked with the fact that $\xi \to |\xi|^{3/2}$ is convex rather than concave as in the gravity water waves 
problem.   Hence there are non trivial (time) quadratic resonances. This was not the case for the gravity water waves.

\paragraph{\it 3) Weighted energy estimates}  In this work, we chose to estimate $u$ in weighted energy spaces given by the invariances of the equation, rather than estimating, 
say $xf$ in $L^2$. This gives better control, since the invariant vector fields have better commuting properties than commuting  $x$ in the equation of the profile. The presence 
of quadratic time resonances in this problem also makes bounding $xf$ in $L^2$ a delicate question.

Through the use of vector fields, we actually  integrate the vector fields method and the space time resonance method to prove our result.   The steps in our proof are 
reminiscent of the proof of global existence in three dimensions for nonlinear wave equations satisfying the null condition by 
Klainerman~\cite{Klainerman2}, see also Sogge~\cite{Sogge}. In these works, the null condition is utilized via a {\it{pointwise}} bound for the quadratic term of the type 
(we denote $Q$ for the quadratic term, and $\Gamma$ for the vector fields) $|Q(u,u)| \leq \frac{C}{t} |\Gamma u| |\Gamma u|$. This inequality in turn is deduced from 
an algebraic identity.   For our problem such an algebraic identity  does not exists for quadratic terms. However by analyzing the space time resonant frequencies and the 
null structure, all cancellations can be taken into account and this makes it possible for us to close our argument.  Another difference is that here we need to use  a cascade of energy estimates 
with controlled growth bounds. 
 
\section{Sketch of the proof}

\subsection{Ideas of the proof}
\subsubsection*{1) \it Weighted Energy estimates} 
The basic estimate is the conservation of energy 
$$ 
\int_\SD |v|^2 \,dxdz +  \int_{\mathbb{R}^2} (\sqrt{1 + |\p h|^2} -1) \,dx 
   = \int_{\mathbb{R}^2} \psi G(h) \psi \,dx +
  \int_{\mathbb{R}^2} (\sqrt{1 + |\p h|^2} -1) \,dx
$$ 
which is a consequence of invariance under time translation. 
Higher energy estimates are derived by using the geometric structure of the problem, as was done in \cite{SZ06}.  This essentially  propagates the regularity of  $u$ at a rate of $3/2$ spatial derivatives at a time, which can be seen from the model linear problem  $ \p_tu =- \Lambda^{\frac 32} u$. To simplify our presentation we elect to propagate the spatial regularity at twice that rate, i.e., $3$.

Since weighted energy estimates are derived from invariances of the equation \cite{Klainerman}, we note that ransformations that keep \eqref{ww} invariant are:
\begin{inparaenum}[\itshape a\upshape)]
\item {\it space translations}, $(h(t,x),\psi(t,x))\mapsto (h(t,x+\delta),\psi(t,x+\delta))$ for  $\delta\in\RR^2$;
\item {\it space rotations}, $(h(t,x),\psi(t,x))\mapsto (h(t,R_\theta x),\psi(t,R_\theta x))$ for any $\theta$, where $R_\theta$ is the rotation of angle $\theta$ around the origin;  and
\item {\it scaling},  $(h(t,x),\psi (t,x)) \mapsto \left(\frac{1}{\lambda} h (\lambda^{3/2} t,\lambda x),\frac{1}{\lambda^{1/2}}\psi(\lambda^{3/2}t,\lambda x) \right)$ for any 
$\lambda >0$.
\end{inparaenum} The vector fields associated with  these transformations  are given by $\p$,  $\Omega$, and  $\mathcal{S}$.

Since these vector fields are  derived  from the symmetries of the equation, they have good commutation properties. Combining them with the physical energy, this naturally leads to the weighted Sobolev spaces defined above. We show that the norm with the highest number of derivatives or vector spaces grows slowly in time (like a small power of $t$):
$\| \Lambda^{1/2}  u_{0} \|_{W^{9/2,2}_{2K} ( \RR^2  )} 
+ \|  \Lambda^{\alpha_*}   u_0 \|_{W^{0,2}_{2K}( \RR^2  ) } \lesssim t^{\delta'}$.

These energy estimates are performed in Section~\ref{sectionenergy}.

\subsubsection*{\it 2)  Dispersive bound} The weighted energy estimate relies crucially on the decay of $\Lambda^{1/2+\alpha}u$ in $W^{5,\infty}_K$ at the rate $1/t$. 
It is obtained from the  linear estimate
$$
\left\| \Lambda^{1/2} e^{-it\Lambda^{3/2}} f \right\|_{L^\infty} \lesssim \frac 1t  \sum_{|k|\leq 3} \left\| Y(\p) \Sigma \Omega^k f \right\|_{L^2},
$$
where $Y(\p) = |\p|^{\iota} + |\p|^{-\iota}$ for arbitrary small $\iota$, derived in  Appendix~\ref{sectionlindecay}.  

\subsubsection*{3) Space-time  resonances}  For the above linear estimate to yield the desired decay of $\Lambda^{1/2+\alpha} u$ in $W^{5,\infty}_K$ at the rate $1/t$, 
it is easily seen that a uniform in time bound on $\Lambda^\alpha u$ in $W^{7,2}_{4+K}$ suffices. This is obtained  by the method of space-time resonances; 
the key analytic property of $(WW)$ being the vanishing of the quadratic interactions at the space time resonant waves. We briefly explain this below.
\label{QR}

The space-time resonances method identifies the wave interactions which lead to significant contributions to the long time dynamics of solutions. It also presents 
an algorithmic method of how to deal with these interactions. It was introduced by the authors
and proved very efficient in dealing with the global existence problem for a variety of dispersive equations: nonlinear Schr\"odinger~\cite{GMS1}~\cite{GMS2}, nonlinear
Klein-Gordon~\cite{PG2}, gravity water waves~\cite{GMS3}, Euler-Maxwell~\cite{GM}.
 We  refer the reader  to~\cite{PG1} for a  comprehensive  presentation of the method. 
 
For \eqref{ww}  the most significant interaction comes from the quadratic terms which 
can be computed explicitly by using  the expansion for $G(h)$ contained in Sulem and Sulem~\cite{SuSu}
$$
\left\{ \begin{array}{l} \partial_t h = \Lambda \psi - \nabla \cdot(h \nabla \psi) - \Lambda (h \Lambda \psi) \mbox{ + cubic and higher order terms}\\
\partial_t \psi = \Delta h - \frac{1}{2} |\nabla \psi|^2 + \frac{1}{2} |\Lambda \psi|^2  \mbox{ + cubic and higher order terms}.  \end{array}\right.
$$
By writing the solution  $u = \Lambda^{1/2} h + i \psi,
$  in terms of its profile $u(t) = e^{-it\Lambda^{3/2}} f(t)$, and using Fourier transform, one quickly discovers that the worst quadratic interaction, as far as space-time 
resonances are concerned, is  $u$ with $\bar u$.  The space time resonant frequencies in this case are given by $ \mathcal{R}_{+-} = \{(\xi,\eta) ; \quad \xi= 0 \}$. 
The vanishing of these quadratic interactions seems needed, which is the case for \eqref{ww}.   These calculations are carried out in Section~\ref{sectionnonlindecay}.

\subsection{Detailed plan of the proof}

\label{tech}

Since local well-posedness is not an issue, the whole proof consists of a global a priori estimate. Define 
\begin{enumerate}
\item The energy norm
$$
\|u\|_{\operatorname{energy}} \overset{def}{=} \sup_t \sum_{j=0}^{2K} \<t\>^{-(j+1) \delta} \left\| \Lambda^{1/2} u \right\|_{W^{3(2K-j)+9/2,2}_j}
$$  
\item The decay norm
\begin{equation*}
\begin{split}
\| u\|_{\operatorname{decay}} \overset{def}{=} & \sup_t \left[ \<t\>^{-\delta'} \left\| Y(\partial)^3 \Lambda^\alpha u \right\|_{W^{9,2}_{8+K}} 
+ \left\|Y(\partial)^2 \Lambda^\alpha u \right\|_{W^{7,2}_{4+K}} \right] \\
& + \sup_t \sup_{0 \leq \beta \leq 1/2} \left[  \<t\>^{1-\delta'-\frac{2}{3} \beta} \left\| Y(\partial)^3 \Lambda^{\frac{1}{2} + \alpha - \beta} u \right\|_{W^{7,\infty}_{4+K}}
+ \<t\>^{1-\frac{2}{3} \beta} \left\| Y(\partial) \Lambda^{\frac{1}{2} + \alpha - \beta} u \right\|_{W^{5,\infty}_K} \right].
\end{split}
\end{equation*}
\item The global norm
$$
\|u\|_{X} \overset{def}{=} \|u\|_{\operatorname{energy}} + \| u\|_{\operatorname{decay}}
$$
\end{enumerate}

We can now define all the parameters appearing in the proof:
\begin{itemize}
\item $K \geq 10$ controls the number of derivatives (see~(\ref{conditiondata})).
\item $\iota << \alpha,\epsilon $: auxiliary parameter appearing in $Y(\xi)$, used when summing dyadic pieces.
\item $\alpha_* <<1$: (fractional) number of derivatives of $u_0$ in $L^2$ for small frequencies (see~(\ref{conditiondata})).
\item $\alpha = \alpha_* + 3 \iota<<1$: used in the definition of $\|u\|_{\operatorname{energy}}$ and $\|u\|_{\operatorname{decay}}$.
\item $\epsilon << 1$: controls the size of the initial data in $X$ (see~(\ref{conditiondata})).
\item $R >> 1$: large parameter which we do not set yet. It will be such that the solution satisfies $\|u\|_X < R \epsilon$.
\item $\delta = \sup (C_0 R \epsilon, \alpha)<<1$, for a big enough constant $C_0$: time growth rate, appears in the definition of $\|u\|_{\operatorname{energy}}$.
\item $\delta' = (2K+1) \delta<<1$: time growth rate, appears in the definition of $\|u\|_{\operatorname{decay}}$.
\end{itemize}

The steps in proving the global a priori estimate are as follows:

\subsubsection*{Control of the energy norm.} This is obtained in Section~\ref{sectionenergy}. There it is proved (Proposition~\ref{prop:hs} and Proposition~\ref{prop:energy})
that if the data is such that \eqref{conditiondata} holds 
and $\|u\|_X < R \epsilon$, then 
$$
\left\| u\right\|_{\operatorname{energy}} \leq C_1' \epsilon,
$$
for some constant $C_1'$ uniform in $R$ as long as $R\epsilon <<1$.

\subsubsection*{Control of the decay norm.} This is obtained in Section~\ref{sectionnonlindecay}. 
There it is proved (Proposition~\ref{eagle}) that if the data is such that \eqref{conditiondata} holds and $\|u\|_X < R \epsilon$, then
$$
\left\| u\right\|_{\operatorname{decay}} \leq C_1'' \left( \epsilon + \|u\|_X^2 \right),
$$
for some constant $C_1''$ uniform in $R$ as long as $R\epsilon <<1$.

\subsubsection*{Control of the $\|\, \|_X$ norm.} From the two previous points, we deduce that if
the data is such that \eqref{conditiondata} holds 
and $\|u\|_X < R \epsilon $, 
then
$$
\left\| u\right\|_X  \leq C_1  \left( \epsilon + \|u\|_X^2 \right),
$$
for some constant $C_1 \geq 1$ uniformly  in $R$ as long as $R\epsilon <<1$. For $\epsilon$ sufficiently small, the inequalities
$$
x \leq C_1 (\epsilon + x^2) \quad \mbox{and} \quad x \geq 0
$$
hold on two connected components, $[0,x_1]$ and $[x_2,\infty]$ with $x_1 < x_2$ and $x_1 \sim C_2 \epsilon$, for a constant $C_2$.
 Picking $R>C_2$, we get the desired estimate $\|u\|_X < R \epsilon$ by a continuity argument.

\subsection{Notations}

We adopt the following standard notations
\begin{itemize}
\item $A \lesssim B$ if $A \leq C B$ for some implicit constant $C$. The value of $C$ may change from line to line.
\item $A \sim B$ means that both $A \lesssim B$ and $B \lesssim A$.
\item If $f$ is a function over $\mathbb{R}^d$ then its Fourier transform, denoted $\widehat{f}$, or $\mathcal{F}(f)$, is given by
\footnote{In the text, we systematically drop the constants such as $\frac{1}{(2 \pi)^{d/2}}$ since they are irrelevant for our purposes.}
$$
\widehat{f}(\xi) = \mathcal{F}f (\xi) = \frac{1}{(2\pi)^{d/2}} \int e^{-ix\xi} f(x) \,dx \;\;\;\;\mbox{thus} \;\;\;\;f(x) = \frac{1}{(2\pi)^{d/2}} \int e^{ix\xi} \widehat{f}(\xi) \,d\xi.
$$
\item The Fourier multiplier with symbol $m(\xi)$ is defined by
$$
m(\frac 1i \p)f = \mathcal{F}^{-1} \left[m \mathcal{F} f \right].
$$
\item In particular, we denote ($\iota$ being a sufficiently small real number) 
$$
\Lambda \overset{def}{=} |\partial| \qquad \mbox{and} \qquad Y(\partial) \overset{def}{=} \Lambda^\iota + \Lambda^{-\iota}
$$
\item The bilinear Fourier multiplier with symbol $m$ is given by
$$T_m(f,g)(x) \overset{def}{=} \mathcal{F}^{-1} \int m(\xi,\eta) \widehat{f}(\eta) \widehat{g}(\xi-\eta)\,d\eta.$$
\item The japanese bracket $\langle \cdot \rangle$ stands for $\langle x \rangle = \sqrt{1 + x^2}$.
\end{itemize} 

\section{Energy Estimates} \label{sectionenergy}  The conserved energy of system \eqref{CW} is given by
\begin{equation}\label{en0}
E_{\rm{physical}} = \int_{\SD} |v|^2 dxdz + \int_{\SCB}1- \frac 1{\sqrt{1+ |\p h|^2}} dS
\end{equation}
which is sufficient to give $\dot H^{\frac 12}(\SCB)$ on the potential $\psi_\CH$,  and $\dot H^{1}(\SCB)$ on $h$  for small $\p h$.
Higher $H^s$ bounds and weighted norms bounds will be derived by commuting vector fields with the evolution equation of the 
mean curvature $\kappa$  of $\SCB$, and from Euler's equation evaluated on the boundary  $\p\CU$ (see \cite{SZ06} for the derivation)

\begin{subequations}\label{firstk}
\begin{align}
&\D_t \kappa =- \Delta_{ \SCB} v \cdot N -
2 \Pi \cdot (\nabla^\top  v) \label{E:dtk2}
\\[.3em]
&N\cdot \D_t v = -\CN\kappa +N\cdot \nabla \Delta_0^{-1}{\rm div} (v\cdot\nabla v)\label{seuler}\\[.3em]
&\D_t(z-h)= v_3 -D_t h =0,\label{boundary}
\end{align}
\end{subequations}
where $\D_t = \p_t +  v \cdot  \nabla \in T_{(x,t)}\p \CU$ is the material derivative on $\p \CU$, $ \Delta_{\SCB} $ is the surface Laplacian, $\Pi$ is the second fundamental form of $\SCB$, $\CN$ is the Dirichlet Neumann operator of $\SD$,  $\top$ and $\perp$ are the tangential and normal component to $\SCB$ of the relevant quantity, $ \Delta_0^{-1}$ is the inverse Laplacian on $\SD$ with $0$ Dirichlet data. Here we also  introduced the notation $A_1 \cdot A_2 =\text{trace} (A_1 (A_2^*))$, for two matrices, where  $A^*$ is the adjoint  of $A$.

In order to compute commutators of vector fields with the above system we need to introduce some geometric notation.

\subsection{Geometric notations}
Recall that  $\CU=\{(t,x,z); \;\; (x,z)\in\SD\}$ denotes the space time fluid domain. Let
 \[
\tilde\p_t = \p_t  +(\p_th)\p_z, \qquad  \tilde\p_i= \p_i + (\p_ih)\p_z, \quad i =1,2,
\]
denote vector fields defined on $\CU$ that  are tangent to $\p\CU$.  For any function $\varphi$ defined on $\CU$ let $\varphi_b$ denote its value on $\CU$, i.e., 
$\varphi_b(t,x) = \varphi(t,x,h(t,x))$, consequently
\[
\tilde\p_\alpha\varphi(t,x,z)|_{z=h(t,x)}= \p_\alpha(\varphi_b(t,x)). 
\]
Also note that since $\tilde\p_i\in T\SCB$ are linearly independent, they can be used to construct an orthonormal frame on  $T\SCB$, which we denote by $(e_1,e_2)$.

For any vector $e\in T\SCB$  the covariant derivative $\CD_e$ on $T\SCB$ can be defined in terms of the outward  unit normal $N\perp T\SCB$ as follows. 
Writing  $\nabla_e=e\cdot\nabla$ for the directional derivative in $\RR^3$,  we have
\begin{equation*}
\CD_e w=( \nabla_e w)^\top = \nabla_e w - (\nabla_e w)^\perp  = \nabla_e w +( w\cdot \nabla_eN) N, \qquad w \in T\SCB.
\end{equation*}
For a vector-valued function on $\RR^3$ $f = (f^j)$, we also denote $Df$ for the partial derivatives matrix
$$
Df = (\partial_i f^j)_{i,j}.
$$
For any orthonormal frame $\{e_a\}$ of  $T\SCB$,  the second fundamental form $\Pi$,  the mean curvature $\kappa$ of $\SCB$, the Hessian $\CD^2$
and the surface Laplacian $\Delta_\SCB$ are given by
\begin{align*}
&\Pi(e_a) = \nabla_{e_a}N, \qquad  \kappa = \nabla_{e_a}N\cdot e_a, \qquad  \CD \phi(X,Y) = \nabla_X \nabla_Y \phi - \nabla_{\CD_X Y} \phi, \\
&\Delta_\SCB \varphi =   \CD^2\varphi(e_a,e_a) =   \nabla_{e_a}\nabla_{e_a} \varphi- \nabla_{\CD_{e_a}e_a}\varphi=    \nabla_{e_a}\nabla_{e_a} \varphi- \nabla_{\nabla_{e_a}e_a}\varphi -\kappa\nabla_N\varphi,
\end{align*}
where we sum over repeated indices.  Note that using the frame $\{e_a\}$ the equation for $\kappa$ can be rewritten as
\begin{equation}\label{dtkframe}
\D_t \kappa =- \Delta_{ \SCB} v \cdot N -2  (\nabla_{e_a}N) \cdot (\nabla_{e_a}  v).
\end{equation}
The harmonic extension to $\mathscr{D}$ of a function $f$ defined on $\mathscr{B}$ is denoted by $f_\CH$, and the Dirichlet-Neumann operator by  $\CN$ 
\[
\CN (f) = \nabla_N f_\CH : \SCB \to \RR.
 \]
We refer to Section~\ref{appendixtaee} for the definitions and main properties of the spaces $H^s(\mathscr{B})$ and $H^s(\mathscr{D})$.

\subsection{Vector fields} \label{Vf}
Recall that we defined $\Omega$ and $\mathcal{S}$ on $\mathbb{R}^2$ by
\begin{equation}
\label{VF}
\Omega = x^1 \partial_2 - x^2 \partial_1 = \omega^i \partial_i \quad \mbox{and} \quad \mathcal{S} = \frac{3}{2} t \partial_t + x^1 \partial_1 + x^2 \partial_2 
= \frac{3}{2} t \partial_t + x^i \partial_i.
\end{equation}
For any function $\varphi: \CU \to \RR$  with   $\varphi_b =\varphi|_{\p \CU}$ we have
\[
\begin{split}
\left(\frac{3}{2} t \partial_t + x^i \partial_i\right)\varphi_b = \left. \left( \frac 32 t\p_t +x^i\p_i + z\p_z+( \CS h-h )\p_z\right)\varphi\right|_{\p \CU},\\
\left( x^1 \partial_2 - x^2 \partial_1\right)\varphi_b = \left. \left(  x^1 \partial_2 - x^2 \partial_1+(  x^1 \partial_2h - x^2 \partial_1h )\p_z\right)\varphi\right|_{\p \CU}.
\end{split}
\]
Thus with a slight abuse of notation, $(\CS,\Omega)$, can be considered as vector fields defined on $\p\CU$ by writing
\begin{align*}
&\TCS  =  \frac 32 t\p_t + x^i\p_i + z\p_z +Z\p_z, \qquad Z= (\frac{3}{2} t \partial_t + x^i \partial_i) h-h,\\
&\TOM = x^1\p_2 -x^2\p_1 + ( x^1 \partial_2 - x^2 \partial_1)h \partial_z.
\end{align*}
The coordinates of the space parts of $\mathcal{S}$ and $\Omega$ are denoted respectively $(S^i)$ and
$(\omega^i)$:
$$
\left(\begin{array}{l} S^1 \\ S^2 \\ S^3 \end{array} \right) = \left(\begin{array}{c} x^1 \\ x^2 \\ z + \mathcal{S} h - h \end{array} \right)
\quad \mbox{and} \quad \left(\begin{array}{l} \omega^1 \\ \omega^2 \\ \omega^3 \end{array} \right) 
= \left(\begin{array}{c} -x^2 \\ x^1 \\ \Omega h \end{array} \right).
$$
These vector fields can also be extended harmonically on  $\mathscr{D}$   $(\TCS_\CH, \TOM_\CH)$,  and thus on $\CU$:
 \begin{align*}
& \TCS_\CH =    \frac 32 t\p_t + x^i\p_i + z\p_z   +Z_\CH\p_z, \qquad \TOM_\CH =   x^1 \partial_2 - x^2 \partial_1 +   \omega^3_\CH\p_z.
 \end{align*}
 
\subsection{Commutators}  Since on $\p\CU$,   $\TCS, \TOM,   \D_t\in T\p\CU$, then $[\TCS, \D_t ],  [\TOM, \D_t ] , [\TOM, \TCS] \in   T\p\CU$, and
\begin{subequations}\label{commute}
\begin{align}
&[\TCS, \D_t ] \varphi   + \frac 32 \D_t \varphi  = (\frac12v^i+\CS v^i)\tilde{\p}_i\varphi  \\
&[\TOM, \D_t ] \varphi   = \Omega v^i \tilde{\p}_i\varphi -v^1  \tilde{\p}_2 \varphi + v^2  \tilde{\p}_1 \varphi.
\end{align}
\end{subequations}
\begin{remark}  One can think of these commutators in the following way. \\

\noindent $\bullet$ $[\TCS, \D_t ] + \frac 32 \D_t$  is a first order differential operator with coefficients depending on $( v,\CS v, \p h v, \p h \CS v)$.\\

\noindent $\bullet$   $[\TOM, \D_t ]$ is a first order differential operator with coefficients depending on $( v,\Omega v, \p h v, \p h\Omega v)$.\\
\end{remark}

To compute commutators of $(\D_t,  \TCS, \TOM)$ with the surface Laplacian $\Delta_\SCB$, we proceed as follows.  
Writing  $\CX = \alpha(t)\p_t + X\cdot\nabla \in T\p\CU$  for any of these vector fields, we first note that for  any vector field $e\in T\SCB$ we have $[ \CX,e]  \in T\SCB$
since $[\CX,e] \in T\p\CU$ and has no $\p_t$ term. Thus
\begin{equation*}
[ \CX,e] = \CX e - \nabla_e X =  (\CX e)^\top -( \nabla_{e }X)^\top
\end{equation*}
We also note that  for  any orthonormal frame $\{e_a\}$ of $T\SCB$ that depends smoothly on $t$, we have 
\begin{equation} \label{antisym}
(\CX e_a)^\top\cdot e_b  +  e_a  \cdot (\CX e_b)^\top = \CX(e_a\cdot e_b) =0.
\end{equation}
With these observations  we compute
\begin{equation*}
\begin{split}
\CX\Delta_\SCB \varphi  &= \CX  \left(\nabla_{e_a}\nabla_{e_a} \varphi- \nabla_{\CD_{e_a}e_a}\varphi\right)\\
&= \Delta_\SCB \CX\varphi  +   \nabla_{ [\CX,  e_a]}\nabla_{e_a}\varphi +  \nabla_{e_a}\nabla_{ [\CX,  e_a]}\varphi - \nabla_{[\CX, \CD_{e_a}e_a]}\varphi\\
&= \Delta_\SCB \CX\varphi  + 2\CD^2\varphi(e_a, [\CX,  e_a]) +  \nabla_{\W}\varphi\\
&= \Delta_\SCB \CX\varphi  - 2\CD^2\varphi(e_a, (\nabla_{e_a}X)^\top) +  \nabla_{\W}\varphi, \qquad \mbox{ by \eqref{antisym}} 
\end{split}
\end{equation*}
where $\W = \CD_{[ \CX,  e_a]}e_a + \CD_{e_a}[ \CX,  e_a]  -[ \CX, \CD_{e_a}e_a]\in T\SCB$. By keeping track of  the tangential terms we get
\begin{equation*}
\begin{split}
\W &=\left( \nabla_{[ \CX,  e_a]}e_a + \nabla_{e_a}[ \CX,  e_a]  - \CX(\nabla_{e_a}e_a -\kappa N)
+ \nabla_{ \CD_{e_a}e_a} X \right)^\top\\
&=  \left( \nabla_{[ \CX,  e_a]}e_a  + \nabla_{e_a}\CX e_a- \nabla_{e_a}\nabla_{e_a}  X- \CX(\nabla_{e_a}e_a) + \CX (\kappa N) + \nabla_{\CD_{e_a}e_a} X\right)^\top\\
&= - \left( \Delta_\SCB X \right)^\top + \kappa\CX N.
\end{split}
\end{equation*}
Thus
\begin{equation}\label{comxlap}
[\CX , \Delta_\SCB]\varphi =  -2\CD^2\varphi\cdot \nabla^\top X - \left( \Delta_\SCB X \right)\cdot (\nabla^\top \varphi) + \kappa\CX N\cdot (\nabla^\top \varphi),
\end{equation}
and the commutators of $\D_t$, $\TCS$, and $\TOM$ with $\Delta_\SCB$ are given by differential operators
\begin{subequations}\label{comm-lap}
\begin{align}
&[\D_t , \Delta_\SCB]\varphi =-2\CD^2\varphi\cdot \nabla^\top v - \left( \Delta_\SCB v \right)\cdot (\nabla^\top \varphi) 
+\kappa\D_t N\cdot (\nabla^\top \varphi)  \label{dtlap}\\
&[\TCS , \Delta_\SCB]\varphi = -2 \Delta_\SCB\varphi  -2\CD^2\varphi\cdot \nabla^\top (Z\mathbf{k}) - 
\left( \Delta_\SCB S \right)\cdot (\nabla^\top \varphi) + \kappa\CS N\cdot (\nabla^\top \varphi)\\
&[\TOM , \Delta_\SCB]\varphi =  -2\CD^2\varphi\cdot \nabla^\top \omega  - \left( \Delta_\SCB \omega \right)\cdot (\nabla^\top \varphi) + \kappa\TOM  N\cdot (\nabla^\top \varphi)
\end{align}
\end{subequations}

\begin{remark} \label{redtlap} Using the relation  $D_th= v_3$ we get\\

\noindent $\bullet$ $[\D_t , \Delta_\SCB]$  is a second order operator with coefficients depending on 
$w_1\overset{def}{=}( \p^2h,(\p h)^2, Dv, D^2v,$

$\qquad \qquad\qquad\qquad\qquad\qquad\qquad\qquad\qquad\qquad\qquad\qquad\qquad\qquad\qquad\qquad\qquad\p h Dv, \p h D^2v)$.\\

\noindent $\bullet$   $[\TCS , \Delta_\SCB] + 2 \Delta_\SCB$ is a second order operator with coefficients depending on $w_2\overset{def}{=}( \p^2 h, (\p h)^2 , \CS  \p h, \CS \p^2h)$.\\

\noindent $\bullet$   $[\TOM , \Delta_\SCB]$ is a second order operator with coefficients depending on $w_3\overset{def}{=}( \p^2 h, (\p h)^2, \TOM  \p h, \TOM \p^2h)$.
\end{remark}

To compute commutators of $(\D_t,  \TCS, \TOM)$ with  the Dirichlet--Neumann operator $\CN$ we proceed as follows.  Again writing 
 $\CX = \alpha(t)\p_t + X\cdot\nabla \in T\p\CU$,  for either $\D_t$, $\TCS$, or $\TOM$,  we have for any  $g : \CU \to\RR$ and $\varphi : \p\CU \to\RR$
\begin{subequations}\label{precmd}
\begin{align}
&  \CX_\CH \nabla g = \nabla   \CX_\CH g  - (DX_\CH)^* \nabla g\label{grad} \\
&  \CX_\CH  \varphi_\CH = (  \CX \varphi)_\CH
+ \Delta^{-1}_0 2 {\rm div}( (DX_\CH)^* \nabla \varphi_\CH) \\
&  \CX_\CH \Delta_0^{-1}g =\Delta_0^{-1}(  \CX_\CH  g) + \Delta^{-1}_0 2 {\rm div}((DX_\CH)^*  \nabla\Delta_0^{-1}g )\label{lap-1}
\end{align}
\end{subequations}
This gives the identities
$$
[ \CX_\CH, \nabla \Delta_0^{-1} {\rm div}] g = (DX_\CH)^* \nabla \Delta_0^{-1}{\rm div}  g 
+ 2 \nabla \Delta_0^{-1} {\rm div} \left( (DX_\CH)^* \nabla \Delta_0^{-1} {\rm div} g \right) + \nabla \Delta_0^{-1} \left( D X_{\CH} \cdot D g \right)
$$
and
\begin{equation}\label{ddn}
[ \CX, \CN]\varphi =  (\CX N)\cdot \nabla \varphi_\CH   - N\cdot ((DX_\CH)^* \nabla \varphi_\CH)
+ N\cdot \nabla  \Delta_0^{-1} 2{\rm div}((DX_\CH)^*  \nabla \varphi_\CH).
\end{equation}
In particular
\begin{subequations}\label{comm-n}
 \begin{align}
&[\D_t, \CN]\varphi =\D_tN\cdot \nabla \varphi_\CH   - N\cdot ((Dv)^*\nabla \varphi_\CH)
+ N\cdot \nabla  \Delta_0^{-1} 2{\rm div}((Dv)^*  \nabla \varphi_\CH),  \label{dtn}\\[.3em]
\begin{split}\label{sdn}
&[\TCS, \CN]\varphi + \CN \varphi = (\CS N) \cdot  \nabla \varphi_\CH  - N\cdot( \nabla Z_\CH) \p_z\varphi_\CH
 + N\cdot \nabla  \Delta_0^{-1} 2{\rm div}(( \nabla Z_\CH) \p_z\varphi_\CH)
\end{split} \\[.3em]
\begin{split}\label{rdn}
&[\TOM, \CN] \varphi=  (\TOM N) \cdot  \nabla \varphi_\CH   -  N^1\p_2\varphi_\CH + N^2\p_1\varphi_\CH -   N\cdot (\nabla \omega^3_\CH)\p_z \varphi_\CH\\
&\phantom{[\TOM, \CN] \varphi=} + N\cdot \nabla  \Delta_0^{-1} 2{\rm div}((\nabla \omega^3_\CH) \p_z \varphi_\CH);
\end{split}
\end{align}
\end{subequations}
The only term that prevents the order of these commutators to be apparent is $N\cdot \nabla  \Delta_0^{-1} 2{\rm div}$.   However by proposition \eqref{elliptic}
\[
N\cdot \nabla  \Delta_0^{-1} {\rm div} :\; H^{s+\frac 12}(\SD)\to H^s(\SCB)
\]
 is a bounded operator.    Thus if $\SE: H^s(\SCB)  \to  H^{s+\frac 12}(\SD)$, is an extension operator, then
\[ N\cdot \nabla  \Delta_0^{-1} {\rm div}\SE :\; H^{s}(\SCB)\to H^s(\SCB),
\]
is a bounded operator.
\begin{remark}   \label{renlap}  Using the relation $D_th= v_3$  we get\\

\noindent $\bullet$ $[\D_t, \CN]$  is a first order operator with coefficients depending on $\tilde w_1 \overset{def}{=} (Dv, \p h Dv)$.\\

\noindent $\bullet$   $[\TCS, \CN] +\CN$ is a first order operator with coefficients depending on $\tilde w_2 \overset{def}{=}( \p h, \CS  \p h)$.\\

\noindent $\bullet$   $[\TOM, \CN] $ is a first order operator with coefficients depending on $\tilde w_3\overset{def}{=}( \p h, \TOM  \p h)$.
\end{remark}
Bounds on commutators are given in the following lemmata.
\begin{lemma}\label{com1-est} Assume $(v,\p h)$ are smooth and $\|(v,\p h)\|_{W^{\frac s2,\8}}\lesssim 1$, for $s >2$, 
then 
\begin{subequations}\label{estd}
\begin{align}
&\begin{cases}\label{estd1}
\|[\D_t,\CD]\varphi\|_{L^2} \lesssim  \|(Dv, \p h Dv)\|_{L^{\infty}}\|\nabla \varphi\|_{L^2}\\
\|[\D_t,\CD]\varphi\|_{H^s} \lesssim  \|(Dv, \p h Dv)\|_{W^{s/2 ,\infty}}\|\nabla \varphi\|_{H^s}  +   
 \| \nabla \varphi \|_{W^{s/2 ,\infty}}\| (Dv, \p h Dv)  \|_{H^s} \\
\end{cases}\\
&\begin{cases}\label{estd2}
\|[D_{t},\CN] \varphi\|_{L^2} \lesssim \|\tilde w_1\|_{L^\infty} \|\nabla \varphi\|_{L^2},\\
\|[D_{t},\CN] \varphi\|_{H^s} \lesssim \|\tilde w_1\|_{W^{[s/2],\infty}} \|\nabla \varphi\|_{H^{s}} +\| \tilde w_1\|_{H^{s}} \, \|\nabla \varphi\|_{W^{[s/2],\infty}}.
\end{cases}\\
&\begin{cases}\label{estd3}
\|[\CD,\CN] \varphi\|_{L^2} \lesssim \|\p^2h\|_{L^\infty} \|\nabla \varphi\|_{L^2},\\
\|[\CD,\CN] \varphi\|_{H^s} \lesssim \|\p^2h\|_{W^{[s/2],\infty}} \|\nabla \varphi\|_{H^{s}} +\| \p^2h\|_{H^{s}} \, \|\nabla \varphi\|_{W^{[s/2],\infty}}.
\end{cases}\\
&| \langle[D_t, \CN] \varphi,   \varphi \rangle  |  \lesssim  \left[ \|Dv\|_\infty + \|Dh\|_{W^{1,\infty}} \|Dv\|_{W^{1,\infty}}
+ \|D^2 v\|_\infty \right] \|\varphi\|^2_{H^{1/2}}\label{estd4}
\end{align}
\end{subequations}
\end{lemma}
\proof
$[\D_t,\CD]$ is a first order differential operator whose coefficients depend on  $(Dv, \partial h Dv)$, 
thus  \eqref{estd1}  is straightforward. 

To compute the norm of $[D_{t},\CN] $, given by  \eqref{dtn}, on $H^s$ we have  to bound terms of the form 
$G \cdot \nabla^\top$, $g\CN \varphi$, $|N \cdot \nabla \Delta^{-1}_{0} { {\rm div}}(A \nabla \cdot)$,   where  $G$, $g$, and $A$ are given in terms of  $\tilde w_1= (Dv, \p hDv)$.  
By  proposition  \ref{elliptic}  these terms can be bounded by 
\[
\begin{split}
&\|G \cdot \nabla^\top \varphi \|_{L^2}  + \|g\CN\varphi \|_{L^2} \lesssim 
 \| (G, G \partial h, g, g\p h)  \|_{L^{\infty}}  \| \nabla \varphi \|_{L^2} \\
&\|G \cdot \nabla^\top \varphi \|_{H^{s}}  + \|g\CN\varphi \|_{H^{s}} \lesssim 
  \| (G, G \partial h, g, g \partial h)  \|_{W^{{[\frac{s}{2}}],\infty}}  \| \nabla \varphi \|_{H^{s}} 
+ \| (G, G \partial h, g, g\p h)\|_{H^s} \| \nabla \varphi \|_{W^{{[\frac{s}{2}}],\infty}} \\
&\|N \cdot \nabla \Delta^{-1}_{0} { {\rm div}}(A \nabla \varphi ) \|_{L^2}
\lesssim    \|A\|_{L^\infty}\| \nabla \varphi \|_{L^2}\\ 
&\|N \cdot \nabla \Delta^{-1}_{0} { {\rm div}}(A \nabla \varphi ) \|_{H^{s}}
\lesssim    (1 + \| \partial h
\|_{W^{{[\frac{s}{2}} -1],\infty}}) \|A \nabla \varphi \|_{H^{s}} + \| \partial h\|_{H^{s-1}}\| A\nabla \varphi \|_{W^{[\frac{s}{2}],\infty}},
\end{split}
\]
and this implies \eqref{estd2}. Equation \eqref{estd3} follows from a similar computation.

Finally to derive \eqref{estd4}  we note that  $ \langle [D_t,\CN]\varphi,   \varphi\rangle$ can be written as a sum of three terms which can be estimated as follows:
\begin{align*}
&\bigg|\int\limits_\SCB\varphi G \cdot \nabla^\top \varphi\, dS \bigg|=\bigg|\int\limits_\SCB \frac G2 \cdot \nabla^\top \varphi^2\, dS \bigg|\lesssim
(\|\nabla^\top G\|_{L^\infty}  + \| G\|_{L^\infty}  \|\partial^2 h \|_{L^\infty} )|\varphi|^{2}_{L^2},\\
&\bigg|\int\limits_\SCB g \varphi \, \CN \varphi \, dS \bigg| =
\bigg| \frac{1}{2} \int_{\SCB} \varphi^2 \mathcal{N}g + \int_{SCD} g_{\CH} \nabla \varphi_{\CH} \nabla \varphi_{\CH} \bigg| \\
& \phantom{\bigg|\int\limits_\SCB g \varphi \, \CN \varphi \, dS \bigg|} 
\lesssim \left[ \|Dv\|_\infty + \|Dh\|_{W^{1,\infty}} \|Dv\|_{W^{1,\infty}} + \|D^2 v\|_\infty \right] |\varphi|^2_{H^{1/2}(\SCB)}
\end{align*}
where the last inequality sign follows from Proposition~\ref{prop:elliptic}, and finally
\begin{align*}
&\bigg|\int\limits_\SCB \varphi \, N \cdot \nabla \, \Delta^{-1}_{0} { {\rm div}}(D v
\nabla \varphi_{\CH}) \, dS \bigg| =  \bigg| \int\limits_\SD \varphi_{\CH}  { {\rm div}}(D v
\nabla \varphi_{\CH}) \, dx \bigg|  \\
&\phantom{ \bigg|\int\limits_\SCB \varphi \, N \cdot \nabla \, \Delta^{-1}_{0} { {\rm div}}(D v
\nabla \varphi_{\CH}) \, dS \bigg|} =  \bigg|\int\limits_\SCB N \cdot (D v \nabla \varphi_\CH)
\varphi_\CH dS \bigg| +\bigg| \int\limits_\SD \nabla \varphi_\CH \cdot D v \nabla \varphi_\CH dx\bigg|\\
&\phantom{  \bigg|\int\limits_\SCB \varphi \, N \cdot \nabla \, \Delta^{-1}_{0} { {\rm div}}(D v
\nabla \varphi_{\CH}) \, dS \bigg|}  \lesssim  (\|D v\|_{W^{1,\infty}} + \|\partial^2 h\|_{L^\infty}  \|Dv\|_{L^\infty} )\|\varphi\|^2_{H^{1/2}},
\end{align*}
which implies the stated inequality. 
\endproof

Using these commutators we can derive  a second order evolution equation for $\kappa$.
We start  by  choosing an orthonormal frame $\{e_a\}$ which is parallelly  transported along  the particle path i.e., 
\[
\D_t e_a = (\nabla_{e_a}v\cdot N)N.
\]
Applying $\D_t$ to  \eqref{dtkframe}, we obtain (see equation~(3.15) in~\cite{SZ06})
\[
\D_t^2 \kappa =- \D_t \Delta_{\SCB} v\cdot N  -  \Delta_{\SCB} v\cdot \D_t N  -2   (\D_t\nabla_{e_a}N) \cdot (\nabla_{e_a}  v)  -2   (\nabla_{e_a}N) \cdot (\D_t\nabla_{e_a}  v).
\]
Using the fact that 
\[
\D_t N = - [(D v)^*N]^\top, \mbox{ and }  [\D_t, \nabla_{e_a}] = -  \nabla_{(\nabla_{e_a}v)^\top},
\]
and the commutator formula $[\D_t, \Delta_{\SCB}]$, we obtain (see equation~(3.16) in~\cite{SZ06})
\begin{equation*} \begin{split}
\D_t^2 \kappa = &- \Delta_{\SCB} \D_t v\cdot N  - 2  \nabla_{e_a}N\cdot \nabla_{e_a} \D_t v  -2 \Delta_{\SCB} v \cdot \D_tN \\
&-\left( 4\CD_{e_a} \lf(\D_tN\rt) + 2 \nabla_{(\nabla_{e_a}v)^\top}N)\right) \cdot \nabla_{e_a} v - \kappa |\D_tN|^2 
- 2 \D_tN \cdot \nabla_{(\D_tN)}N
\end{split}
\end{equation*}
Using Euler's equation evaluated on $\SCB$ to substitute for $\D_t v$ on $\SCB$, we obtain
\begin{equation} \begin{split}
\D_t^2 \kappa =  &\Delta_{\SCB} \CN (\kappa) - \nabla \kappa_\CH \cdot
\Delta_{\SCB}N  -\Delta_{\SCB} \nabla\Delta_0^{-1} {\rm div}(v\cdot\nabla v) \cdot N \\
&- 2 \Delta_{\SCB} v \cdot \D_tN
-\left( 4\CD_{e_a} \lf(\D_tN\rt) + 2 \nabla_{(\nabla_{e_a}v)^\top}N)\right) \cdot \nabla_{e_a} v
 - \kappa |\D_tN|^2 \\
&  -2 \nabla_{e_a}N \cdot (\nabla_{e_a} \nabla \Delta_0^{-1}{\rm div}( v\cdot\nabla v))  -  2 \D_tN \cdot \nabla_{(\D_tN)}N \\
\overset{def}{=} &    \Delta_{\SCB} \CN (\kappa)  + R_0.
\label{E:dttk2} \end{split}
\end{equation}
Note that in terms of regularity on the boundary $\SCB$,  we have  from \eqref{firstk}
\[
v\sim \D_th, \qquad \Delta_\SCB v \sim \D_t\kp, \qquad N\cdot \D_tv  \sim \CN\kappa \sim D^3 h.
\]
Thus   the nonlinear  terms of $R_0$ are 
\[
\begin{split}
&\nabla \kappa_\CH \cdot \Delta_{\SCB}N =( \p^3h)^2 + cubic\\
&\Delta_{\SCB} \nabla\Delta_0^{-1} {\rm div}(v\cdot\nabla v) \cdot N = DvD^2v + cubic\\
&\Delta_{\SCB} v \cdot \D_tN =  D^2v\p_t\p h +cubic\\
&\CD_{e_a} \lf(\D_tN\rt) \cdot \nabla_{e_a} v = \p_t\p^2h Dv + cubic.
\end{split}
\]
If we introduce $w\overset{def}{=} (v,\p h)$, then we can write
\begin{equation}\label{r0approx}
R_0 \sim  \D_t\kp D_t\p h + (\p\kappa)^2 + \text{  cubic} \sim  (Dw, D^2 w)^2   +(w)^3.
\end{equation}
Energy estimates will be derived using equation \eqref{E:dttk2}.

\begin{remark} The energy and weighted energy estimates can be done simultaneously.
However we elected to separate the two for the benefit of the reader. The $H^s$ estimate is derived by commuting  $\CA =  \Delta_\SCB\CN$ 
with the equation which leads to a schematic equation 
$\p_t^2 \CA g -\p^3 \CA g = (\p \CA g)\p g$, while the weighted estimates which are derived from commuting $\CS$ and $\TOM$ lead to $\p_t^2 \CS g -\p^3 \CS g = (\p^3 g)\p \CS g$.
Thus energy estimates for $\CS g$ require estimates on $\p^3 g$. In terms of regularity, $\CS$ acts therefore as $3/2$ derivatives.
To avoid fractional derivatives in energy (and weighted energy) estimates we estimate derivatives in multiples of $3$.
\end{remark}

\subsection{${H^s(\SCB)}$  \label{energy-sec} estimates} For the remainder of this section all norms of $v$ are computed for $v|_\SCB$,  thus  $H^s$ stands for $H^s(\SCB)$, etc.
Bounds  of $\partial{h}$ and $v$ in $H^s$ will be  derived from the equation

\[
D^2_t \kappa - \Delta_\SCB \CN \kappa = R_0 \quad \text{ on }\quad \p\CU,
\]
and as such they are implicit in \cite{SZ06}.  

Energy estimates are usually obtained in a straightforward manner by
differentiating the equation and multiplying it by $D_t \kappa$. However in
this case, there will be commutators present, and bounding them
requires some care.

%

Consider smooth solutions of
\begin{equation}\label{lin-eq}
D^2_t g - \Delta_\SCB \CN g = F \quad \text{ on }\quad \p\CU.
\end{equation}
For functions defined on $\SCB$ let $\langle\; , \;  \rangle$ denote the  inner product on $L^2(\SCB)$,
and  define
\[
e_{g}(t) =  \langle \CD \CN D_tg,   \CD \CN\D_tg \rangle  +  \langle \CN\Delta_\SCB\CN g,   \Delta_\SCB\CN g \rangle = \|\mathcal{D}\CN D_t g \|^2_{L^2} + \| \Delta_\SCB\CN g \|^2_{H^{\frac12}}.
\]
\begin{lemma}\label{linear:prop}
Smooth solutions to the above equation satisfy, if $(Dv,\partial h)$ is small in $W^{3,\infty}$,
\begin{align*}  
\dot e_g \lesssim    \|(Dv,\p h Dv, \p^2 h)\|_{W^{3, \infty}} (e_g + \|D_tg\|_{L^2}^2 + \|\CD g\|^2_{\dot H^{\half}}) + \|\CD\CN F\|_{L^2}\sqrt{e_g},
\end{align*}
\end{lemma}

\proof Recall first that
$$
D_t dS = ( v^\perp \kappa + \CD \cdot v^\top ) dS.
$$
Multiply \eqref{lin-eq}  by $\CN\Delta_\SCB\CN D_tg$ and integrate over $\SCB$ to get
\[
\begin{split}
\dot  e_g(t) =& \langle ( v^\perp \kappa + \CD \cdot v^\top )\CD  \CN\D_tg,\CD  \CN\D_tg \rangle
+ 2\langle [D_t, \CD] \CN D_tg,  \CD  \CN\D_tg \rangle + 2 \langle\CD  [D_t, \CN ]D_tg,   \CD \CN\D_tg \rangle \\
&+ \langle[D_t, \CN] \Delta_\SCB\CN g,   \Delta_\SCB\CN g \rangle  
+2 \langle \CN[D_t , \Delta_\SCB\CN] g,   \Delta_\SCB\CN g \rangle  +  \langle \CD \CN F,   \CD \CN\D_tg \rangle.
\end{split}
\]
Thus to estimate $e_g$ we use the commutator bounds given in Lemma  \ref{com1-est}.
These bounds imply
\[
 |\langle [D_t, \CD] \CN D_tg,  \CD  \CN\D_tg \rangle|+  |\langle\CD  [D_t, \CN ]D_tg,   \CD \CN\D_tg\rangle | \lesssim \|(Dv,\p hDv, \p^2 h)\|_{W^{1,\infty}} 
\| D_tg\|^2_{H^2},
 \] 
%
%
and by \eqref{estd4} of the same lemma we have 
\[
| \langle[D_t, \CN] \Delta_\SCB\CN g,   \Delta_\SCB\CN g \rangle  |
+| \langle \CN[D_t , \Delta_\SCB\CN] g,   \Delta_\SCB\CN g \rangle | \lesssim  \|(D v,\p hDv,\p^2 h)\|_{W^{3,\infty}} \|\CD g\|^2_{H^{5/2}}.
\]
Consequently 
\[
\dot e_g \lesssim  \|(Dv,\p hDv, \p^2 h)\|_{W{1, \infty}} (e_g + \|D_tg\|_{L^2}^2 + \|\CD g\|^2_{\dot H^{\half}}) + \|\CD\CN F\|_{L^2}\sqrt{e_g},
\]
which implies the desired result.
\endproof 

$H^s$ estimates for solutions of  \eqref{CW} follow in a straightforward manner from Lemma \ref{linear:prop}. Let  $\kappa_n =\CA^n\kappa$, where $\CA= \Delta_\SCB\CN$, then     
\begin{equation}
\D_t^2 \kappa_n =  \Delta_{\SCB} \CN (\kappa_n)  + R_n,\label{kn}
\end{equation}
where $R_n$ is defined inductively by
\[
R_n \overset{def}{=}- [\CA, \D^2_t] \kappa_{n-1}  +\CA R_{n-1}= D_t[D_t,\CA]\kappa_{n-1} + [D_t,\CA]\D_t\kappa_{n-1}+ \CA R_{n-1}   ,
\]with $R_0$ given by \eqref{E:dttk2}.  Thus 
\begin{align}\label{Rninduc}
R_n =\sum_{i=1}^n\CA^{n-i} \left(D_t[D_t,\CA]\kappa_{i-1} + [D_t,\CA]\D_t\kappa_{i-1}\right)  + \CA^{n}R_0
\end{align}

Define the energy  
\[
\CE_n(t) \overset{def}{=} \half \langle \CD \CN \D_t\kappa_n,   \CD \CN \D_t\kappa_n \rangle  +  \half \langle \CN\CA\kappa_n,   \CA\kappa_n \rangle
\]
where $\CA= \Delta_\SCB\CN$. Finally, let
$$
E_n(t) \overset{def}{=} \sum_{k=0}^{n} \CE_k(t) + E_{\rm{physical}},
$$
where $n$ is a large integer and $E_{\rm{physical}}$ is given by \eqref{en0}. Then $E_n$ is equivalent to
\begin{equation*}
\begin{split}
E_n & \sim  \|D_t\kappa\|^2_{H^{3n + 2}}   + \|\kappa\|^2_{H^{3n+ 7/2}} + \|\Lambda^{1/2} \psi \|^2_{L^2}+ \|\p h\|^2_{L^{2}} \\
& \sim \|\Lambda^{1/2} \psi\|^2_{H^{3n+\frac{9}{2}}}+ \|\p h\|^2_{H^{3n+ 9/2}} \\
& \sim \| \Lambda^{1/2} u \|^2_{H^{3n+\frac{9}{2}}}.
\end{split}
\end{equation*}

Global energy estimates will be derived under the assumptions
\begin{equation}\label{assume1}
\begin{cases}\tag{a1}
E_{2K}(t) \lesssim  \epsilon, \mbox{ for } \epsilon \ll 1, \quad \text{for  $K$ large},\\
b(t) \overset{def}{=} \|\p^2 h\|_{W^{3K + 7/2 ,\infty}} +\| D v\|_{W^{3K + 3,\infty}} \lesssim \epsilon t^{-1},  \quad \text{for $ t\ge 1$}. 
\end{cases}
\end{equation}
which follow from $\|u\|_{\rm{decay}} \lesssim \epsilon$.
\begin{proposition}\label{prop:hs}
Under assumption \eqref{assume1}, solutions of \eqref{E:dttk2} with data $(h_0,v_0)$
\[   E_{2K}(0)
 \sim   \|v_0\|^2_{H^{6K +4}}+ \|\p h_0\|^2_{H^{6K +9/2}} \le \epsilon^2 \ll1,
\]
satisfy
\begin{equation}
\label{energyn}
E_{2K}(t) \lesssim \epsilon^2 t^{2\delta},
\end{equation}
\end{proposition}
\proof

Recall that  on the boundary $\SCB$,  we have  from \eqref{firstk}
\[
v\sim \D_th, \qquad \Delta_\SCB v \sim \D_t\kp, \qquad N\cdot \D_tv  \sim \CN\kappa \sim D^3 h.
\]
From Lemma \ref{linear:prop} we have for $n \le 2K$
\begin{align}  \label{edot:eq}
\dot \CE_n \lesssim    \|(Dv,\p h Dv, \p^2 h)\|_{W^{1, \infty}} (\CE_n + \|D_t\kappa_n\|_{L^2}^2 + \|\CD \kappa_n\|^2_{\dot H^{\half}}) + \|\CD\CN R_n\|_{L^2}\sqrt{\CE_n}.
\end{align}
From Lemma \ref{com1-est} and \eqref{assume1}, we have  for $n\ge4$
\[
\|\CD\CN\CA^{n-i} D_t[D_t,\CA]\kappa_{i-1} \|_{L^2} +  \|\CD\CN\CA^{n-i} [D_t,\CA]\D_t \kappa_{i-1} \|_{L^2} \lesssim \frac \epsilon t \sqrt{E_n},
\]
where we bound the low derivatives in $L^\8$ and the high derivatives in $L^2$.  From   
\eqref{r0approx}   we know that    $R_0 \sim  \D_t\kp D_t\p h + (\p\kappa)^2 $, and thus 
$$
\|\CD\CN\CA^nR_0\|_{L^2}    \lesssim \frac \epsilon t \sqrt{E_{n}},
$$
by assumption \eqref{assume1}.  Consequently \eqref{edot:eq} can be rewritten as $\dot \CE_n \lesssim    \frac \epsilon t E_{n}$, 
and by summing over $n \le 2K$ we get the stated inequality.
\endproof

Note that in terms of $(\psi,h)$, where $\psi$ is the boundary value of the potential $\psi_\CH$,  $E_0$ bounds
$\psi \in  \dot H^{\frac 12}(\RR^2)$   and $h\in \dot H^{1}(\SCB)$, and thus 
Proposition \ref{prop:hs}  gives control on
\[
E_{2K} \sim \|\Lambda^\half\psi\|^2_{H^{6K + 9/2}(\RR^2)} +  \|\p h \|^2_{H^{6K +9/2}(\RR^2)}
\sim \|\Lambda^\half u\|^2_{H^{6K + 9/2}(\RR^2)},
\]
where $u = \Lambda^{1/2} h + i \psi$.  


\subsection{Weighted energy  estimates}  Weighted energy estimates will be derived by commuting   $( \TCS, \TOM)$  with  the $\kappa$ equation  \eqref{E:dttk2}. 
The main difference, and difficulty, between the energy estimate and weighted estimates is the following.
Commuting $D_t$ with $\CN$ gives a first order differential operator with coefficients depending on  $\tilde w_1 = (Dv, \p h Dv)$, 
while commuting $\CS$ or $\Omega$ with $\CN$ gives a first order differential operator with coefficients depending on  $\tilde w_2 = (\p h, \CS\p h)$, 
The crucial difference is that $\tilde{w}_1$ decays $\sim \frac{1}{t}$ in $L^\infty$, whereas $\tilde{w}_1$ decays $\sim \frac{1}{t^{1-\delta}}$
(compare assumptions \eqref{assume1} and \eqref{assume2} for instance).
For this reason  weighted estimates are slightly worse than regular energy estimates.
We also would like to remind the reader that  since    $\CS$  in commutators  acts as $3/2$ derivatives, we will treat $( \TCS, \TOM,\CA)$, where $\CA= \Delta_\SCB\CN$,   
on equal footing; these operators are collectively denoted by $\Gamma$.  

In the sequel we will let  $\hd$ stand for  either $\p$ or $\CN$, and $\CX$ stand for 
$$
\CX = \TCS, \TOM.
$$
Finally, we define the spaces, 
%
%
\begin{gather}
\widetilde W^{k,p}(\SCB) = \{g \in W^{k,p}(\SCB)  ; \;  \CN g  \in W^{k-1,p}(\SCB)\}\\
L^p_\ell(\SCB) = \left\{ \varphi : \; \SCB \to \RR; \;  \|\varphi\|_{L^p_\ell} \overset{def}{=} \; \sum_{i=0}^\ell\|\Gamma^i \varphi\|_{L^p} \right\},\\
W^{k,p}_\ell(\SCB) =  \left\{ \varphi : \; \SCB \to \RR; \;  \|\varphi\|_{W^{k,p}_\ell} \overset{def}{=} \; \sum_{i=0}^\ell\|\Gamma^i \varphi\|_{\widetilde W^{k,p}} \right\},
\end{gather}
for $p= 2 \text{ or } \infty$, $k, \ell\in \NN$. For $p=2$ we allow $k\in\NN/2$.  Note that  $\widetilde W^{k,2}(\SCB) =  W^{k,2}(\SCB) $, however  $\widetilde W^{k,\8}(\SCB) \neq W^{k,\8}(\SCB)$,  due to the presence of $\CN$.  

The relation between $v$, $h$, and $\kappa$ in $W^{k,p}_\ell$ norms is as follows.
\begin{proposition} Under the assumption
\[
\|\p h\|_{W^{3,\infty}_{K}} + \|v\|_{W^{2,\infty}_{K} }\ll 1, \qquad  \text{ for }\; K>1,
\] 
 solutions  to  system \eqref{firstk}   satisfy \label{psidtk}
\begin{gather}
\|\p h\|_{W^{a+1,2}_{2K}(\SCB)} \sim \|\kappa\|_{W^{a,2}_{2K}(\SCB)}+ \|\p h\|_{L^2_{2K}(\SCB)}\label{hk}\\[.3em]
\|v \|_{W^{a+2,2}_{2K}(\SCB)}   \sim \| \D_t \kappa \|_{W^{a,2}_{2K}(\SCB)}      + \|\kappa\|_{W^{a,2}_{2K}(\SCB)}  +  \|\p h\|_{L^2_{2K}(\SCB)} +  \|v^\perp \|_{L^2_{2K}(\SCB)}.
\end{gather}
\end{proposition}
\proof
The first inequality follows from $\kappa = -(1+|\p h|^2)^{\frac 12} \De_\SCB h$, commuting $\CX$ with $\De_\SCB$,  and standard elliptic theory.  To prove the second inequality we  note that since $v$ is the gradient of a harmonic function, then  by  equation \eqref{eqhk}
\[
\|D^2v \|_{L^{2}(\SCB)}  \lesssim \|v \|_{H^{2}(\SCB)}  \lesssim  \|v^\perp \|_{H^{2}(\SCB)}  \lesssim  \|\Delta_{\p \SCB} v^\perp  \|_{L^{2}(\SCB)} +  \|v^\perp \|_{L^2(\SCB)}
\]
From equation \eqref{E:dtk2}
\[
\Delta_{\SCB} v^\perp=-\D_t \kappa  - v^\perp |\Pi|^2 + (\CD \cdot\Pi)(v^\top),
\]
the commutation relations of $\CX$, and by the fact that  $|\Pi| + |\CD\Pi|\lesssim | \p h| + |D^2h| +|D^3h|$,  we conclude  the stated result.
\endproof
%
%
%
%
%
%
%

Set $\tilde \kappa_j  = \kappa_{2K-j} = \CA^{2K-j}\kappa$, for any $0 \le j\le 2K$. It solves the equation
\begin{equation}
\D_t^2 \tilde\kappa_j =  \Delta_{\SCB} \CN (\tilde \kappa_j)  + \tilde R_j, \label{tildekj}
\end{equation}
where $\tilde R_j = R_{2K-j}$ is defined by \eqref{Rninduc}.
Next, let
$$
\mathscr{F}_j^{\ell} \overset{def}{=} 
\frac{1}{2} \left\| \CD \CN D_t \CX^j \tilde \kappa_{j+\ell} \right\|_{L^2}^2 + \frac{1}{2} \langle \CN \Delta_{\mathscr{B}} \CN \CX^j
\kappa_{j+\ell} , \Delta_{\mathscr{B}} \CN \CX^j \kappa_{j+\ell} \rangle
$$
where $\ell = 0 \dots 2K - j$. Finally, define the weighted energy that will be controlled:
$$
F_n \overset{def}{=} \sum_{j = 0}^n \sum_{\ell = 0}^{2K-j} \mathscr{F}_j^{\ell} + E_{\rm{physical}}.
$$
Notice that
\begin{equation*}
\begin{split}
F_n(t) & \sim \left\| D_t \kappa \right\|_{W^{3(2K-n)+2,2}_n}^2 + \left\| \kappa \right\|_{W^{3(2K-n)+\frac{7}{2},2}_n}^2 + \|\Lambda^{1/2} \psi\|_2^2 + \| \p h \|_2^2 \\
& \sim \left\| v \right\|_{W^{3(2K-n)+4,2}_n}^2 + \left\| \p h \right\|_{W^{3(2K-n)+\frac{9}{2},2}_n}^2 + \|\Lambda^{1/2} \psi\|_2^2 \\
& \sim \left\| \Lambda^{1/2} u \right\|_{W_n^{3(2K-n)+\frac{9}{2},2}}^2.
\end{split}
\end{equation*}
%

To derive the weighted energy estimate we need to commute $\CX$ with the operator $\D_t^2 -\Delta_\SCB\CN$.  The next lemma illustrates how and where terms such as $\CS\p h$ appear in the commutators.  
\begin{lemma}\label{lem:skj} Let $g$ denote a smooth function on $\p\CU$ and write 
\begin{equation}\label{eq:lin-eq}
D^2_t g - \Delta_\SCB \CN g = F \quad \text{ on }\quad \p\CU.
\end{equation}
Then 
\begin{equation}
\begin{split}  \label{eq:xg}
\| \CD\CN [ \CX, \De_{\SCB} \CN- D^2_t  ]  g \|_{L^2}  \lesssim & \left(\|\CX\hd h\|_{\widetilde W^{4,\8}}+ \|\CX v\|_{\widetilde W^{2,\8}}\right)\left(
\|  \hd   g\|_{H^4}+ \|\hd D_tg\|_{H^2}\right)\\
&+ \|\CD\CN F\|_{L^2}.
\end{split}
\end{equation}
\end{lemma}
\proof
For $\CX=\CS$, we have from  the commutation relations \eqref{comm-lap} and \eqref{comm-n}  \begin{align*}
 [\TCS, \Delta_\SCB \CN -\D^2_t]  g =&\left( \Delta_\SCB[\TCS,\CN] + [\TCS, \Delta_\SCB]\CN  - \D_t[\TCS,D_t] - [\TCS,D_t]D_t\right) g\\[.3em]
 =& \Delta_\SCB  \left( \nabla_{\CS N} g_{\CH} -( \CN  Z )\p_z g_{\CH}  + N\cdot \nabla  \Delta_0^{-1} 2{\rm div}(( \nabla Z_\CH) \p_z g_{\CH} )\right) \\[.3em]
 &  -\left(2 \Delta_{\SCB} \CN g + 2\CD^2\CN g \cdot \nabla^\top (Z\mathbf{k}) + \left( \Delta_\SCB S \right)\cdot (\nabla^\top \CN g ) +\kappa\CS N\cdot (\nabla^\top\CN  g )\right)\\[.3em]
 &-D_t(\frac12v^i+\CS v^i)\p_i g -(\frac12v^i+\CS v^i)\p_iD_t  g +3 D_t F.
\end{align*}
Writing  $A_{\nabla Z_\CH}(\p_z g) =  \Delta_\SCB\left( N\cdot \nabla  \Delta_0^{-1} 2{\rm div}(( \nabla Z_\CH) \p_zg_{\CH}\right)$ (which is an operator of order 2), 
this commutator can be written as 
\[
\begin{split}
[ \TCS, \De_{\SCB} \CN- D^2_t  ]  g = &  A_3 (\TCS \hd h) \hd^3  g + A_2 (\TCS\hd^2  h) \hd^2  g
+ A_1 ( \TCS\hd^3 h) \hd  g \\
&+ A_{\nabla Z_\CH}(\p_z g_{\CH}) +  B_1(v,\TCS v)\hd \D_t  g + 3 D_t F,
\end{split}
\]
where we exhibited the dependence of the coefficients  on the least regular terms and terms  with the least number of $\hat D$. 
Here we introduced the notation $A_3\hd^3 =\sum_{i=1}^{3} A^i_3 \CN^iD^{3-i}$, and so on. This implies 
\[
\begin{split}
\| \CD\CN [ \TCS, \De_{\SCB} \CN- D^2_t  ]  g \|_{L^2} & \lesssim \left(\|\CS\hd h\|_{\widetilde W^{4,\8}}+ \|\CS v\|_{\widetilde W^{2,\8}}\right)\left(
\|  \hd   g\|_{H^4}  + \|\hd D_tg\|_{H^2}\right)+ \|\CD\CN D_t F\|_{L^2}.
\end{split}
\]
Similar estimates hold for $\CX=\Omega$.
\endproof
\begin{remark}
Note that this lemma, when combined with Lemma \ref{linear:prop}, gives an estimate with $1/2$ regular derivatives loss on $g$.
However since the loss is in regular derivatives, we can handle this by regular energy estimates. 
Moreover since terms like $\CS\hd h$ decay at a rate less than $t^{-1}$, this will cause a cascading effect on the growth rate of the  weighted energy estimates.
\end{remark}

Bounds for $F_n(t)$  will be derived under the assumptions
\begin{equation}\label{assume2}
\begin{cases}\tag{a2}
F_0(t) \lesssim \epsilon^2 t^{2\delta}, \; 
\|v\|^2_{W^{4,2}_K}+ \|\p h\|^2_{W^{{9/2},2}_K} \lesssim  \epsilon, \mbox{ for } \epsilon \ll 1\\
a(t) \overset{def}{=} \|\p h\|_{{\widetilde W}_{K}^{9/2 ,\infty}} +\| v\|_{{\widetilde W}_{K}^{4,\infty}} \lesssim \epsilon t^{-1+\delta},  \quad   t\ge 1,\\
b(t) \overset{def}{=} \|\p^2 h\|_{{\widetilde W}_{K}^{7/2 ,\infty}} +\| D v\|_{{\widetilde W}_{K}^{3,\infty}} \lesssim \epsilon t^{-1},  \quad t\ge 1,
\end{cases}
\end{equation}
which follow from $\|u\|_{\rm{decay}} \lesssim \epsilon$.

\begin{proposition}\label{prop:energy} Under assumption \eqref{assume2}, solutions of \eqref{E:dttk2} with initial  data 
\[
\|v_0\|^2_{W^{4,2}_{2K}}+ \|\p h_0\|^2_{W^{9/2,2}_{2K}} \le \epsilon^2 \ll1,
\]
satisfy
\begin{equation}\label{energyw}
F_j(t) \lesssim \epsilon^2 t^{2(j+1)\delta},
\end{equation}
for $0\le j \le 2K$.
\end{proposition}

\begin{proof} \underline{Step 1: the ODE controlling $F_j$.} The proof will be constructed inductively on $j$.  The case $j=0$ is the non-weighted estimates.
Since the difficulty in controlling $F_j$ is in controlling high order weighted derivatives, we will only keep track of the highest order weighted derivative terms in $F_j$,
that is $\mathscr{F}_j^0$.  
Recall that     $\tilde\kappa_j=\CA^{2K-j}\kappa$ satisfies \eqref{tildekj}, and that
\begin{equation}
\SCF_j^0 (t)  =\frac 12\|\CD\CN \D_t \CX^j \tilde\kappa_j\|_{L^2} ^2 + \frac 12 \langle \CN \Delta_\SCB \CN  \CX^j\tilde\kappa_j,   \Delta_\SCB\CN \CX^j \tilde\kappa_j\rangle.
\end{equation}

For $ 1\le j \le 2K$,   $\CX^j\tilde\kappa_j$ satisfies the equation
\begin{align}
&\D_t^2\CX^j \tilde\kappa_j =  \Delta_{\SCB} \CN (\CX^j \tilde \kappa_j)  + \CR_j,\\
&\CR_{j} \overset{def}{=} \sum_{i=1}^{j}\CX^{{j}-i} \left([\CX, \Delta_\SCB\CN - \D^2_t] \CX^{i-1}\tilde \kappa_{j}\right)  + \CX^{{j}}{\tilde R}_j,  \label{eq:calr}
\end{align}
and consequently from Lemma \ref{linear:prop} we have
\be\label{dte}
\begin{split}
\frac d{dt} \SCF_j^0 (t) =& \langle ( v^\perp \kappa + \CD \cdot v^\top )\CD  \CN\D_t\tilde \kappa_j  ,\CD  \CN\D_t \tilde \kappa_j \rangle 
+ 2\langle [D_t, \CD] \CN D_t\tilde\kappa_j,  \CD  \CN\D_t\tilde\kappa_j \rangle + 2 \langle\CD  [D_t, \CN ]D_t\tilde\kappa_j,   \CD \CN\D_t\tilde\kappa_j \rangle  
\\ &+ \langle[D_t, \CN] \Delta_\SCB\CN \tilde \kappa_j,   \Delta_\SCB\CN \tilde \kappa_j \rangle  
+2 \langle \CN[D_t , \Delta_\SCB\CN] \tilde \kappa_j,   \Delta_\SCB\CN \tilde \kappa_j \rangle +  \langle \CD \CN {\CR}_j,   \CD \CN\D_t\tilde \kappa_j \rangle,
\end{split}
\ee
and
\begin{align*}  
\frac d{dt} \SCF_j^0   \lesssim    \|(Dv, \p h Dv,\p^2 h)\|_{\widetilde W^{3, \infty}}\mathscr{F}_j  + \|\CD\CN \CR_1\|_{L^2}\sqrt{ \SCF_j^0 } 
\lesssim \frac{\epsilon}t \mathscr{F}_j  + \|\CD\CN \CR_j\|_{L^2}\sqrt{ \SCF_j^0 }.
\end{align*}
Thus to prove the proposition we have to bound $\CR_j$, which will be done inductively. 

\bigskip

\noindent
\underline{Step 2: the case $j=1$.}
For $j=1$
\[
\begin{split}
&\CR_{1} =[\CX, \Delta_\SCB\CN - \D^2_t] \tilde \kappa_{1} + \CX{\tilde R}_1
\end{split}
\]
has to be bounded carefully due to the presence of $\CX\p h$ in the commutators whenever  $\CX=\CS$ or $\Omega$.
To estimate this term when $\CX= \CS$, we use Lemma \ref{lem:skj} to obtain 
\[
\begin{split}
\| \CD\CN [ \TCS, \De_{\SCB} \CN- D^2_t  ]  \tilde\kappa_1 \|_{L^2} & \lesssim \left(\|\CS\hd h\|_{\widetilde W^{4,\8}}+ \|\CS v\|_{\widetilde W^{2,\8}}\right)\left(
\|  \hd    \tilde\kappa_1\|_{H^4}  + \|\hd D_t  \tilde\kappa_1\|_{H^2}\right)+ \|\CD\CN D_t \tilde R_1\|_{L^2}
\end{split}
\]
From  the regular energy estimates  and   assumption  \eqref{assume2}  we can bound the right-hand side as follows
\begin{align*}
&\|  \hd    \tilde\kappa_1\|_{H^4}  + \|\hd D_t  \tilde\kappa_1\|_{H^2}\lesssim \sqrt{E_{2K}} \lesssim  \epsilon
 t^\delta\\[.3em]
& \|\CD\CN D_t \tilde R_1\|_{L^2}   \lesssim  \frac \epsilon{t}\sqrt{E_{2K}} \lesssim \frac {\epsilon^2} {t^{1-\delta}},\\[.3em]
&\|\CS\hd h\|_{\widetilde W^{4,\8}}+ \|\CS v\|_{\widetilde W^{2,\8}}  \lesssim \frac \epsilon {t^{1-\delta}}.
\end{align*}
Thus $  \| \CD\CN [ \TCS, \De_{\SCB} \CN- D^2_t  ]  \tilde\kappa_1 \|_{L^2}    \lesssim  \frac {\epsilon^2} {t^{1-2\delta}}$.
Similarly, observe that
\[
\tilde R_1 = R_{2K -1}= \sum_{i=1}^ {2K -1}\CA^{ {2K -1}-i} \left(D_t[D_t,\CA]\kappa_{i-1} + [D_t,\CA]\D_t\kappa_{i-1}\right)  + \CA^{ {2K -1}}R_0,
\]
is a sum of products of the following type: a function which carries the highest number of derivative, which can be estimated in $L^2$; and a function
that decays $\sim \frac{1}{t}$ in $L^\infty$ due to Assumption~(\ref{assume2}). Therefore, we get
\[
\| \CD\CN \CS \tilde R_1   \|_{L^2}  \lesssim     \frac \epsilon{t}\sqrt{F_1},
\]
Similar estimates  holds when $\CX=\Omega$.  
Thus  we conclude 
\[
\frac d{dt} \sqrt{\SCF^0_1(t)}  \lesssim  \frac \epsilon{t}\sqrt{F_1}  +\frac {\epsilon^2} {t^{1-2\delta}}  \sqrt{F_0} \implies 
\frac d{dt} \sqrt{F_1(t)}  \lesssim  \frac \epsilon{t}\sqrt{F_1}  + \frac{\epsilon^2}{t^{1-2\delta}}
\implies F_1(t)  \lesssim \epsilon^2 t^{4\delta},
\]
which proves the proposition for $j=1$.   

\bigskip

\noindent
\underline{Step 3: the case $j\geq 2$.}
For $j\ge2$  we proceed inductively in $j$ to estimate  $\|\CD\CN \CR_j\|_{L^2}$. Assume we have verified the proposition for $j-1$. 
To show it holds for $j$ we have to apply  $\CX^j$ to the $\tilde\kappa_j$ equation. Recall that
\[
\begin{split}
[ \CX, \De_{\SCB} \CN- D^2_t  ]  \CX^{i-1}\tilde \kappa_{j} = &  A_3 (\CX \hd h) \hd^3  \CX^{i-1}\tilde \kappa_{j} + A_2 (\CX\hd^2  h) \hd^2  \CX^{i-1}\tilde \kappa_{j}
+ A_1 ( \CX\hd^3 h) \hd  \CX^{i-1}\tilde \kappa_{j} \\
&+ A_{\nabla Z_\CH}(\p_z \CX^{i-1}\tilde \kappa_{j \CH}) +  B_1(\CX v )\hd \D_t  \CX^{i-1}\tilde \kappa_{j} + 3 D_t \CX^{i-1}\tilde \kappa_{j}.
\end{split}
\]
In equation \eqref{eq:calr}, these terms
generate quadratic terms of the form
\[
\begin{split}
A_3 ({\CX^{\ell+1}}\hd h) \hd^3\CX^{j-\ell-1}  \tilde\kappa_j+ &A_2 (\CX^{\ell+1}\hd^2  h) \hd^2\CX^{j-\ell-1} \tilde\kappa_j\\
+ &A_1 ( {\CX^{\ell+1}}\hd^3 h) \hd \CX^{j-\ell-1}\tilde\kappa_j +  B_1({\CX^{\ell+1}}v )\hd \D_t \CX^{j-\ell-1} \tilde\kappa_j,
\end{split}
\]
(for $ 0\le \ell \le j-1$) as well as terms of the type
\begin{equation}
\label{albatros}
\begin{split}
& \Delta_\SCB \left( N\cdot \nabla  \Delta_0^{-1} 2{\rm div}(( \CX_\CH^{\ell}  \nabla X_\CH)  \p_z \CX^{j-\ell-1}\tilde \kappa_{j \CH}\right),\quad 0\le \ell \le j-1 \\
&   \Delta_\SCB\left( N\cdot \nabla  \Delta_0^{-1} 2\CX_\CH^{\ell -1} \left( (\nabla X_\CH)\cdot(  \nabla X_\CH \p_z \CX^{j-\ell-1}\tilde \kappa_{j \CH} )\right)\right), \quad 1\le \ell \le j-1,
\end{split}
\end{equation}
and finally terms which more regular and of higher degree.

The first term in~(\ref{albatros}) is easy to bound, while the second term term gives rise to smooth regular cubic terms and to quartic terms of order $-1$ of the form 
\begin{equation}
 (\Delta_\SCB N)\cdot \nabla \Delta_0^{-1}\CX_\CH^{\ell-1} \left( (\nabla X_\CH)\cdot(  \nabla X_\CH \p_z\CX^{j-\ell-1}\tilde \kappa_{j \CH}  )\right)
\end{equation}
Thus all these commutations lead to quadratic terms 
\begin{equation}
 \begin{split}
Q_j = \sum_{\ell=0}^{j-1}&
  A_3 ({\CX^{\ell+1}}\hd h) \hd^3\CX^{j-\ell-1}  \tilde\kappa_j+ A_2 ({\CX^{\ell+1}}\hd^2  h) \hd^2\CX^{j-\ell-1} \tilde\kappa_j\\
+ &A_1 ( {\CX^{\ell+1}}\hd^3 h) \hd \CX^{j-\ell-1}\tilde\kappa_j +  B_1({\CX^{\ell+1}}v )\hd \D_t \CX^{j-\ell-1} \tilde\kappa_j\\
\end{split}
\end{equation}
quartic terms of order $-1$ of the form 
\begin{equation}
\CK_j =  ( \Delta_\SCB\ N) \cdot \nabla  \Delta_0^{-1} 2\CX_\CH^{\ell -1} \left( (\nabla X_\CH)\cdot(  \nabla X_\CH \p_z \CX^{j-\ell-1}\tilde \kappa_{j \CH} )\right),  \quad 1\le \ell \le j-1,
\end{equation}
and more regular degree 3 or higher terms.  Thus 
\[
\CR_j = Q_j + \CK_j +   \CX^j{\tilde R}_{0} + \mbox{ similar or more regular terms of  degree 3 or higher}.
\]
To bound $\| \CD\CN \CR_{j}\|^2_{L^2}$,  we deal first with the quadratic terms $Q_{j}$.  In this case  
we bound each term in $L^2$  or  $L^\8$, depending on   whether  $\ell \ge K$ or $\ell \le K-1$.
To estimate $A_3 ({\CX^\ell\CS}\hd h) \hd^3\CX^{j-\ell-1}  \tilde\kappa_j$, with $\ell \ge K$,  we proceed as follows: 
\[
\begin{split}
\| \CD\CN   A_3 ({\CX^\ell\CS}\hd h) \hd^3\CX^{j-\ell-1}  \tilde\kappa_j\|_{L^2} & \lesssim 
\|\  \hd^3\CX^{j-\ell-1}  \tilde\kappa_j\|_{\widetilde W^{2,\8}}
\| {\CX^\ell\CS}\hd h \|_{H^2}.
\end{split}
\]
The above expression is bounded by $\epsilon t^{-1}\sqrt{F_j}$.  

For $\ell\le K-1$,  
\[
\begin{split}
&\|  {\CX^\ell\CS}\hd h \|_{\widetilde W^{2,\8}}  \lesssim\frac{\epsilon}{ t^{1-\delta}}, \quad   \text{by assumption \eqref{assume2}},\\
 &\|    \hd^3\CX^{j-\ell-1}  \tilde\kappa_j \|_{H^2}  \lesssim \sqrt{ F_{j-\ell-1}} \lesssim \epsilon  t^{(j-\ell)\delta}, \quad  \text{by the inductive step}.
\end{split}
\]  
Therefore 
\[
\begin{split}
\| \CD\CN   A_3 ({\CX^\ell\CS}\hd h) \hd^3\CX^{j-\ell-1}  \tilde\kappa_j\|_{L^2} & \lesssim 
\|  {\CX^\ell\CS}\hd h \|_{\widetilde W^{2,\8}} 
\|    \hd^3\CX^{j-\ell-1}  \tilde\kappa_j \|_{H^2}  \lesssim 
\frac{\epsilon^2} {t^{1-(j + 1)\delta}}.
\end{split}
\]
The remaining terms in $Q_{j}$ are estimated in the same manner.
The cubic terms are much easier to bound since they appear from multiple commutations of $\CX$ and thus  they have fewer derivatives and better $t$ decay.

The quartic terms $\|\CD\CN\CK_j   \|_{L^2}$ can be  estimated in the following manner: when $\CD\CN$ hits $\Delta_\SCB N$ we bound the term by
\begin{multline*}
\left\|(\CD\CN \Delta_\SCB N)\cdot \nabla \Delta_0^{-1}  \CX_\CH^{\ell -1} \left( (\nabla Z_\CH)\cdot(  \nabla Z_\CH \p_z \CX^{j-\ell-1}\tilde \kappa_{j \CH} )\right)\right\|_{L^2} \\
\lesssim \left\|\p h\right\|_{H^4}\left\| \nabla \Delta_0^{-1}  \CX_\CH^{\ell -1} \left( (\nabla Z_\CH)\cdot(  \nabla Z_\CH \p_z \CX^{j-\ell-1}\tilde \kappa_{j \CH} )\right)\right\|_{L^\8(\SD)} \\
\lesssim\|\p h\|_{H^4} \left\|  \CX_\CH^{\ell -1} \left( (\nabla Z_\CH)\cdot(  \nabla Z_\CH \p_z \CX^{j-\ell-1}\tilde \kappa_{j \CH} )\right)\right \|_{H^1(\SD)}
\end{multline*}
By distributing the  $\CX_\CH^{\ell -1}$ vector fields on the cubic term  we end up with a highest  derivative of $\CX^{j-2}$,  and at least one term with less than $K$ derivatives.  Thus
\[
\left\|(\CD\CN \Delta_\SCB N)\cdot \nabla \Delta_0^{-1}  \CX_\CH^{\ell -1} \left( (\nabla Z_\CH)\cdot(  \nabla Z_\CH \p_z \CX^{j-\ell-1}\tilde \kappa_{j \CH} )\right)\right\|_{L^2} \lesssim  \frac{\epsilon^2} {t^{1-(j + 1)\delta}}.
\]
The remaining terms are quartic with fewer derivatives and are of positive order, thus their estimates are straightforward.

This leaves $\|\CD\CN\CX^{j}{\tilde R}_0\|_{L^2}$ to estimate. Recall that ${\tilde R}_0 = R_{2K-j}$ is given by \eqref{Rninduc}
\begin{align*}
R_{2K-j} =\sum_{i=1}^{2K-j}\CA^{n-i} \left(D_t[D_t,\CA]\kappa_{i-1} + [D_t,\CA]\D_t\kappa_{i-1}\right)  + \CA^{2K-j}R_0,
\end{align*}
so these terms can be bounded by
\[
\|\CD\CN\CX^{j}{\tilde R}_0\|_{L^2} \lesssim \frac{\epsilon}{t} \sqrt{F_j}.
\]
Putting all of these bounds together implies
\[
\frac d{dt} \CF_{j} (t)  \lesssim \frac \epsilon t \sqrt{F_{j}}  +  \frac{\epsilon^2} {t^{1-(j + 1)\delta}},
\]
which gives the desired estimate after summing over $j$.
\end{proof}

\section{Decay estimates}

\label{sectionnonlindecay}

\subsection{Rewriting the equation}

\label{subsecrewrit}

Recall that the equations for the trace of the potential and the graph of the surface are given by (see \eqref{Lin})
 \begin{equation*}
\left\{
\begin{array}{l} \partial_t h = \Lambda \psi - \nabla \cdot (h \nabla \psi) - \Lambda (h\Lambda \psi)  +\tilde R_1 \\
\partial_t \psi = \Delta h - \frac{1}{2}|\nabla \psi|^2 + \frac{1}{2} |\Lambda \psi|^2  + R_2\end{array}
\right.   
\end{equation*}
where $\tilde R_1$ and $R_2$ are terms of degree $3$ and higher,  and where 
$\Lambda$ is the Fourier multiplier of symbol $|\xi|$.
By introducing the variable  $H = \Lambda^{1/2} h$, the above equations can be written as
\begin{equation}
\label{aa}
\left\{
\begin{array}{l} \partial_t H = \Lambda^{3/2} \psi -  \Lambda^{1/2} 
 \nabla \cdot (  \Lambda^{-1/2}  H \nabla \psi) - \Lambda^{3/2} (  \Lambda^{-1/2}  H   \Lambda \psi)  +R_1 \\
\partial_t \psi = -\Lambda^{3/2}  H   - \frac{1}{2}|\nabla \psi|^2 + \frac{1}{2} |\Lambda \psi|^2   +R_2,
\end{array}
\right. 
\end{equation}
where $R_1 \overset{def}{=} \Lambda^{1/2}\tilde R_1  $.   The above system can be written in a compact form in terms of the complex variable  $u = H + i \psi$:
\begin{equation}
\label{aau} \partial_t u = - i \Lambda^{3/2} u + T_{m_{++}}(u,u) + T_{m_{--}}(\bar u,\bar u) + T_{m_{+-}}(u,\bar u) + R,
\end{equation}
where $m_{\pm \pm}$ are linear combinations of 
\begin{equation}
\begin{aligned} \label{m1m2} 
& m_1(\xi,\eta) \overset{def}{=} \frac{|\xi|^{1/2}}{|\eta|^{1/2}}   \left(\xi \cdot (\xi-\eta) - |\xi||\xi-\eta| \right), \\
& m_2(\xi,\eta) \overset{def}{=}  \eta \cdot (\xi-\eta) + |\eta||\xi-\eta|, \\
& \mbox{and} \;\; m_1(\xi,\xi-\eta),
\end{aligned}
\end{equation}
and we set $R \overset{def}{=}R_1+i R_2$. 
The next step is to define
$$
f(t) \overset{def}{=} e^{it\Lambda^{3/2}} u
$$
and to write Duhamel's formula for $f$:
\begin{multline}
\label{eqfourier}
\widehat{f}(t,\xi)  =  \widehat u_0(\xi) + \sum_{(\tau_1,\tau_2) = (++),(+-),(--)} \int_0^t 
\!\!\int e^{is\phi_{\tau_1,\tau_2}(\xi,\eta)} m_{\tau_1 \tau_2} (\xi,\eta) \widehat{f_{-\tau_1}}(s,\eta) \widehat{f_{-\tau_2}}(s,\xi-\eta)\,d\eta \,ds \\ 
\quad \quad \quad + \int_0^t e^{is |\xi|^{3/2}} \widehat{R} (s,\xi )\,ds, \\
\end{multline} 
where $f_{+} \stackrel{def}{=} f$, $  f_{-} \stackrel{def}{=} \bar{f}$, and the phases  are given  by
\begin{equation*}
\begin{aligned}
& \phi_{\pm,\pm}(\xi,\eta) = |\xi|^{3/2} \pm |\eta |^{3/2} \pm |\xi-\eta|^{3/2}.
\end{aligned}
\end{equation*}

Sometimes, the exact structure of~(\ref{aau}) will not matter and we will write it in a simplified form
\begin{equation}
\label{simpleform}
\partial_t u = - i \Lambda^{3/2} u + T_{m}(u,u) + R,
\end{equation}
with $m$ being a linear combination of $m_1$ and $m_2$.

We will also often write indistinctively $f$ for $f$ and $\bar f$; similarly, we will not distinguish between $e^{is|D|^{3/2}}$ and $e^{-is|D|^{3/2}}$.
This alleviates somewhat the notations, and has of course no impact on the estimates.

\subsection{Examination of the bilinear symbols and the phases}

\label{eotbsatp}

\subsubsection{The symbols}

Start with the two multilinear symbols
\begin{equation*}
\begin{aligned}
& \tilde m_1(\xi,\eta) \overset{def}{=}  \xi \cdot (\xi-  \eta)  - |\xi|| \xi - \eta|  \\
& \tilde m_2(\xi,\eta) \overset{def}{=}  \eta \cdot (\xi-\eta) + |\eta||\xi-\eta|  
\end{aligned}
\end{equation*}
Due to the relation $\tilde m_1(\eta,\xi) = - \tilde m_2 (-\xi,-\eta ) $, it is equivalent to prove estimates for either of these symbols. 

Using the simple fact that
$$
|X+\eps| = |X| + X\cdot \frac{\eps}{|X|}  + \frac12 \left(\frac{|\eps|^2}{|X|} \right) 
 - \left(\frac{(\eps \cdot X)^2}{|X|^3}   \right)  + O \left(\frac{|\eps|^3}{|X|^2}  \right),  
$$
we can deduce that for $|\eta|  << |\xi|  $, 
\begin{equation*}
\begin{aligned}
\tilde m_1(\xi,\eta)  & = |\xi|^2 - \eta \cdot \xi - |\xi| \left( |\xi| + \xi \cdot \frac{\eta}{|\eta|} + O \left( \frac{|\eta|^2}{|\xi|} \right) \right) \\
 & = O (|\eta |^2). 
\end{aligned}
\end{equation*}
Therefore
\begin{itemize}
\item If $|\xi|  << |\eta|  $, then   $ | \tilde m_1(\xi,\eta) |  \lesssim |\xi| |\eta|  $ 
\item If $|\eta|  << |\xi|  $, then   $ | \tilde m_1(\xi,\eta) |  \lesssim |\eta|^2 $ 
\item If $|\xi-\eta|  << |\eta|  $, then   $ | \tilde m_1(\xi,\eta) |  \lesssim |\xi-\eta| |\eta|  $ 
\end{itemize}
and thus
\begin{itemize}
\item If $|\xi|  << |\eta|  $, then   $ | \tilde m_2(\xi,\eta) |  \lesssim |\xi|^2  $ 
\item If $|\eta|  << |\xi|  $, then   $ | \tilde m_2(\xi,\eta) |  \lesssim |\eta| |\xi| $ 
\item If $|\xi-\eta|  << |\eta|  $, then   $ | \tilde m_2(\xi,\eta) |  \lesssim |\xi-\eta| |\eta|  $.  
\end{itemize}
In terms of the classes defined in Appendix~\ref{appendixbilin}, this means that $\tilde m_1(\xi,\eta) \in \mathcal{M}^{2,1,2,1}$ and 
$\tilde m_2(\xi,\eta) \in \mathcal{M}^{2,2,1,1}$.

Now 
$$
m_1(\xi,\eta) = \frac{|\xi|^{1/2}}{|\eta|^{1/2}} \tilde m_1(\xi,\eta) \quad \mbox{and} \quad m_2(\xi,\eta) = \tilde m_2(\xi,\eta).
$$
This implies that $m_1(\xi,\eta) \in \mathcal{M}^{2,3/2,3/2,1}$, $m_1(\xi,\xi-\eta) \in \mathcal{M}^{2,3/2,1,3/2}$, and $m_2(\xi,\eta) \in \mathcal{M}^{2,2,1,1}$.
Finally, since any $m_{\pm \pm}$ is a linear combination of
$m_1(\xi,\eta)$, $m_2(\xi,\eta)$, and $m_1(\xi,\xi-\eta)$, we deduce that
$$
m_{\pm \pm} \in \mathcal{M}^{2,3/2,1,1}.
$$

\subsubsection{The phases}

Recall that for a phase $\phi$, the space, time, and space-time resonant sets are given by
$$
\mathcal{T} = \{ \phi = 0 \} \quad,\quad \mathcal{S} = \{ \partial_\eta \phi = 0 \} \quad \mbox{and} \quad \mathcal{R} = \mathcal{S} \cap \mathcal{T}.
$$
Let us study these sets for each of the phases $\phi_{\pm \pm}$.
\begin{itemize}
 \item First $\phi_{++}$: it is obvious that $\mathcal{T}_{++} = \{(0,0)\}$, which is the most favourable case. We will therefore deal with this interaction by using 
a normal form transform, producing the symbol $\frac{m_{++}}{\phi_{++}}$, which belongs to $\mathcal{M}^{1/2,3/2,1,1}$.
\item For $\phi_{+-}$, we have $\partial_\eta\phi_{+-} =  + \frac32  \frac{\eta}{|\eta|^{1/2}}  +\frac32  \frac{\xi -  \eta}{|\xi- \eta|^{1/2}}   $ and hence 
 we find $\mathcal{S}_{+-} = \mathcal{R}_{+-} = \{\xi=0\}$. Thus in this case estimates will be obtained by an integration by parts in $\eta$.
 Expanding $\partial_\eta \phi_{+-}$ around $\xi=0$ gives:
$$
\partial_\eta \phi = \frac{3}{2|\eta|^{1/2}} \left( \xi + \frac{1}{2} \left( \xi \cdot \frac{\eta}{|\eta|} \right) \frac{\eta}{|\eta|} \right) + 
O\left(\frac{|\xi|^2}{|\eta|^{1/2}}\right).
$$
When integrating by parts in $\eta$, the symbol will be changed to $\frac{m_{+-}}{|\partial_\eta \phi_{+-}|^2} \partial_\eta \phi$, which, by the above estimate, belongs 
to $\mathcal{M}^{3/2,1/2,1,1}$.
\item Finally, for $\phi_{--}$, we have  $\partial_\eta \phi =  - \frac32  \frac{\eta}{|\eta|^{1/2}}  +\frac32  \frac{\xi -  \eta}{|\xi- \eta|^{1/2}}    $ 
and hence 
 we  find that  $\mathcal{S}_{+-}   =\{ \xi = 2\eta    \}     $  and 
  $  \mathcal{R}_{--} = \{(0,0)\}  $. 
As we will see, by splitting the $(\xi,\eta)$ space into regions where $\phi$ and
$\partial_\eta \phi$ do not vanish, this case can be reduced to estimates similar to those for $\phi_{++}$ and $\phi_{+-}$.
\end{itemize}

\subsection{Invariances and commutators}

\label{invandcom}

The solutions of equation~(\ref{aau}) are invariant by translation $u \mapsto u(x+\delta)$, $\delta \in \mathbb{R}^2$, rotation $u \mapsto u(R_\theta x)$, and dilation
$u \mapsto \frac{1}{\lambda^{1/2}} u ( \lambda^{3/2} t,\lambda x)$. 

The generators of these transformations give the vector fields\footnote{Notice that the vector field $\mathcal{S}$ equals the generator of the scaling transformation up 
to an additive constant only.}
$$
\p^3 = (\partial_1,\partial_2)^3,\;\;   
\Omega = x^1 \partial_2 - x^2 \partial_1=\omega^i\p_i,  
\mbox{ and } \mathcal{S} = \frac{3}{2}t\partial_t + x^i \p_i.
$$
We shall also use the space part of $\mathcal{S}$, namely
$$
\Sigma = x^i \p_i.
$$
The vector fields $\partial^3,\Omega,\mathcal{S}$ are collectively denoted by $\Gamma$, and the multiindex notation for powers of $\Gamma$ is used.

Of course, $\p^3$ commutes exacly with the linear part of the equation, and the linear group. 
As for the nonlinear part of the equation, it can be expanded into pseudo-product operators
(by translation invariance) and regular derivatives can be commuted with pseudo-product operators according to the Leibniz rule.

The commutation of the vector fields $\Omega$ and $\mathcal{S}$ with the linear part of the equation and the linear group are given by
\begin{equation}
\begin{split}
& \left[ \Omega,\partial_t + i \Lambda^{3/2} \right] = 0 \\
& \left[ \mathcal{S},\partial_t + i \Lambda^{3/2} \right] = - \frac{3}{2} \left( \partial_t + i \Lambda^{3/2} \right) \\
& \left[ \Omega, e^{\pm i t \Lambda^{3/2}} \right] =  \left[ \mathcal{S}, e^{\pm i t \Lambda^{3/2}} \right] = 0 .
\end{split}
\end{equation}

Next, to commute these vector fields with the nonlinear part of the equation, expand it into a series of multilinear operators: 
in other words, write~(\ref{aau}) as
$$
\partial_t u = -i \Lambda^{3/2} u + \sum_{k=2}^\infty L_k (u,\dots,u),
$$
where $L_k$ is $k$-linear. The invariance by translation and rotation gives immediately that the Leibniz rule applies to $\partial$ and $\Omega$, namely
\begin{equation}
\begin{split}
& \partial L_k (u,\dots,u) = L_k (\partial u,\dots,u) + \dots + L_k (u,\dots,\partial u) \\
& \Omega L_k (u,\dots,u) = L_k (\Omega u,\dots,u) + \dots + L_k (u,\dots,\Omega u).
\end{split}
\end{equation}
As for the dilation operators, notice that the scaling invariance gives
$$ 
L_k \left(\frac{1}{\lambda^{1/2}} u ( \lambda^{3/2} t,\lambda x),\dots,\frac{1}{\lambda^{1/2}} u ( \lambda^{3/2} t,\lambda x) \right)
= \lambda \left[ L_k (u,\dots,u) \right] ( \lambda^{3/2} t,\lambda x)
$$
Taking the derivative in $\lambda$ yields the modified Leibniz rule:
$$
\mathcal{S} L_k (u,\dots,u) = L_k (\mathcal{S} u,\dots,u) + \dots + L_k ( u,\dots,\mathcal{S}u) - \left( \frac{k}{2} + 1 \right) L_k (u,\dots,u).
$$
Finally, we will need to commute vector fields with Fourier multipliers of the type $\Lambda^\alpha$. The formulas are easily computed:
\begin{equation*}
\begin{split}
& \left[\p,\Lambda^\alpha\right] = \left[\Omega,\Lambda^\alpha\right] = 0 \\
& \left[\mathcal{S},\Lambda^\alpha\right] = \left[\Sigma,\Lambda^\alpha\right] = \alpha \Lambda^\alpha.
\end{split}
\end{equation*}
In particular, we shall remember that
\begin{equation*}
\begin{split}
& \Gamma \Lambda^\alpha = \Lambda^\alpha \Gamma + \mbox{ \{ lower order terms \} } \\
& \Sigma \Lambda^\alpha = \Lambda^\alpha \Sigma + \mbox{ \{ lower order terms \} }.
\end{split}
\end{equation*}

\subsection{The estimates}

We prove the following a priori estimate (recall that
$f = e^{it\Lambda^{3/2}} u$ and that $\alpha = \alpha_* + 3 \iota$).

\begin{prop}
\label{eagle}
Assuming $\|u\|_X < \infty$ and~(\ref{conditiondata}),
\begin{align}
\label{D1} \left\|Y(\partial)^3 \Lambda^\alpha u \right\|_{W^{9,2}_{8+K}} &\lesssim \left( \eps + \|u\|_X^2 \right) \<t\>^{\delta'} \\
\label{D1'} \left\| Y(\partial)^3 \Lambda^\alpha \Sigma f \right\|_{W^{9,2}_{7+K}} &\lesssim \left( \eps + \|u\|_X^2 \right) \<t\>^{\delta'}   \\
\label{D2} \left\| Y(\partial)^3 \Lambda^{1/2+\alpha-\beta} u \right\|_{W^{7,\infty}_{4+K}} & \lesssim  \frac{\eps + \|u\|_X^2}{ \<t\>^{1-\delta'-\frac{2}{3}\beta}} 
\qquad \mbox{if $0\leq \beta \leq 1/2$} \\
\label{D3}  \left\|Y(\partial)^2 \Lambda^\alpha u \right\|_{W^{7,2}_{4+K}} & \lesssim \eps + \|u\|_X^2  \\
\label{D3'} \left\| Y(\p)^2 \Sigma \Lambda^\alpha f \right\|_{W^{7,2}_{3+K}} & \lesssim \eps + \|u\|_X^2  \\
\label{D4} \left\| Y(\p) \Lambda^{1/2+\alpha-\beta} u \right\|_{W^{5,\infty}_{K}} & \lesssim  
\frac{\epsilon + \|u\|_X^2}{\<t\>^{1-\frac{2}{3} \beta}} \qquad \mbox{if $0 \leq \beta \leq 1/2$}.
\end{align}
\end{prop}
This proposition gives in particular the desired a priori estimate
$$
\|u\|_{\operatorname{decay}} \lesssim \epsilon + \|u\|_X^2.
$$

\begin{rem}
Notice that one of the reasons why many derivatives are lost in this 
argument is that we cannot interpolate easily the derivatives $\mathcal{S}$ and $\Omega$, 
namely we do not have a Trudinger type inequality for such derivatives. 
\end{rem}

\subsection{Proof of~(\ref{D1})}\label{lowfreq}

The growth estimate~(\ref{D1}) follows from the result of Section~\ref{sectionenergy} as far as high frequencies are concerned. Thus it
suffices to control $Y(\p)^3 \Lambda^\alpha \Gamma^{k} u$ in $L^2$ for $|k| \leq K+8$. We fix such a $k$ and apply 
$Y(\p)^3 \Lambda^\alpha \Gamma^k$ to~(\ref{simpleform}). This gives
\begin{equation}
\label{pinguin}
Y(\p)^3 \Lambda^\alpha \Gamma^{k} \left( \partial_t  + i \Lambda^{3/2} \right) u =  Y(\p)^3 \Lambda^{\alpha} \Gamma^{k} T_{m}(u,u) + Y(\p)^3 \Lambda^{\alpha} \Gamma^{k} R.
\end{equation}
Let us discard for the moment the remainder terms; we will come back to them in Section~\ref{remainderterm}.

Next apply Leibniz rule (see Subsection~(\ref{invandcom})) to commute $\Gamma^{k}$ with the linear equation and the bilinear term. Out of the many terms coming out, we single out 
two representative examples: $Y(\p)^3 \Lambda^\alpha T_{m}(\Gamma^{k}u,u)$ and $Y(\p)^3 \Lambda^\alpha T_{m}(\Gamma^{k/2}u,\Gamma^{k/2}u)$ 
(assuming for simplicity that $k$ is even). Thus the equation now reads
\begin{equation*}
\begin{split}
\left( \partial_t  + i \Lambda^{3/2} \right) \Gamma^{k} Y(\p)^3 \Lambda^\alpha u = & Y(\p)^3 \Lambda^\alpha T_{m}(\Gamma^{k}u,u) + Y(\p)^3 \Lambda^\alpha T_{m}(\Gamma^{k/2}u,\Gamma^{k/2}u) \\
& \qquad + \mbox{ \{ similar terms \} }.
\end{split}
\end{equation*}
To estimate the first term on the right-hand side, recall that $m \in \mathcal{M}^{2,3/2,1,1}$ and use Corollary~\ref{Bili-prop1} to get
\begin{equation*}
\begin{split}
\left\|Y(\p)^3 \Lambda^\alpha T_{m}(\Gamma^{k}u,u) \right\|_2 & \lesssim \left\| \Gamma^{k} \Lambda^\alpha u \right\|_{H^2} \left\| \Lambda^{3/4}  u \right\|_{W^{2,\infty}} \\
& \lesssim  \left\| \Gamma^{k}\Lambda^\alpha  u \right\|_2 \left\| \Lambda^{3/4} u \right\|_{W^{2,\infty}} 
+ \left\| \Gamma^{k}\Lambda^\alpha  u \right\|_{\dot{H}^2} \left\| \Lambda^{3/4} u \right\|_{W^{2,\infty}}\\
& \lesssim \frac{\|u\|_X}{\left< t \right>} \left\| \Gamma^{k} \Lambda^\alpha u \right\|_2  + \left< t \right>^{\delta'-1} \|u\|_X^2.
\end{split}
\end{equation*}
To estimate the second term on the right-hand side, we also rely on Corollary~\ref{Bili-prop1} to obtain
\begin{equation*}
\begin{split}
\left\|Y(\p)^3 \Lambda^\alpha T_{m}(\Gamma^{k/2}u,\Gamma^{k/2}u) \right\|_2 & \lesssim \left\| \Gamma^{k/2} \Lambda^\alpha u \right\|_{H^2}  
\left\| \Lambda^{3/4} \Gamma^{k/2}u \right\|_{W^{2,\infty}} \\
& \lesssim \frac{\|u\|_X}{t} \left\|\Gamma^{k} \Lambda^\alpha u \right\|_2 .
\end{split}
\end{equation*}
Performing an energy estimate, (\ref{pinguin}) gives thus the differential inequality
$$
\frac{d}{dt} \left\|Y(\partial)^3 \Lambda^\alpha u \right\|_{W^{0,2}_{K+8}} 
\lesssim \frac{\|u\|_X}{\left< t \right>} \left\|Y(\partial)^3 \Lambda^\alpha u \right\|_{W^{0,2}_{K+8}}  
+ \left< t \right>^{\delta'-1} \|u\|_X^2.
$$
The hypothesis on the initial data~(\ref{conditiondata}) and Gronwall's inequality give the desired conclusion.

\subsection{Proof of~(\ref{D1'})}

It follows from~(\ref{D1}) in a very similar, but slightly simpler, way to how~(\ref{D3'}) follows from~(\ref{D3}). We therefore refer the reader to the proof of~(\ref{D3'}).

\subsection{Proof of~(\ref{D2})}

It follows from~(\ref{D1'}) and Proposition~\ref{decay-prop}.

\subsection{Proof of (\ref{D3})}

\subsubsection{First reduction}

The estimate~(\ref{D3}) is the crucial one, and the one for which space-time resonances will play a role. We want to control 
$\|Y^2(\p) \Gamma^{k} \Lambda^\alpha \p^j u \|_2$ for $|k|\leq K+4$ and $|j|\leq 7$, which are from now on fixed. We will denote
$$
z = \Gamma^{k} \p^j u \quad \mbox{and} \quad g(t) = e^{it\Lambda^{3/2}} z(t).
$$

Start by applying $Y^2(\p) \Lambda^\alpha \Gamma^{k} \p^j$ to the equation~(\ref{aau}). 
Using Leibniz rule (see Subsection~\ref{invandcom}), $\Gamma^{k} \p^j$ can be commuted with the linear and quadratic terms, which produces many terms.
We only record the extremal ones: these are the most difficult to estimate, and we will focus on them. Thus we find
\begin{equation}
\begin{split}
\partial_t Y^2(\p) \Lambda^\alpha z = & - i \Lambda^{3/2} Y^2(\p) \Lambda^\alpha z + Y^2(\p) \Lambda^\alpha T_{m_{++}}(z,u) + Y^2(\p) \Lambda^\alpha T_{m_{--}}(\bar z,\bar u) \\
& \qquad + Y^2(\p)  \Lambda^\alpha T_{m_{+-}}(z,\bar u) + Y^2(\p) \Lambda^\alpha T_{m_{+-}}(u,\bar z) \\
& \qquad + \mbox{\{ Mixed terms \}} + Y^2(\p) \Lambda^\alpha \Gamma^{k} \p^jR.
\end{split}
\end{equation}
Proceeding as in Subsection~\ref{subsecrewrit}, this can be translated in Fourier space into
\begin{equation*}
\begin{split}
Y^2(\xi) |\xi|^\alpha \widehat{g}(t,\xi)& \\
 =  Y^2(\xi) &|\xi|^\alpha \widehat z_0(\xi)  + \sum_{\tau_{1,2} = \pm} \int_0^t Y^2(\xi) |\xi|^\alpha
\!\!\int e^{is\phi_{\tau_1,\tau_2}(\xi,\eta)} m_{\tau_1 \tau_2} (\xi,\eta) \widehat{f_{-\tau_1}}(s,\eta) \widehat{g_{-\tau_2}}(s,\xi-\eta)\,d\eta \,ds \\ 
& \quad \quad \quad + \mbox{\{ Mixed terms \}} + \int_0^t e^{is |\xi|^{3/2}}Y^2(\xi) |\xi|^\alpha \widehat{\Gamma^{k} \p^j R} (s,\xi )\,ds.
\end{split}
\end{equation*}
The constant term $Y^2(\xi) |\xi|^\alpha \widehat z_0(\xi)$, by assumption, can be bounded in $L^2$ by $\epsilon$.

In the three following subsections, we analyze separately the terms corresponding to $++$, $--$, and $+-$ in the above sum to show that
$$
\left\| \int_0^t e^{is|D|^{3/2}} Y(\partial)^2 \Lambda^\alpha \Gamma^k \partial^j \sum T_{m_{\pm,\pm}}(u,u)\,ds \right\|_2 \lesssim \|u\|_X^2.
$$

The last term, which involves the remainder $R$, will be dealt with in Section~\ref{remainderterm} where we show
$$
\left\| \int_0^t e^{is|D|^{3/2}} Y(\partial)^2 \Lambda^\alpha \Gamma^k \partial^j R\,ds \right\|_2 \lesssim \|u\|_X^3.
$$

\subsubsection{The $++$ case}

Let us consider for now the term
$$
Y^2(\xi) |\xi|^\alpha \int_0^t \int e^{is\phi_{++}(\xi,\eta)} m(\xi,\eta) \widehat{f}(s,\eta) \widehat{g}(s,\xi-\eta) \,d\eta\,ds
$$
where $m \in \mathcal{M}^{2,3/2,1,1}$ (recall that we denoted indistinctively $f$ for $f$ or $\bar f$ since the distinction is irrelevant for our purposes).
Using the identity $e^{is\phi} = \frac{1}{i\phi} \partial_s e^{is\phi}$ to integrate by parts gives
\begin{subequations}
\begin{align}
Y^2(\xi) |\xi|^\alpha \int_0^t & \int e^{is\phi_{++}(\xi,\eta)} m(\xi,\eta) \widehat{f}(s,\eta) \widehat{g}(s,\xi-\eta) \,d\eta\,ds \\
& \label{flamingo1} \qquad  = Y^2(\xi) |\xi|^\alpha \int e^{it\phi_{++}(\xi,\eta)} \frac{m(\xi,\eta)}{i \phi_{++}(\xi,\eta)} \widehat{f}(t,\eta) \widehat{g}(t,\xi-\eta) \,d\eta\,ds \\
& \label{flamingo2} \qquad \qquad  - Y^2(\xi) |\xi|^\alpha \int \frac{m(\xi,\eta)}{i \phi_{++}(\xi,\eta)} \widehat{f}(0,\eta) \widehat{g}(0,\xi-\eta) \,d\eta\,ds \\
& \label{flamingo3} \qquad \qquad  - Y^2(\xi) |\xi|^\alpha \int_0^t \int e^{is\phi_{++}(\xi,\eta)} \frac{m(\xi,\eta)}{i \phi_{++}(\xi,\eta)} \partial_s \widehat{f}(s,\eta) 
\widehat{g}(s,\xi-\eta) \,d\eta\,ds \\
& \label{flamingo4} \qquad \qquad  - Y^2(\xi) |\xi|^\alpha \int_0^t \int e^{is\phi_{++}(\xi,\eta)} \frac{m(\xi,\eta)}{i \phi_{++}(\xi,\eta)} \widehat{f}(s,\eta) 
\partial_s \widehat{g}(s,\xi-\eta) \,d\eta\,ds.
\end{align}
\end{subequations}
Keeping in mind that $ \frac{m(\xi,\eta)}{i \phi_{++}(\xi,\eta)}$ belongs to $\mathcal{M}^{1/2,3/2,1,1}$, applying Corollary~\ref{Bili-prop1} gives
$$
\left\|  (\ref{flamingo1}) \right\|_2 \lesssim \left\| \Lambda^{1/4} u \right\|_{W^{1,\infty}} 
\left\| \Lambda^{1/4} z \right\|_{H^1} \lesssim \frac{\|u\|_X^2 }{\left< t \right>^{1-\delta'-\frac{2}{3}(\alpha + \frac{1}{4})}}.
$$
The term~(\ref{flamingo2}) is easier to treat, thus we skip it.
Proceeding as in the proof of~(\ref{D3'}) below, we can prove that
$$
\left\| Y(\p)^3 \Lambda^{1/4} \partial_t f \right\|_2 \lesssim \frac{\|u\|_X^2}{t} \quad \mbox{and} \quad
\left\|Y(\p)^3 \Lambda^{1/4} \partial_t g \right\|_2 \lesssim \frac{\|u\|_X^2}{\<t\>^{1-\delta}}
$$
This gives, with the help of Corollary~\ref{Bili-prop1},
\begin{equation*}
\begin{split}
\left\|  (\ref{flamingo3}) \right\|_2 & =
\left\| \int_0^t \int e^{is\phi_{++}(\xi,\eta)} Y^2(\xi) |\xi|^\alpha \frac{m(\xi,\eta)}{i \phi_{++}(\xi,\eta)} \partial_s \widehat{f}(s,\eta) \widehat{g}(s,\xi-\eta) \,d\eta\,ds \right\|_2 \\
& \lesssim \int_0^t \left\| Y^2(\xi) |\xi|^\alpha e^{is\Lambda^{3/2}} T_{\frac{m(\xi,\eta)}{i \phi_{++}(\xi,\eta)}} \left(  e^{is\Lambda^{3/2}} \partial_s \widehat{f}(s,\eta)\,,
\, e^{is\Lambda^{3/2}} \widehat{g}(s,\xi-\eta) \right) \right\|_2 \,ds \\
& = \int_0^t \left\|Y^2(\xi) |\xi|^\alpha T_{\frac{m(\xi,\eta)}{i \phi_{++}(\xi,\eta)}} \left(  e^{is\Lambda^{3/2}} \partial_s \widehat{f}(s,\eta)\,,
\, e^{is\Lambda^{3/2}} \widehat{g}(s,\xi-\eta) \right) \right\|_2 \,ds \\
& \lesssim \int_0^t \left\| \Lambda^{1/4} \partial_t f \right\|_{H^1} \left\| \Lambda^{1/4} z \right\|_{W^{1,\infty}} \,ds \\
& \lesssim \int_0^t  \frac{\|u\|_X}{\<s\>} \frac{\|u\|_X}{\<s\>^{1-\delta'}}\,ds \\
& \lesssim \|u\|_X^3.
\end{split}
\end{equation*}
Finally, the term~(\ref{flamingo4}) can be treated in a very similar way, thus we omit the details here.

\subsubsection{The $+-$ case}

Let us consider here the term
$$
Y^2(\xi) |\xi|^\alpha \int_0^t \int e^{is\phi_{+-}(\xi,\eta)} m(\xi,\eta) \widehat{f}(s,\eta) \widehat{g}(s,\xi-\eta) \,d\eta\,ds
$$
where $m \in \mathcal{M}^{2,3/2,1,1}$.
The piece of the integral corresponding to $s \in (0,1)$ is easily dealt with; thus we shall consider in the following
$$
Y^2(\xi) |\xi|^\alpha \int_1^t \int e^{is\phi_{+-}(\xi,\eta)} m(\xi,\eta) \widehat{f}(s,\eta) \widehat{g}(s,\xi-\eta) \,d\eta\,ds
$$
(the aim of this manipulation is to allow for the upcoming integration by parts, which will produce a $\frac{1}{s}$ factor, not integrable near $s=0$).
Using the identity $e^{is\phi} = \frac{1}{is|\partial_\eta \phi|^2} \partial_\eta \phi \cdot \partial_\eta e^{is\phi}$ to integrate by parts gives
\begin{subequations}
\begin{align}
& Y^2(\xi) |\xi|^\alpha \int_1^t \int e^{is\phi_{+-}(\xi,\eta)} m(\xi,\eta) \widehat{f}(s,\eta) \widehat{g}(s,\xi-\eta) \,d\eta\,ds \\
& \label{woodpecker1} \qquad  = - Y^2(\xi) |\xi|^\alpha \int_1^t \int e^{is\phi_{+-}(\xi,\eta)} \frac{m(\xi,\eta)\partial_\eta \phi_{+-}(\xi,\eta)}{i s |\partial_\eta \phi_{+-}(\xi,\eta)|^2}
\cdot \widehat{f}(s,\eta) \partial_\eta \widehat{g}(s,\xi-\eta) \,d\eta\,ds \\
& \label{woodpecker2} \qquad \qquad  - Y^2(\xi) |\xi|^\alpha \int_1^t \int e^{is\phi_{+-}(\xi,\eta)} \frac{m(\xi,\eta)\partial_\eta \phi_{+-}(\xi,\eta)}{i s |\partial_\eta \phi_{+-}(\xi,\eta)|^2}
\cdot \partial_\eta \widehat{f}(s,\eta) \widehat{g}(s,\xi-\eta) \,d\eta\,ds \\
& \label{woodpecker3}\qquad \qquad  - Y^2(\xi) |\xi|^\alpha \int_1^t \int e^{is\phi_{+-}(\xi,\eta)} \partial_\eta \cdot \left[ \frac{m(\xi,\eta)\partial_\eta \phi_{+-}(\xi,\eta)}
{i s |\partial_\eta \phi_{+-}(\xi,\eta)|^2} \right] 
\widehat{f}(s,\eta) \widehat{g}(s,\xi-\eta) \,d\eta\,ds.
\end{align}
\end{subequations}
Notice that, due to the vanishing properties of $m$ and $\partial_\eta \phi$ established in Subsection~\ref{eotbsatp} 
the symbol $\frac{m(\xi,\eta)\partial_\eta \phi_{+-}}{i|\partial_\eta \phi_{+-}(\xi,\eta)|^2 |\xi-\eta|}$ belongs to the class $\mathcal{M}^{1/2,1/2,1,0}$. 
Therefore, by Corollary~\ref{Bili-prop1},
\begin{equation}
\label{tree1}
\begin{split}
 & \left\|  (\ref{woodpecker1}) \right\|_2= \left\| Y^2(\xi) |\xi|^\alpha \int_1^t \int e^{is\phi_{+-}(\xi,\eta)} 
\frac{m(\xi,\eta)}{i s |\partial_\eta \phi_{+-}(\xi,\eta)|^2} \widehat{f}(s,\eta) \partial_\eta \widehat{g}(s,\xi-\eta) \,d\eta\,ds \right\|_2 \\
& \qquad \lesssim \int_1^t \frac{1}{s} \left\| e^{is\Lambda^{3/2}}
T_{Y^2(\xi) |\xi|^\alpha \frac{m(\xi,\eta)\partial_\eta \phi_{+-}}{i |\partial_\eta \phi_{+-}(\xi,\eta)|^2 |\xi-\eta|} }
\left( e^{is\Lambda^{3/2}} f(s)\,,\, e^{is\Lambda^{3/2}}\Lambda x g(s) \right) \right\|_2 \,ds \\
& \qquad \lesssim \int_1^t \frac{1}{s} \left\| \Lambda^{1/2} u \right\|_{W^{1,4}} \left\| Y(\partial) e^{is\Lambda^{3/2}}\Lambda x g \right\|_{W^{1,4}} \,ds
\end{split}
\end{equation}
On the one hand, we can interpolate between $L^2$ and $L^\infty$ to obtain
\begin{equation}
\label{tree2}
\left\| \Lambda^{1/2} u \right\|_{W^{1,4}} \lesssim \sqrt{ \left\| \Lambda^\alpha u \right\|_{W^{1,2}} \left\| \Lambda^{1-\alpha} u \right\|_{W^{1,\infty}} } \lesssim \frac{\|u\|_X}{\sqrt{\< t \>}}.
\end{equation}
On the other hand, the pointwise bound
\begin{equation*}
\begin{split}
|\xi| |\partial_\xi \widehat{g}(\xi)| \lesssim \left|\widehat{g}(\xi)\right| + \left|\widehat{\Omega g} (\xi)\right| + \left| \widehat{\Sigma g} (\xi)\right| \lesssim \left|\widehat{g}(\xi)\right| + \left|\widehat{\Omega g} (\xi)\right| + \left| \widehat{\mathcal{S} g} (\xi)\right| + t \left| \partial_t \widehat{g}(\xi)  \right|
\end{split}
\end{equation*}
entails the $L^2$ bound
$$
\left\|Y(\partial) \Lambda^{1/2} \Lambda x g \right\|_{H^1} 
\lesssim \|Y(\partial) \Lambda^{1/2} g \|_2 + \|Y(\partial) \Lambda^{1/2} g\|_{W^{0,2}_1} + t \|Y(\partial) \Lambda^{1/2} \partial_t g\|_2 
\lesssim (\|u\|_X + \|u\|_X^2) \<t\>^{\delta'}
$$
(the term involving the time derivative being bounded as in the proof of~(\ref{D3'}) below) and, by Sobolev embedding,
\begin{equation}
\label{tree3}
\left\|Y(\partial) e^{is\Lambda^{3/2}}\Lambda x g \right\|_{W^{1,4}} \lesssim \left\| Y(\partial) e^{is\Lambda^{3/2}} \Lambda^{1/2} \Lambda x g \right\|_{H^1}
 = \left\|  Y(\partial) \Lambda^{1/2} \Lambda x g \right\|_{H^1} \lesssim (\|u\|_X + \|u\|_X^2) \<t\>^{\delta'}.
\end{equation}
Plugging~(\ref{tree2}) and~(\ref{tree3}) in the last line of~(\ref{tree1}) gives
\begin{equation*}
\begin{split}
\left\|  (\ref{woodpecker1}) \right\|_2 \lesssim \int_1^t \frac{1}{s} \frac{\|u\|_X}{\sqrt{s}} \|u\|_X \<s\>^{\delta'} \,ds \lesssim \|u\|_X^2.
\end{split}
\end{equation*}
The term~(\ref{woodpecker2}) is bounded in a similar fashion. Finally, in order to deal with~(\ref{woodpecker3}), observe that 
$\partial_\eta \cdot \left[ \frac{m(\xi,\eta)\partial_\eta \phi_{+-}}{i s |\partial_\eta \phi_{+-}(\xi,\eta)|^2} \right]$ belongs to 
$\mathcal{M}^{1/2,1/2,0,0}$, thus we can use Corollary~\ref{Bili-prop1} to get
\begin{equation}
\begin{split}
& \left\| Y^2(\xi) |\xi|^\alpha (\ref{woodpecker3}) \right\|_2 = \left\| Y^2(\xi) |\xi|^\alpha \int_1^t \int e^{is\phi_{+-}(\xi,\eta)} \partial_\eta \cdot 
\left[ \frac{m(\xi,\eta)\partial_\eta \phi_{+-}}{i s |\partial_\eta \phi_{+-}(\xi,\eta)|^2} \right] \widehat{f}(s,\eta) \widehat{g}(s,\xi-\eta) \,d\eta\,ds \right\|_2 \\
& \qquad \lesssim \int_1^t \frac{1}{s} \left\| e^{is\Lambda^{3/2}}
T_{Y^2(\xi) |\xi|^\alpha \partial_\eta \cdot \left[ \frac{m(\xi,\eta)\partial_\eta \phi_{+-}}{i s |\partial_\eta \phi_{+-}(\xi,\eta)|^2}\right]}
\left( e^{is\Lambda^{3/2}} f(s)\,,\, e^{is\Lambda^{3/2}} g(s) \right) \,ds \right\|_2 \, ds \\
& \qquad \lesssim \int_1^t \frac{1}{s} \left[ \left\| Y(\partial)z \right\|_{W^{1,4}} \left\| \Lambda^{1/2} u \right\|_{W^{1,4}} 
+ \left\| Y(\partial) u \right\|_{W^{1,4}} \left\| \Lambda^{1/2} z \right\|_{W^{1,4}} \right] \,ds. \\
\end{split}
\end{equation}
By Sobolev embedding, and using the bound for $\| \Lambda^{1/2} u \|_{W^{1,4}}$ in~(\ref{tree2}), as well as a similar bound for 
$\| \Lambda^{1/2} z \|_{W^{1,4}}$, the above can be estimated by
\begin{equation*}
\begin{split}
\left\| Y^2(\xi) |\xi|^\alpha (\ref{woodpecker3}) \right\|_2 & \lesssim \int_1^t \frac{1}{s} \left[ \left\| Y(\partial) \Lambda^{1/2} g \right\|_{H^1} \left\| \Lambda^{1/2} u \right\|_{W^{1,4}} 
+ \left\| Y(\partial) \Lambda^{1/2} f \right\|_{H^1} \left\| \Lambda^{1/2} z \right\|_{W^{1,4}} \right] \,ds \\
& \lesssim \int_1^t \frac{1}{s} \frac{\|u\|_X}{\sqrt{s}} \|u\|_X \<s\>^{\delta'} \,ds \lesssim \|u\|_X^2.
\end{split}
\end{equation*}

\subsubsection{The $--$ case}

Upon partitioning the $(\xi,\eta)$ space into regions where $\phi$, respectively $\partial_\eta \phi$ do not vanish, one can proceed as in the $++$ case 
(if $\phi \neq 0$) or $--$ case (if $\partial_\eta \phi \neq 0$). We refer to 
\cite{GMS2} for a similar partitioning. Indeed, recall that  for 
$\xi = 2 \eta$, $\phi (2\eta, \eta) = ( 2^{3/2} - 2 ) |\eta|^{3/2} $  and hence 
$\mathscr{R} = \{ (0,0) \} $. We can partition the plan $\xi, \eta$ into  two regions 
$\RR^4 = \Omega_1 \cup \Omega_2$  such that 
  $ |\xi|^{3/2}+|\eta|^{3/2}  \lesssim      |\phi_{--}|  $ in the region $\Omega_1$
and $ |\xi|^{1/2}+|\eta|^{1/2}  \lesssim      |\partial_\eta \phi_{--}|  $ in the region $\Omega_2$. 

More precisely, we can take a smooth cut-off  $\rho : \RR_+ \to \RR_+$ such that 
$\rho(a) = 1$ for  $a \leq 1 $ and 
$\rho(a) = 0$ for  $a \geq 2 $ and then 
define
\begin{align}
&\chi^T ( \xi,\eta) = 
\rho(\frac{200 |\xi -2\eta|  }{|\xi|} )   \\ 
&\chi^S ( \xi,\eta) =  
 \Big( 1- \rho(\frac{200 |\xi -2\eta|  }{|\xi|} )   \Big). 
\end{align} 
   
Notice that on the support of $ \chi^T  $, we have $|\xi-2\eta| \leq \frac{|\xi|}{100}  $ and 
hence, $ \frac{99}{100} |\xi| \leq 2|\eta| \leq \frac{101}{100} |\xi|   $
and $|\xi-\eta| \leq |\eta| + \frac{|\xi|}{100} \leq  \frac{101}{99} |\eta|  $.  This yields 
$$  \phi_{--}  \geq [  \left(\frac{200}{101} \right)^{3/2} - 1 -  \left(\frac{101}{99} \right)^{3/2} ]  |\eta|^{3/2} 
  \geq c    \phi_{--} . $$
In a similar way, we have on the support of $\chi^S  $  that 
$|\xi|^{1/2}+|\eta|^{1/2}  \lesssim  |\partial_\eta \phi_{--} | $. 
   
The rest of the argument is similar to what was done before. 
We rewrite the Duhamel term in  \eqref{eqfourier} as the sum of 
two terms: 
 $$
Y^2(\xi) |\xi|^\alpha \int_0^t \int \chi(\xi,\eta)  e^{is\phi_{--}(\xi,\eta)} m(\xi,\eta) \widehat{f}(s,\eta) \widehat{g}(s,\xi-\eta) \,d\eta\,ds
$$
where $\chi = \chi^T $ or  $\chi = \chi^S $.  
For the term with  $ \chi^T  $, we integrate by parts in time and recover terms 
similar to \eqref{flamingo1}- \eqref{flamingo4}. We estimate them exactly as 
in the case of the phase $\phi_{++}$.    For the term with  $\chi^S  $, we integrate by 
parts in $\eta$, we get terms similar to  \eqref{woodpecker1}-\eqref{woodpecker3}.

\subsection{Proof of~(\ref{D3'})}

\label{phoebe}

We want to bound
$$
\left\| \partial^j \Gamma^k Y^2(\partial) \Sigma \Lambda^\alpha f \right\|_2,
$$
for some $j,k$ less than respectively $7$, $K+3$, which we fix from now on. First, since the commutators are controlled by~(\ref{D3}), it suffices to bound
$$
\left\| Y^2(\p) \Lambda^\alpha \partial^j \Gamma^k \Sigma f \right\|_2.
$$
Since $Y^2(\p)$, $\partial$, $\Gamma$ commute with the group $e^{it\Lambda^{3/2}}$, we infer from~(\ref{D3}) that
$$
\left\| Y^2(\p) \Lambda^\alpha \partial^j \Gamma^k \mathcal{S} f \right\|_2 \lesssim \epsilon + \|u\|_X^2.
$$
Since $\Sigma = \mathcal{S} - \frac{3}{2} t \partial_t$, we can bound
\begin{equation}
\begin{split}
\left\| Y^2(\partial) \Lambda^\alpha \p^j \Gamma^k \Sigma f \right\|_2 & \lesssim \left\|  Y^2(\p) \Lambda^\alpha \p^j \Gamma^k \mathcal{S} f \right\|_2
+ t \left\| Y^2(\p) \Lambda^\alpha \p^j \Gamma^k \partial_t f \right\|_2 \\
& \lesssim \epsilon + \|u\|_X^2 + t \left\| Y^2(\p) \p^j \Gamma^k \partial_t f \right\|_2.
\end{split}
\end{equation}
Thus it suffices to prove that
$$
\left\| Y^2(\p) \Lambda^\alpha \p^j \Gamma^k \partial_t f \right\|_2 \lesssim \frac{\|u\|_X^2}{t}.
$$
Going back to~(\ref{simpleform}), $f$ solves
$$
\partial_t f = e^{it\Lambda^{3/2}} T_m(u,u) + e^{it\Lambda^{3/2}} R.
$$
Applying $Y^2(\p) \Lambda^\alpha \p^j \Gamma^k$ gives
\begin{equation}
\label{eqdt}
Y^2(\p) \Lambda^\alpha \p^j \Gamma^k \partial_t f = e^{it\Lambda^{3/2}} Y^2(\p) \Lambda^\alpha \p^j \Gamma^k T_m(u,u) + e^{it\Lambda^{3/2}} Y^2(\p) \Lambda^\alpha \p^j \Gamma^k R.
\end{equation}
The remainder term $R$ is easily dealt with, thus we discard it. 

Next, we can compute $Y^2(\p) \Lambda^\alpha \p^j \Gamma^k T_m(u,u)$ by applying Leibniz' rule. This manipulation produces a great number of terms, 
but for simplicity we keep only two representative ones, namely (assuming for simplicity $j,k$ even)
$$
Y^2(\p) \Lambda^\alpha T_m(\p^j \Gamma^k u,u) \qquad  \mbox{and} \qquad Y^2(\p) \Lambda^\alpha T_m(\p^{j/2} \Gamma^{k/2} u,\p^{j/2} \Gamma^{k/2} u).
$$
Applying Corollary~\ref{Bili-prop1} gives
$$
\left\| Y^2(\p) \Lambda^\alpha T_m(\p^j \Gamma^k u,u) \right\|_2 \lesssim \left\|\Lambda^\alpha \p^j \Gamma^k u \right\|_{H^2} \left\| \Lambda^{3/4} u \right\|_{W^{2,\infty}}
\lesssim \frac{\|u\|_X^2}{t} 
$$
where the last inequality follows from~(\ref{D3}). Similarly, we find
\begin{equation*}
\begin{split}
\left\| Y^2(\p) \Lambda^\alpha T_m( \p^{j/2} \Gamma^{k/2} u,\p^{j/2} \Gamma^{k/2} u) \right\|_2 
& \lesssim \left\|\Lambda^\alpha  \partial^{j/2} \Gamma^{k/2} u \right\|_{H^2} \left\| \Lambda^{3/4} \p^{j/2} \Gamma^{k/2} u) \right\|_{W^{2,\infty}} \\
& \lesssim \frac{\|u\|_X^2}{t}.
\end{split}
\end{equation*}
Coming back to~(\ref{eqdt}), this yields
$$
\left\| Y^2(\p) \Lambda^\alpha \p^j \Gamma^k \partial_t f \right\|_2 \lesssim \frac{\|u\|_X^2}{t},
$$
thus giving the desired estimate.

\subsection{Proof of (\ref{D4})}

It follows from (\ref{D3'}), and Proposition~\ref{decay-prop}.

\subsection{Estimates for the remainder term $R$}

\label{remainderterm}

Recall that we wrote in~(\ref{simpleform}) the equation under the form
$$
\partial_t u =  - i \Lambda^{3/2} u + T_m(u,u) + R
$$
with
$$
R = \left[ \Lambda^{1/2} G(h) \psi + i \left( \Delta_h h + \frac{1}{2(1 + |\p h|^2)} \left( G(h) \psi + \p h \cdot \nabla \psi \right)^2 \right) \right]_3
$$
where $[ \cdot ]_3$ means the terms of order 3 and higher in the expansion of the expression between brackets. In proving the estimates~(\ref{D1}) to~(\ref{D4}), 
we always skipped the term $R$ to focus on the more difficult bilinear term $T_m(u,u)$. We now come back to the term $R$ and show how it can be estimated.
We will prove the a priori estimate
$$
\int_0^t \left\| Y(\partial)^3 \Lambda^\alpha \p^j \Gamma^k R \right\|_2 \,ds \lesssim \|u\|_X^3.
$$
for $|j|\leq 9$ and $|k| \leq K+8$. We fix such a $k$ from now on.
Instead of treating all the terms which appear in the above expression of $R$, we retain one, $\left[ \Lambda^{1/2} G(h) \psi  \right]_3$; this will alleviate the notations and the
other terms can be treated similarly. Thus we want to show that
$$ 
\int_0^t \left\| Y(\partial)^3 \Lambda^\alpha  \p^j \Gamma^k \Lambda^{1/2} \left[ G(h) \psi \right]_3 \right\|_2 \,ds \lesssim \|u\|_X^3.
$$
Commuting $\Lambda^{1/2}$ and $\Gamma^k$, interpolating, and applying Cauchy-Schwarz, we see that the above can be bounded up to lower order terms by terms of the form
$$
\left( \int_0^t \left\|\Gamma^k  \p^j \left[ G(h) \psi \right]_3  \right\|_2 \,ds \right)^{\theta}
\left( \int_0^t \left\| \nabla \Gamma^k  \p^j\left[ G(h) \psi \right]_3 \right\|_2^{1/2} \, ds \right)^{1-\theta} \;\;\;\; \mbox{with $0 \leq \theta \leq 1$}.
$$
Since the real problem is low frequencies, we focus on the first factor above. Expanding first $G(h)$ into multilinear operators as follows from 
Proposition~\ref{propexpansion}, and then applying Leibniz rule as described in Lemma~\ref{lemmasymmetries} gives that
$$
\Gamma^k \p^j \left[ G(h) \psi \right]_3 = \Gamma^k  \p^j \sum_{n \geq 2} M_n(h,\dots,h,\psi)
$$
is a sum of terms of the type
$$
\sum_{n \geq 2} \sum_{i_1 + \dots + i_{n+1} = k} M_n \left( \p^{j_1} \Gamma^{i_1} h,\dots, \p^{j_n} \Gamma^{i_n} h,
 \p^{j_{n+1}} \Gamma^{i_{n+1}}\psi \right).
$$
with $|i_1| + \dots + |i_{n+1}| \leq K+8$ and $|j_1| + \dots + |j_{n=1}| \leq 9$.
Using the bound given in Proposition~\ref{propexpansion}, each of these terms can be bounded by
\begin{equation*}
\begin{split}
& \int_0^t \left\| M_n \left(\p^{j_1} \Gamma^{i_1} h,\dots,\p^{j_n} \Gamma^{i_n} h,\p^{j_{n+1}} \Gamma^{i_{n+1}}\psi \right) \right\|_2 \, ds \\ 
& \qquad \qquad\lesssim C_*^n \int_0^t \left\|\p^{j_1+1}  \Gamma^{i_1} h \right\|_\infty \dots \left\|\p^{j_n+1} \Gamma^{i_n} h \right\|_\infty \left\|\p^{j_{n+1}+1}
\Gamma^{i_{n+1}} \psi \right\|_2 \,ds \\
& \qquad \qquad\lesssim C_*^n \int_0^t \left\| Y(\partial) \Lambda^{3/4}\p^{j_1} \Gamma^{i_1} u \right\|_\infty \dots  
\left\| Y(\partial) \Lambda^{3/4}\p^{j_n} \Gamma^{i_n} u \right\|_\infty \left\|\nabla \p^{j_{n+1}} \Gamma^{i_{n+1}} u \right\|_2 \,ds \\ 
& \qquad \qquad\lesssim C_*^n \|u\|_X^{n+1} \int_0^t \frac{1}{\< s \>^{1-\delta'}} 
\dots \frac{1}{\< s \>^{1-\delta'}} \<s\>^{\delta'} \,ds \\
\end{split}
\end{equation*}
Summing over $n$ gives the desired result since $\|u\|_X$ is small enough.

\section{Dispersive estimate for the linear group}

\label{sectionlindecay}

\begin{prop} \label{decay-prop}
For any $\beta \in [0,\frac{1}{2})$,
\begin{equation}
\label{dispersivestimate2}
\left\| e^{it\Lambda^{3/2}} f \right\|_\infty \lesssim t^{-1+\frac{2\beta}{3}} \sum_{j=0}^1
\sum_{k=0}^3 \left\| Y(\p) \Lambda^{\beta-1/2} \Sigma^j \Omega^k f \right\|_{L^2(\mathbb{R}^2)}.
\end{equation}
In particular,
\begin{equation}
\label{dispersivestimate1}
\left\| e^{it\Lambda^{3/2}} f \right\|_\infty \lesssim \frac{1}{t} 
\sum_{j=0}^1 \sum_{k=0}^3 \left\| Y(\p) \Lambda^{-1/2} \Sigma^j \Omega^k f \right\|_{L^2(\mathbb{R}^2)}
\end{equation}
\end{prop}

\begin{rem} These estimates are probably far from being optimal, but, omitting $Y(\partial)$, they are at the scaling of the equation. They should be understood as giving
the right decay if $f$ is smooth, and has a behaviour at $\infty$ in space (that is, at zero in Fourier) like an inverse power.

A similar estimate was proved by Klainerman~\cite{Klainerman}, see also H\"ormander~\cite{Hormander} for the linear group of the wave equation, and by Wu~\cite{Wu} 
for the group $e^{it\Lambda^{1/2}}$. These authors argued in physical space, whereas we rely here on a Fourier space approach. \end{rem}

After some preliminary steps, the proof reduces the problem to an oscillatory integral, which is then estimated.

\subsection{Preliminary steps}

\subsubsection{Why~(\ref{dispersivestimate2}) follows from~(\ref{dispersivestimate1})}

Split $\left\| e^{it\Lambda^{3/2}} f \right\|_\infty$ as follows, for $j_0 \leq 0$:
$$
\left\| e^{it\Lambda^{3/2}} f \right\|_\infty \leq \left\| P_{<j_0} e^{it\Lambda^{3/2}} f \right\|_\infty 
+ \left\| P_{\geq j_0} e^{it\Lambda^{3/2}} f \right\|_\infty .
$$
Bound the first piece directly, using successively the Hausdorff-Young and the Cauchy-Schwarz inequalities,
$$
\left\| P_{<j_0} e^{it\Lambda^{3/2}} f \right\|_\infty \lesssim \left\| \Theta \left( \frac{\xi}{2^{j_0}} \right) \widehat{f}(\xi) \right\|_1
\leq \left\| \Theta \left( \frac{\xi}{2^{j_0}} \right) |\xi|^{1/2-\beta} \right\|_2 \left\| |\xi|^{\beta-1/2} \widehat{f}(\xi) \right\|_2
\lesssim 2^{j_0 \left( \frac{3}{2} - \beta \right)} \left\| \Lambda^{\beta-1/2} f \right\|_2,
$$
and the second one using~(\ref{dispersivestimate1})
$$
\left\| P_{\geq j_0 } e^{it\Lambda^{3/2}} f \right\|_\infty 
\lesssim \frac{1}{t} \sum_{k=0}^3 \left\| P_{\geq j_0 } Y(\p) \Lambda^{-1/2} \Sigma \Omega^k f \right\|_{2}
\lesssim \frac{1}{t} 2^{-\beta j_0} \sum_{k=0}^3 \left\| Y(\p) \Lambda^{\beta - 1/2} \Sigma \Omega^k f \right\|_{2}.
$$
Optimizing over $j_0$ gives the desired result.

\subsubsection{A Hardy-type estimate}

We denote the polar coordinates in Fourier space by $(\rho,\theta)$. A function $\widehat{f}(\rho)$ is understood to depend only on the radial variable. We will denote
$$
A(\rho) \overset{def}{=} \left\{ \begin{array}{l} \rho^{1/2-\epsilon}\quad \mbox{if $\rho \leq 1$} \\ \rho^{1/2+\epsilon}\quad \mbox{if $\rho \geq 1$}. \end{array} \right.
$$
Then one can estimate by Cauchy Schwarz
$$
|\widehat{f}(\rho)| = \left| \int_\rho^\infty \partial_\rho \widehat{f}(\sigma) \,d\sigma \right| 
\leq \left( \int_{\rho}^\infty \frac{d\sigma}{\sigma^2} \right)^{1/2} \left( \int_\rho^\infty \left| \partial_\rho \widehat{f}(\sigma) \right|^2 \sigma^2\,d\sigma \right)^{1/2} 
\lesssim \frac{1}{\sqrt{\rho}} \|\rho^{1/2} \partial_\rho  \widehat{f}\|_{L^2(\mathbb{R}^2)}.
$$
We record the result
\begin{equation}
\label{boundf}
|\widehat{f}(\rho)| \lesssim \frac{1}{\sqrt{\rho}} \|\rho^{1/2} \partial_\rho \widehat{f}\|_{L^2(\mathbb{R}^2)}.
\end{equation}
By a trivial modification of the above one can prove the slightly more precise bound
\begin{equation}
\label{boundfprecised}
|\widehat{f}(\rho)| \lesssim \frac{1}{A(\rho)} \|A(\rho) \partial_\rho f\|_{L^2(\mathbb{R}^2)} .
\end{equation}

\subsubsection{Estimates on Bessel functions}

\label{subsubbessel}

Recall that the $m$-th Bessel function is given by
$$
J_m(s) = \int_{\mathbb{S}^1} e^{i(s \cos(\theta)+m \theta)}\,d\theta.
$$
It is well-known (see Stein~\cite{Stein}, Chapter VIII) that
$$
J_m(s) = e^{is} g^m_1(s) + e^{-is} g^m_2(s) \quad \mbox{with} \quad \left| \left( \frac{d}{ds} \right)^k g^m_i(s) \right| \leq C(m,k) \frac{1}{\<s\>^{1/2+k}}.
$$
In order to bound the constant $C(m,k)$, we will need the following lemma

\begin{lem}
Assume that $f$ is a fixed smooth function such that $f(0) = 0$, $f'(0) = 0$, $f'(x) \neq 0$ if $x\neq 0$, and $f''(0) = 1$. Let 
$$
I(s) := \int_{\mathbb{R}} u(\theta) e^{is f(\theta)} \,d\theta
$$
where $u$ is a smooth function with support in $[-1,1]$. Then
$$
\left| \left( \frac{d}{ds} \right)^k I(s) \right| \lesssim \frac{1}{\< s \>^{1/2+k}} \left( \|u\|_\infty + \|u'\|_\infty + \|u''\|_\infty \right) \quad \mbox{if $k=0,1$}.
$$
\end{lem}
In the definition of $J_m$, the phase $\cos(\theta)$ has two stationary points, $0$ and $\pi$, both of which are non degenerate. 
Successively restricting attention to one of them by an appropriate cut-off, we see that the above lemma implies
\begin{cor}
The constant $C(m,0)$ and $C(m,1)$ can be bounded by
$$
C(m,0) + C(m,1) \lesssim m^2.
$$
\end{cor}

\begin{proof} Choosing a compactly supported, smooth cut-off function, equal to one in a neighbourhood of zero, write
$$
I(s) = \int ( u(0) + \theta u'(0)) \chi(\theta) e^{isf(\theta)} \,d\theta + \int \left[ u(\theta) - ( u(0) + \theta u'(0)) \chi(\theta) \right] e^{isf(\theta)} \,d\theta := I_1(s) + I_2(s).
$$
The term $I_1(s)$ is easily estimated by the standard stationary phase lemma. We are thus left with $I_2(s)$. In order to estimate it, it suffices to show that, under the assumptions of the lemma, and if furthermore $u(0) = u'(0) = 0$, then
$$
\left| \left(\frac{d}{ds} \right)^k I(s) \right| \lesssim \frac{1}{s^{k+1}}\left( \|u\|_\infty + \|u'\|_\infty + \|u''\|_\infty \right) 
$$
(this estimate is even stronger than needed). We prove this inequality for $k=0$, the case $k=1$ being similar. The idea is of course to integrate by parts, to get
$$
\int u(\theta) e^{is f(\theta)} \,d\theta = \frac{i}{s} \int \frac{u'(s)}{f'(s)} e^{isf(\theta)}\,d\theta - \frac{i}{s} \int \frac{u(s) f''(s)}{f'(s)^2} e^{is f(\theta)} \,d\theta,
$$
which can easily be estimated by
$$
\left| I(s) \right| \lesssim \frac{1}{s} \left(\|u\|_\infty + \|u'\|_\infty + \|u''\|_\infty \right).
$$
This concludes the proof of the lemma.
\end{proof}

\subsection{The oscillatory integral point of view}

\subsubsection{Reduction of the proposition to an oscillatory integral estimate}

Expand the function $f$ in the statement of the proposition in spherical (or rather circular) harmonics:
$$
\widehat{f}(\xi) = \sum_{m \in \mathbb{Z}} \widehat{f_m}(\rho) e^{im\theta}.
$$
Suppose $x$, a point in physical space, has the polar coordinates $(r,\theta_0)$. Then
$$
\left[ e^{it\Lambda^{3/2}} f \right] (x)= \sum_m e^{im\theta_0} \int \int e^{i(r\rho \cos(\theta)+m \theta)} e^{it \rho^{3/2}} \widehat{f_m}(\rho) \rho \,d\rho\,d\theta = \sum_m e^{im\theta_0} \int J_m(r \rho) e^{it \rho^{3/2}} \widehat{f_m}(\rho) \rho \,d\rho
$$
where $J_m$ is the $m$-th Bessel function. By the results of Section~\ref{subsubbessel}, the problem reduces to estimating an integral of the type 
$$
\int e^{i(t\rho^{3/2} - r\rho)} g(r\rho)\widehat{f}(\rho) \rho\,d\rho
$$
with
\begin{equation}
\label{boundg}
\left| \left( \frac{d}{ds} \right)^k g(s) \right| \lesssim \frac{1}{\< s \>^{1/2+k}} \quad \mbox{for $k=0,1$}
\end{equation}
(the other integral, with phase $t\rho^{3/2} + r\rho$ being easier, we skip it). Denoting
$$
R \overset{def}{=}\frac{r}{t} \quad \mbox{and} \quad \phi(\rho) = \rho^{3/2} - R \rho,
$$
the integral becomes
$$
\int e^{i(t\phi(\rho))} g(R \rho t)\widehat{f}(\rho) \rho\,d\rho.
$$
Our aim will be to prove the
\begin{claim}
\label{clay}
If $g$ satisfies~(\ref{boundg}), then with an implicit constant independent of $R$,
$$
\left| \int e^{it\phi(\rho)} g(R \rho t)\widehat{f}(\rho) \rho\,d\rho \right| \lesssim \frac{1}{t} \left\| A(\rho) \partial_\rho \widehat{f}(\rho) \right\|_{L^2(\mathbb{R}^2)}.
$$
\end{claim}

\subsubsection{How Proposition~\ref{decay-prop} follows from Claim~\ref{clay}}
The proposition follows easily from the claim. Indeed, using in addition the results of Section~\ref{subsubbessel}, it implies that
\begin{equation}
\begin{split}
\left\| e^{it\Lambda^{3/2}} f \right\|_\infty & \lesssim \frac{1}{t} \sum_{m} m^2 \left\| A(\rho) \partial_\rho \widehat{f_m}(\rho) \right\|_2 \\
& \lesssim \frac{1}{t} \left( \sum_m m^6 \left\| A(\rho) \partial_\rho \widehat{f_m}(\rho) \right\|_2^2 \right)^{1/2} \\
& \lesssim \frac{1}{t} \sum_{j=0}^1 \sum_{k=0}^3 \left\| Y(\partial) \Lambda^{-1/2} \Sigma^j \Omega^k f \right\|_{L^2(\mathbb{R}^2)}.
\end{split}
\end{equation}

\subsubsection{Estimates on $\phi$}

Estimates on $\phi$ will be needed, we will record them below. First notice that $\phi'$ vanishes for $\rho = \frac{4}{9}R^2$. This motivates the introduction of the new variable
$$
h \overset{def}{=} \rho - \frac{4}{9}R^2.
$$
We start with estimates on $\phi'$:
$$
\phi'(\rho) \sim 
\left\{ \begin{array}{ll}
R & \mbox{if $\rho < \frac{2 R^2}{9}$} \\ 
\frac{h}{R} & \mbox{if $\frac{2}{9}R^2 \leq \rho \leq 10 R^2$} \\
\sqrt{h} \sim \sqrt{\rho} & \mbox{if $\rho \geq 10R^2$}. 
\end{array} \right.
$$
Estimates on $\partial_\rho \frac{1}{\phi'}$ will also be necessary:
$$
\left| \partial_\rho \frac{1}{\phi'} \right| \sim
\left\{ \begin{array}{ll} 
\frac{1}{R^2 \sqrt{\rho}} & \mbox{if $\rho < \frac{2 R^2}{9}$} \\
\frac{R}{h^2} & \mbox{if $\frac{2}{9}R^2 \leq \rho \leq 10 R^2$} \\ 
\frac{1}{\rho^{3/2}} & \mbox{if $\rho \geq 10R^2$}. 
\end{array} \right.
$$

Aside from direct estimates, the basic tool will be integration by parts using the identity:
\begin{equation}
\label{ibp}
\frac{1}{it\phi'(\rho)} \partial_\rho e^{it\phi(\rho)}=e^{it\phi(\rho)}.
\end{equation}

\subsection{Proof of the claim~(\ref{clay}), the case $\rho<\frac{1}{Rt}$}

We assume from now on for simplicity that
$$
\left\| A(\rho) \partial_\rho \widehat{f}(\rho) \right\|_2 = 1.
$$
We estimate here
$$
\int_0^{\frac{1}{Rt}} e^{it\phi(\rho)} g(R \rho t)\widehat{f}(\rho) \rho\,d\rho.
$$

\subsubsection{The case $R^3 \gtrsim \frac{1}{t}$}

Using~(\ref{boundg}),~(\ref{boundf}), and the assumption $R^3 \gtrsim \frac{1}{t}$, it follows easily that
$$
\left|\int_0^{\frac{1}{Rt}} e^{it\phi(\rho)} g(R \rho t)\widehat{f}(\rho) \rho\,d\rho\right| \lesssim \int_0^{\frac{1}{Rt}} \sqrt{\rho} \,d\rho \lesssim \frac{1}{(Rt)^{3/2}} \lesssim \frac{1}{t}.
$$

\subsubsection{The case $R^3 < \frac{1}{100 t}$}

Split then
$$
\int_0^{\frac{1}{Rt}} e^{i(t\phi(\rho))} g(R \rho t)\widehat{f}(\rho) \rho\,d\rho = \int_0^{\frac{1}{t^{2/3}}} + \int_{\frac{1}{t^{2/3}}}^{\frac{1}{Rt}} \dots \overset{def}{=} I + II
$$

\bigskip

\noindent \underline{Estimate for $I$} The term $I$ can be estimated with the help of~(\ref{boundg}),~(\ref{boundf}):
$$
\left| I \right| \lesssim \int_0^{\frac{1}{t^{2/3}}} \sqrt{\rho} \,d\rho \lesssim \frac{1}{t}.
$$ 
\bigskip

\noindent \underline{Estimate for $II$} 
Integration by parts with the help of (\ref{ibp}) gives
\begin{subequations}
\begin{align}
\label{hummingbird1}
II = & - e^{it\phi(t^{-2/3})} \frac{1}{it\phi'\left( t^{-2/3} \right)} g(Rtt^{-2/3}) \widehat{f}(t^{-2/3}) t^{-2/3} \\
\label{hummingbird2}
& + e^{it\phi\left(\frac{1}{Rt}\right)} \frac{1}{it\phi'\left( \frac{1}{Rt} \right)}g\left(Rt \frac{1}{Rt}\right) \widehat{f}\left(\frac{1}{Rt}\right) \frac{1}{Rt} \\
\label{hummingbird3}
& - \int_{\frac{1}{t^{2/3}}}^{\frac{1}{Rt}} e^{it\phi} \frac{1}{it\phi'} Rt g'(Rt\rho) \widehat{f}(\rho) \rho\,d\rho \\
\label{hummingbird4}
& - \int_{\frac{1}{t^{2/3}}}^{\frac{1}{Rt}} e^{it\phi} \frac{1}{it\phi'} g(Rt\rho) \partial_\rho \widehat{f}(\rho) \rho\,d\rho \\
\label{hummingbird5}
& - \int_{\frac{1}{t^{2/3}}}^{\frac{1}{Rt}} e^{it\phi} \frac{1}{it\phi'} g(Rt\rho) \widehat{f}(\rho) \,d\rho \\
\label{hummingbird6}
& - \int_{\frac{1}{t^{2/3}}}^{\frac{1}{Rt}} e^{it\phi} \partial_\rho\left( \frac{1}{it\phi'} \right) g(Rt\rho) \widehat{f}(\rho) \rho\,d\rho.
\end{align}
\end{subequations}
Observe that on the integration domain of $II$, $\phi'(\rho) \sim \sqrt{\rho}$, $\partial_\rho\left( \frac{1}{\phi'(\rho)} \right) \sim\frac{1}{\rho^{3/2}}$, and $g,g'\sim 1$. Using in addition~(\ref{boundg}),~(\ref{boundf}),~(\ref{boundfprecised}) and the Cauchy-Schwarz inequality, this leads to the estimates:
\begin{equation}
\begin{split}
|(\ref{hummingbird1})| & \lesssim \frac{1}{t} \frac{1}{\sqrt{t^{-2/3}}} \frac{1}{\sqrt{t^{-2/3}}}  t^{-2/3} = \frac{1}{t} \\
|(\ref{hummingbird2})| & \lesssim \frac{1}{t} \sqrt{Rt} \sqrt{Rt} \frac{1}{Rt} = \frac{1}{t} \\
|(\ref{hummingbird3})| & \lesssim \int_{\frac{1}{t^{2/3}}}^{\frac{1}{Rt}} \frac{1}{t\sqrt{\rho}} R t\frac{1}{\sqrt{\rho}} \rho \,d\rho = R \int_{\frac{1}{t^{2/3}}}^{\frac{1}{Rt}}d\rho \leq \frac{1}{t} \\
|(\ref{hummingbird4})| & \lesssim \int_{\frac{1}{t^{2/3}}}^{\frac{1}{Rt}} \frac{1}{t\sqrt{\rho}} \rho \partial_\rho \widehat{f}(\rho) d\rho \leq \frac{1}{t} \left(\int \rho A^2 \partial_\rho \widehat{f}(\rho)^2 \,d\rho\right)^{1/2} \left(\int A^{-2} \,d\rho\right)^{1/2} \lesssim \frac{1}{t} \\
|(\ref{hummingbird6})| & \lesssim \int_{\frac{1}{t^{2/3}}}^{\frac{1}{Rt}} \frac{1}{t} \frac{1}{\rho^{3/2}} \frac{1}{A(\rho)} \,d\rho \lesssim \frac{1}{t}.
\end{split}
\end{equation}
(the term~(\ref{hummingbird5}) can be treated like~(\ref{hummingbird4}), thus we skipped it).

\subsection{Proof of the claim~(\ref{clay}), the case $\rho>\frac{1}{Rt}$}

We estimate here
$$
\int_{\frac{1}{Rt}}^\infty e^{it\phi(\rho)} g(R \rho t)\widehat{f}(\rho) \rho\,d\rho.
$$

\subsubsection{The case $R^3 \gtrsim \frac{1}{t}$}

Split the integral as follows
\begin{equation*}
\begin{split}
\int_{\frac{1}{Rt}}^\infty e^{i(t\phi(\rho))} g(R \rho t)\widehat{f}(\rho) \rho\,d\rho & = \int_{1/Rt}^{R^2/10} + \int_{R^2/10}^{\frac{4}{9}R^2 - \sqrt{\frac{R}{t}}} + \int_{\frac{4}{9}R^2 - \sqrt{\frac{R}{t}}}^{\frac{4}{9}R^2 + \sqrt{\frac{R}{t}}} + \int_{\frac{4}{9}R^2 + \sqrt{\frac{R}{t}}}^{10R^2} + \int_{10R^2}^\infty \dots \\
& \overset{def}{=} I + II + III + IV + V.
\end{split}
\end{equation*}

\bigskip
\noindent \underline{Estimate for $I$} Integration by parts with the help of (\ref{ibp}) gives
\begin{subequations}
\begin{align}
\label{baldeagle1}
I = & - e^{it\phi\left(\frac{1}{Rt}\right)} \frac{1}{it\phi'\left( \frac{1}{Rt} \right)}g\left(Rt \frac{1}{Rt}\right) \widehat{f}\left(\frac{1}{Rt}\right) \frac{1}{Rt} \\
\label{baldeagle2}
& + e^{it\phi\left(\frac{R^2}{10}\right)} \frac{1}{it\phi'\left( \frac{R^2}{10} \right)}g\left(Rt \frac{R^2}{10}\right) \widehat{f}\left(\frac{R^2}{10}\right) \frac{1}{Rt} \\
\label{baldeagle3}
& - \int_{\frac{1}{Rt}}^{\frac{R^2}{10}} e^{it\phi} \frac{1}{it\phi'} Rt g'(Rt\rho) \widehat{f}(\rho) \rho\,d\rho \\
\label{baldeagle4}
& - \int_{\frac{1}{Rt}}^{\frac{R^2}{10}} e^{it\phi} \frac{1}{it\phi'} g(Rt\rho) \partial_\rho \widehat{f}(\rho) \rho\,d\rho \\
\label{baldeagle5}
& - \int_{\frac{1}{Rt}}^{\frac{R^2}{10}} e^{it\phi} \frac{1}{it\phi'} g(Rt\rho) \widehat{f}(\rho) \,d\rho \\
\label{baldeagle6}
& - \int_{\frac{1}{Rt}}^{\frac{R^2}{10}} e^{it\phi} \partial_\rho\left( \frac{1}{it\phi'} \right) g(Rt\rho) \widehat{f}(\rho) \rho\,d\rho.
\end{align}
\end{subequations}
Observe that on the integration domain of $II$, $\phi'(\rho) \sim R$, $\partial_\rho\left( \frac{1}{\phi'(\rho)} \right) \sim\frac{1}{R^2 \sqrt{\rho}}$, $g(s)\sim\frac{1}{\sqrt{s}}$, $g'(s)\sim \frac{1}{s^{3/2}}$. Using in addition~(\ref{boundg}),~(\ref{boundf}),~(\ref{boundfprecised}) and the Cauchy-Schwarz inequality, this leads to the estimates:
\begin{equation}
\begin{split}
|(\ref{baldeagle1})| & \lesssim \frac{1}{Rt} \sqrt{Rt} \frac{1}{Rt} \lesssim \frac{1}{(Rt)^{3/2}} \lesssim \frac{1}{t} \\
|(\ref{baldeagle2})| & \lesssim \frac{1}{Rt} \frac{1}{\sqrt{tR^3}} \frac{1}{R} \frac{1}{Rt} = \frac{1}{R^{9/2} t^{5/2}} \lesssim \frac{1}{t} \\
|(\ref{baldeagle3})| & \lesssim \int_{\frac{1}{Rt}}^{\frac{R^2}{10}} \frac{1}{tR} Rt \frac{1}{(Rt\rho)^{3/2}} \widehat{f}(\rho)\rho\,d\rho \lesssim \frac{1}{(Rt)^{3/2}} \lesssim \frac{1}{t} \\
|(\ref{baldeagle4})| & \lesssim \int_{\frac{1}{Rt}}^{\frac{R^2}{10}} \frac{1}{tR} \frac{1}{\sqrt{Rt\rho}} |\partial_\rho \widehat{f}(\rho)|\rho\,d\rho = \frac{1}{(Rt)^{3/2}} \int \rho^{1/2} |\partial_\rho \widehat{f}(\rho\,d\rho)| \lesssim \frac{1}{t} \\
|(\ref{baldeagle6})| & \lesssim \int_{\frac{1}{Rt}}^{\frac{R^2}{10}} \frac{1}{tR^2 \sqrt{\rho}} \frac{1}{\sqrt{\rho}} \frac{1}{\sqrt{Rt\rho}}\rho\,d\rho \lesssim \frac{1}{(Rt)^{3/2}} \lesssim \frac{1}{t}.
\end{split}
\end{equation}

\bigskip
\noindent \underline{Estimate for $II$} Integration by parts with the help of (\ref{ibp}) gives
\begin{subequations}
\begin{align}
\label{cardinal1}
I = & - e^{it\phi\left(\frac{R^2}{10}\right)} \frac{1}{it\phi'\left( \frac{R^2}{10} \right)}g\left(Rt \frac{R^2}{10}\right) \widehat{f}\left(\frac{R^2}{10}\right) \frac{1}{Rt} \\
\label{cardinal2}
& + e^{it\phi\left(\frac{4}{9}R^2 - \sqrt{\frac{R}{t}}\right)} \frac{1}{it\phi'\left( \frac{4}{9}R^2 - \sqrt{\frac{R}{t}} \right)}g\left(Rt \left[ \frac{4}{9}R^2 - \sqrt{\frac{R}{t}}\right] \right) \widehat{f}\left(\frac{4}{9}R^2 - \sqrt{\frac{R}{t}}\right) \left[ \frac{4}{9}R^2 - \sqrt{\frac{R}{t}} \right] \\
\label{cardinal3}
& - \int_{\frac{R^2}{10}}^{\frac{4}{9}R^2 - \sqrt{\frac{R}{t}}} e^{it\phi} \frac{1}{it\phi'} Rt g'(Rt\rho) \widehat{f}(\rho) \rho\,d\rho \\
\label{cardinal4}
& - \int_{\frac{R^2}{10}}^{\frac{4}{9}R^2 - \sqrt{\frac{R}{t}}} e^{it\phi} \frac{1}{it\phi'} g(Rt\rho) \partial_\rho \widehat{f}(\rho) \rho\,d\rho \\
\label{cardinal5}
& - \int_{\frac{R^2}{10}}^{\frac{4}{9}R^2 - \sqrt{\frac{R}{t}}} e^{it\phi} \frac{1}{it\phi'} g(Rt\rho) \widehat{f}(\rho) \,d\rho \\
\label{cardinal6}
& - \int_{\frac{R^2}{10}}^{\frac{4}{9}R^2 - \sqrt{\frac{R}{t}}} e^{it\phi} \partial_\rho\left( \frac{1}{it\phi'} \right) g(Rt\rho) \widehat{f}(\rho) \rho\,d\rho.
\end{align}
\end{subequations}
Observe that on the integration domain of $II$, $\phi'(\rho) \sim \frac{h}{R}$, $\partial_\rho\left( \frac{1}{\phi'(\rho)} \right) \sim\frac{R}{h^2}$, $g(s)\sim\frac{1}{\sqrt{s}}$, $g'(s)\sim \frac{1}{s^{3/2}}$. Using in addition~(\ref{boundg}),~(\ref{boundf}),~(\ref{boundfprecised}) and the Cauchy-Schwarz inequality, this leads to the estimates:
\begin{equation}
\begin{split}
|(\ref{cardinal2})| & \lesssim \frac{1}{t} \frac{R}{\sqrt{R/t}} \frac{1}{\sqrt{tR^3}} \frac{1}{R} R = \frac{1}{t} \\
|(\ref{cardinal3})| & \lesssim \int_{\frac{R^2}{10}}^{\frac{4}{9}R^2 - \sqrt{\frac{R}{t}}} \frac{R}{th}Rt\frac{1}{(Rt\rho)^{3/2}} \frac{1}{\sqrt{\rho}} \rho\,d\rho \lesssim \frac{1}{t} \frac{1}{\sqrt{tR^3}} \log(tR^3) \lesssim \frac{1}{t} \\
|(\ref{cardinal4})| & \lesssim \int_{-\frac{31R^2}{90}}^{- \sqrt{\frac{R}{t}}} \frac{R}{th} \frac{1}{\sqrt{Rt\rho}} \rho |\partial_\rho \widehat{f}(\rho)|\,dh \lesssim \frac{1}{\sqrt{Rt}} \frac{1}{t} \left(\int_{\sqrt{R/t}}^\infty \frac{dh}{h^2} \right)^{1/2} \lesssim \frac{1}{t} \frac{1}{(R^3t)^{1/4}} \lesssim \frac{1}{t} \\
|(\ref{cardinal6})| & \lesssim \int_{-\frac{31R^2}{90}}^{- \sqrt{\frac{R}{t}}} \frac{R}{t h^2} \frac{1}{\sqrt{tR^3}} \frac{1}{\sqrt{\rho}} \rho\,dh \lesssim \frac{\sqrt{R}}{t^{3/2}} \int_{\sqrt{R/t}}^\infty \frac{dh}{h^2} \lesssim \frac{1}{t}.
\end{split}
\end{equation}
(notice that~(\ref{cardinal1}) was already estimated for $I$, and that~(\ref{cardinal5}) can be estimated like~(\ref{cardinal6}).

\bigskip
\noindent \underline{Estimate for $III$} This term can be estimated directly, using simply~(\ref{boundg}) and~(\ref{boundf}):
$$
|III| \lesssim \int_{\frac{4}{9}R^2 - \sqrt{\frac{R}{t}}}^{\frac{4}{9}R^2 + \sqrt{\frac{R}{t}}} \frac{1}{\sqrt{Rt\rho}} \frac{1}{\sqrt{\rho}} \rho d\rho \lesssim \frac{1}{t}.
$$

\bigskip
\noindent \underline{Estimate for $IV$} Identical to the estimate for $II$!

\bigskip
\noindent \underline{Estimate for $V$} Integration by parts with the help of (\ref{ibp}) gives
\begin{subequations}
\begin{align}
\label{bluejay1}
V = & - e^{it\phi\left(10R^2\right)} \frac{1}{it\phi'\left(10R^2 \right)}g\left(Rt 10R^2\right) \widehat{f}\left(10R^2\right) 10 R^2 \\
\label{bluejay2}
& - \int_{10R^2}^{\infty} e^{it\phi} \frac{1}{it\phi'} Rt g'(Rt\rho) \widehat{f}(\rho) \rho\,d\rho \\
\label{bluejay3}
& - \int_{10R^2}^{\infty} e^{it\phi} \frac{1}{it\phi'} g(Rt\rho) \partial_\rho \widehat{f}(\rho) \rho\,d\rho \\
\label{bluejay4}
& - \int_{10R^2}^{\infty} e^{it\phi} \frac{1}{it\phi'} g(Rt\rho) \widehat{f}(\rho) \,d\rho \\
\label{bluejay5}
& - \int_{10R^2}^{\infty} e^{it\phi} \partial_\rho\left( \frac{1}{it\phi'} \right) g(Rt\rho) \widehat{f}(\rho) \rho\,d\rho.
\end{align}
\end{subequations}
Observe that on the integration domain of $V$, $\phi'(\rho) \sim \sqrt{\rho} $, $\partial_\rho\left( \frac{1}{\phi'(\rho)} \right) \sim \frac{1}{\rho^{3/2}}$, $g(s)\sim\frac{1}{\sqrt{s}}$, $g'(s)\sim \frac{1}{s^{3/2}}$. Using in addition~(\ref{boundg}),~(\ref{boundf}),~(\ref{boundfprecised}) and the Cauchy-Schwarz inequality, this leads to the estimates:
\begin{equation}
\begin{split}
|(\ref{bluejay2})| & \lesssim \int_{10R^2}^\infty \frac{1}{t\sqrt{\rho}} tR \frac{1}{(tR\rho)^{3/2}} \frac{1}{\sqrt{\rho}}\rho\,d\rho \lesssim \frac{1}{t} \frac{1}{\sqrt{tR^3}} \lesssim \frac{1}{t} \\
|(\ref{bluejay3})| & \lesssim \int_{10R^2}^\infty \frac{1}{t\sqrt{\rho}} \frac{1}{\sqrt{Rt\rho}} |\partial_\rho \widehat{f}(\rho)|\rho\,d\rho \lesssim \frac{1}{t\sqrt{Rt}} \left( \int_{R^2}^\infty \frac{d\rho}{\rho^2} \right)^{1/2} \lesssim \frac{1}{t} \frac{1}{\sqrt{tR^3}} \lesssim \frac{1}{t} \\
|(\ref{bluejay5})| & \lesssim \int_{10R^2}^\infty \frac{1}{t} \frac{1}{\rho^{3/2}} \frac{1}{\sqrt{Rt\rho}} \frac{1}{\sqrt{\rho}} \rho\,d\rho \lesssim \frac{1}{t} \frac{1}{\sqrt{tR^3}} \lesssim \frac{1}{t}
\end{split}
\end{equation}
(the term~(\ref{bluejay1}) was already estimated, and~(\ref{bluejay4}) can be estimated like~(\ref{bluejay5})).
\subsubsection{The case $R^3 < \frac{1}{100 t}$}
Integration by parts with the help of (\ref{ibp}) gives
\begin{subequations}
\begin{align}
\label{toucan1}
\int_{\frac{1}{Rt}}^\infty e^{i(t\phi(\rho))} g(R \rho t)\widehat{f}(\rho) \rho\,d\rho = & - e^{it\phi\left(\frac{1}{Rt}\right)} \frac{1}{it\phi'\left( \frac{1}{Rt} \right)}g\left(Rt \frac{1}{Rt}\right) \widehat{f}\left(\frac{1}{Rt}\right) \frac{1}{Rt} \\
\label{toucan2}
& - \int_{\frac{1}{Rt}}^\infty e^{it\phi} \frac{1}{it\phi'} Rt g'(Rt\rho) \widehat{f}(\rho) \rho\,d\rho \\
\label{toucan3}
& - \int_{\frac{1}{Rt}}^\infty e^{it\phi} \frac{1}{it\phi'} g(Rt\rho) \partial_\rho \widehat{f}(\rho) \rho\,d\rho \\
\label{toucan4}
& - \int_{\frac{1}{Rt}}^\infty e^{it\phi} \frac{1}{it\phi'} g(Rt\rho) \widehat{f}(\rho) \,d\rho \\
\label{toucan5}
& - \int_{\frac{1}{Rt}}^\infty e^{it\phi} \partial_\rho\left( \frac{1}{it\phi'} \right) g(Rt\rho) \widehat{f}(\rho) \rho\,d\rho
\end{align}
\end{subequations}
Observe that on the domain of the above integral, $\phi' \sim \sqrt{\rho}$ and $\partial_\rho\left( \frac{1}{\phi'(\rho)} \right) \sim\frac{1}{\rho^{3/2}}$. Using in addition~(\ref{boundg}),~(\ref{boundf}),~(\ref{boundfprecised}) and the Cauchy-Schwarz inequality, this leads to the estimates:
\begin{equation*}
\begin{split}
|(\ref{toucan1})| & \lesssim \frac{1}{t} \sqrt{Rt} \sqrt{Rt} \frac{1}{Rt} \frac{1}{t} \\
|(\ref{toucan2})| & \lesssim \int_{\frac{1}{Rt}}^\infty \frac{1}{t\sqrt{\rho}}  Rt \frac{1}{(Rt\rho)^{3/2}}\frac{1}{\sqrt{\rho}}\rho\,d\rho \lesssim \frac{1}{t} \\
|(\ref{toucan3})| & \lesssim \int_{\frac{1}{Rt}}^\infty \frac{1}{t\sqrt{\rho}} \frac{1}{\sqrt{R t \rho}} \partial_\rho \widehat{f}(\rho) \rho\,d\rho \leq \frac{1}{\sqrt{R}t^{3/2}} \left( \int_{1/Rt}^\infty \rho^2 \left| \partial_\rho \widehat{f}(\rho) \right|^2 \,d\rho \right)^{1/2}\left( \int_{1/Rt}^\infty \frac{d\rho}{\rho^2} \right)^{1/2} \lesssim \frac{1}{t} \\
|(\ref{toucan5})| & \lesssim \int_{\frac{1}{Rt}}^\infty \frac{1}{t\rho^{3/2}} \frac{1}{\sqrt{Rt\rho}} \frac{1}{\sqrt{\rho}}\rho\,d\rho \lesssim \frac{1}{t}
\end{split}
\end{equation*}
(the term~(\ref{toucan4}) can be estimated similarly to~(\ref{toucan3}), thus we skipped it).

\section{Traces and elliptic estimates }
\label{appendixtaee}

In this appendix we present bounds on harmonic extensions and traces of functions defined  in a domain  $\mathscr{D}\subset \mathbb{R}^3$.  We assume that $\mathscr{D}$ is bounded by the graph of a smooth function $h$ defined on  $\mathbb{R}^2$ such that
\be\label{A2}
\tag{A1}
\|\Lambda^{\half +\alpha}  h\|_{H^{2s - \half -\alpha }}< \infty,  \qquad  \|\p h\|_{H^{s+3}}    \ll 1\quad \text{for $s\gg 1$ }, \quad 0 < \alpha < \half.
\ee
This implies that $h\in L^p$ for some $p<\8$.
For a function $u$ defined on $\mathscr{D}$,
\[
\mathscr{D}= \{(x,z) ; \;x\in\RR^2,  z \le h(x)\},
\]
we denote by  $u_b(x)= u(x,h(x))$  the trace of  $u$ on the boundary $\SCB =\p \mathscr{D}$, and for a function $\varphi$ defined on $\SCB$ we denote 
by $\varphi_\mathcal{H}$ its harmonic extension to $\mathscr{D}$. Finally we denote  by $\Delta_0^{-1}g$ the solution to the Dirichlet problem on $\mathscr{D}$
with $0$ boundary value
\[
\Delta u = g \quad \text{on  $\mathscr{D}$}, \quad u_b= 0.
\]
The surface $\SCB$ is a Riemannian submanifold of $\RR^3$ with covariant derivative  given by
\begin{equation}\label{cov}
\begin{split}
\CD_e w=( \nabla_e w)^\top = \nabla_e w - (\nabla_e w\cdot N) N  = \nabla_e w +( w\cdot \nabla_eN) N, \qquad e,w \in T\SCB\\
\end{split}
\end{equation}
where $N= \frac{ \nabla(z-h(x))}{|\nabla(z-h(x))|}$, is the unit normal to $\SCB$.  In terms of coordinates $(x^1, x^2)\in \RR^2$
\begin{equation}\label{localcoor}
\SCB \cong \{\RR^2 ;  \quad g_{ij}=  \delta_{ij} + \p_i h \p_j h,  \quad \CD_i = \p_i + C_i \}.
\end{equation}
where $C_i$ are the Christoffel symbols matrices.

Recall that Sobolev spaces on $\mathbb{R}^2$ or $\mathbb{R}^3$ are given by the norm
$$
\| u \|_{H^k(\mathbb{R}^n)} = \left\| \< \p \>^k u \right\|_{L^2(\mathbb{R}^n)}
$$
We will now define Sobolev spaces on $\mathcal{D}$ and $\mathcal{B}$. First, the function spaces $H^k(\SD)$ are defined for any $k$ as restrictions:
$$
\|g\|_{H^k(\mathcal{D})} \overset{def}{=} \inf \left\{ \left\| G \right\|_{H^k(\mathbb{R}^3)}, G_{|\mathbb{D}} = g \right\}.
$$
Notice that this definition is equivalent for $k$ an integer to a more direct one, as can be seen using Lemma~\ref{hreg} below:
$$
\sum_{|s|\leq k} \left\| \p^s g \right\|_{L^2(\mathbb{D})} + \left\| \partial h \right\|_{H^{k-1}(\RR^2)} \left\| \p u \right\|_{L^\infty(\RR^2)} 
\sim \|g\|_{H^k(\mathcal{D})} + \left\| \partial h \right\|_{H^{k-1}(\RR^2)} \left\| \p u \right\|_{L^\infty(\RR^2)}
$$
For $k\in \NN$ and $0\le k\le 2s$, $H^k(\SCB)$ is given by the following norm
$$
\|u\|_{H^k(\SCB)} \overset{def}{=} \sum_{i=0}^k \|\CD^iu\|_{L^2(\SCB)}.
$$ 
By \eqref{A2}, for $3\le  k \le 2s$
\[
\|u\|_{H^k(\RR^2)} + \left\| \partial h \right\|_{H^{k-1}(\RR^2)} \left\| \p u \right\|_{L^\infty(\RR^2)}
\sim \|u\|_{H^k(\mathcal{B})} + \left\| \partial h \right\|_{H^{k-1}(\mathcal{B})} \left\| \p u \right\|_{L^\infty(\mathcal{B})}
\]
The next lemma shows that in  the $L^2(\SCB)$ norm, the Dirichlet to Neumann operator is the same as one derivative.

\begin{lemma}\label{harmonic0}
 Given a harmonic function $\phi$ on $\SD$  then the following are equivalent
\begin{equation*}
\|(\p_z\phi)_b\|_{L^2}\sim \|(\p_1\phi)_b\|_{L^2} +  \|(\p_2\phi)_b\|_{L^2}
\sim \|\CN \phi_b\|_{L^2} \sim \|\nabla\phi_b\|_{L^2} \sim \|\Lambda\phi_b\|_{L^2}.
\end{equation*}
\end{lemma}
\proof
Since $\phi$ is harmonic, then
\[\begin{split}
0 = \intl_\SCB \Delta \phi \p_z \phi
& = \intl_\SCB N \cdot \nabla \phi \p_z \phi
- \frac 12 N^3 | \nabla \phi |^2 \Rightarrow\\
\intl_\SCB N^3 |\p _ z \phi |^2
& =  \intl_\SCB N^3 [ (\p_{1} \phi)^2 + ( \p_{2} \phi )^2 ]
- 2\p_z\phi ( N^1 \p_{1} \phi + N^2 \p_{2} \phi ) .
\end{split}\]
Since $N^1$ and $N^2$ are $o(1)$ and $N^3$ is $O(1)$
then we  have the first equivalence.
The second equivalence follows from the definition of $\CN$ and from the relative sizes of $N^1, N^2$, and $N^3$
\[
\CN \phi_b = N \cdot (\nabla \phi )_b
 = N^3 (\p_z \phi)_b + N^1( \p_{x^1} \phi)_b + N^2 (\p_{x^2} \phi )_b.
\]
The third equivalence follows  from
\[
 \p_{i} \phi_b = ( \p_{i} \phi + \p_i h \p_z \phi)_b .
\]
The last equivalence follows from using the Fourier transform and Plancherel's Theorem.
\endproof
The lemma above shows that for $1 \le k \le 2s$ we have
\begin{equation}\label{eqhk}
 \|u\|_{H^k(\SCB)} \sim    \|u\|_{H^{k-1}(\SCB)} + \|\CN\CD^{k-1}u\|_{L^2(\SCB)}.
 \end{equation}
In particular  on $H^1(\SCB)$  we have the equivalent norms
$$ \|\varphi\|_{L^{2}(\RR^2)}  + \|\Lambda\varphi\|_{L^2(\RR^2)}\sim \|\varphi\|_{L^{2}(\SCB)} + \|\CN\varphi\|_{L^2(\SCB)}
$$
and by interpolation we can define $H^{\frac12}(\SCB)$ with equivalent norms
\[
 \|\varphi\|_{L^{2}(\RR^2)}  + \|\Lambda^{\frac 12}\varphi\|_{L^2(\RR^2)}\sim \|\varphi\|_{L^{2}(\SCB)} + \|\CN^{\frac 12}\varphi\|_{L^2(\SCB)}.
\]
The spaces $H^{k+1/2}(\SCB)$, for $1\le k\le 2s-\frac12$,  can be defined by
\begin{equation}\label{eqh1/2}
\|\varphi\|_{H^{k+ 1/2}(\SCB)}  \overset{def}{=} \|\varphi\|_{H^{k}(\SCB)} + \|\CN^{\frac 12}\CD^k\varphi\|_{L^2(\SCB)}.
\end{equation}

To show that  harmonic extensions gain  $1/2$ a derivative and that $\Delta_0^{-1}$ gains $2$ derivatives, in an appropriate range of spaces, 
we introduce an $H^{2s+1/2}$ coordinate system that maps  $\mathbb{R}^3_-  \to \mathscr{D}  $ and the $z=0$ plane to $\SCB$.
\begin{lemma}  
\label{hreg}
1.  There exists an extension $h\to \tilde h$ defined on $\mathbb{R}^3$ such that
\[
\|\nabla \tilde h\|_{H^{2s-1/2}(\RR^3)} \lesssim \|\p h\|_{H^{2s-1}(\RR^2)} \qquad \|\nabla \tilde h\|_{H^{s+1/2}(\RR^3)}
\lesssim \|\p h\|_{H^{s}(\RR^2)}
\]

2.  The map $ \tilde y =(\tilde x, \tilde z) \overset{\Psi}{\longrightarrow} y=(x,z) $ defined by
\begin{align*}
x &= \tilde x\\
z &= \tilde z + \tilde h(\tilde x,\tilde z).
\end{align*}
the $\tilde z=0$ plane to $\SCB$, and maps  the lower half space $\mathbb{R}^3_- \to \mathscr{D}$.  $\Psi$ is an $H^{2s+1/2}$ diffeomorphism  on $\mathbb{R}^3$, since
\begin{equation}\label{lippsi}
\|D\Psi- \rm{I}\|_{H^{2s-1/2}} < \infty, \quad \|D\Psi- \rm{I}\|_{H^{s+7/2}} \ll 1.
\end{equation}
\end{lemma}

\proof
Define $\tilde h$ via its Fourier transform $\hat {\tilde h}$ ($\xi$ denoting the dual variable to $(x_1,x_2)$ and $\eta$ the dual variable to $z$)
\[
\hat {\tilde h}(\xi, \eta) = c\frac{|\xi|^{2\alpha}}{(|\xi|^2 +\eta^2)^{\alpha + 1/2}}\hat h(\xi); \quad c^{-1} = \int \frac 1{(1+s^2)^{\alpha + 1/2}}ds.
\]
Then $\tilde h(x,0)=h(x)$ and for $1\le a+b\le 2s +1/2$
\begin{gather*}
|\xi|^a|\eta|^b\hat{\tilde h}(\xi, \eta)  = |\xi|^{a+b -1}\frac{\left(\frac{|\eta|}{|\xi|}\right)^b}
{\left(1+\left(\frac{|\eta|}{|\xi|}\right)^2\right)^{\alpha +1/2}}\hat h(\xi)\, \Rightarrow\,
\||\xi|^a|\eta|^b\hat{\tilde h}\|_{L^2}  \lesssim  \||\xi|^{a+ b -1/2} \hat h\|_{L^2}
\end{gather*}
By assumption \eqref{A2} and the definition of  $\tilde h$ {\it 2.}  follows.
\endproof

The next proposition states that for $0\le a \le 2s$,    {\it 1.}  $g\to g\circ\Psi$ is  one to one  and onto from $H^{a +\frac 12}(\SD)$  to $ H^{a +\frac12}(\mathbb{R}^3_-)$,   {\it 2.}  harmonic extension of  $f\in H^a(\SCB)$ gains $1/2$ derivative,  and {\it 3.}  $\nabla\Delta_0^{-1}\nabla$  is bounded on $H^{a-\frac 12}(\SD)$.

\begin{proposition} \label{elliptic} Under assumption \eqref{A2}  and  for $a\in \NN/2$  and $3 \le a \le 2s$  we have:\\

1. For $g\in H^{a+\frac12}(\SD)$
\begin{align*}
&\|  g\|_{H^{a+\frac12}(\SD) } \lesssim \| g\circ \Psi\|_{H^{a+\frac12}(\mathbb{R}^3_-) }+  \| \p h\|_{H^{a-1}(\mathbb{R}^2)} \| \nabla (g\circ\Psi)\|_{W^{[\frac a2] +1, \8}(\mathbb{R}^3_-)}\\
&\| g\circ \Psi \|_{H^{a+\frac12}(\mathbb{R}^3_-)} \lesssim  \| g\|_{H^{a+ \frac12}(\SD)} +  \| \p h\|_{H^{a-1}(\mathbb{R}^2)} \| \nabla g\|_{W^{[\frac a2] +1, \8}(\SD)}.
\end{align*}

2.  For $ f\in H^{a}(\SCB)$
\begin{alignat*}{2}
&\| \nabla f_\CH\|_{H^{a-\frac12}(\SD) }  && \lesssim    \| f\|_{H^{a }(\SCB)}+  \| \p h\|_{H^{a-1}(\mathbb{R}^2) } \|\nabla f\|_{W^{[\frac a2] +1, \8}};
\end{alignat*}

3.  For  $g\in H^{a-1/2}(\SD)$
\[
\begin{split}
\|\nabla \Delta^{-1}_0 \nabla g\|_{H^{a-\frac12}(\SD)} \lesssim & \|  g\|_{H^{a-\frac12}(\SD)}  + \|\p h\|_{H^{a-1}(\SCB)}\|\nabla\Delta^{-1}_0 \nabla g\|_{W^{[\frac a2] +1, \8}(\SD)}\\
 \lesssim & \|  g\|_{H^{a-\frac12}(\SD)}  + \|\p h\|_{H^{a-1}(\SCB)}\|\nabla g\|_{H^{[\frac a2] +2}(\SD)}
\end{split}
\]
\end{proposition}
\proof[Proof of  1]    For $a+1/2$  an integer the inequalities follow  immediately  from the chain rule.  For $a$ an integer, extend $g$ to a function $g_e\in H^{a+1/2}(\RR^3)$ with $ \|g_e\|_{H^{a+1/2}(\RR^3)}\le 2    \|g\|_{H^{a+1/2}(\SD)} $ and use  the equivalent norm on $H^{a+\frac12}(\RR^3)$
\[
\| \phi\|_{H^{a + \frac12}}^2 =   \| \phi\|_{H^{a}}^2 +\int\int\frac{ |D^a\phi(x)-D^a\phi(y)|^2}{|x-y|^{4}}dxdy,
\]
and \eqref{lippsi} to  conclude 1. \\

\noindent {\it Proof of 2. and 3.} For $g\in H^{a- \half}(\SD)$, solve $\Delta U = \nabla g $ on $\SD$ with Dirichlet data $f$. Change  variables using~$\Psi$
to flatten the domain, and let $\phi = U \circ \Psi$, and $\tilde g =  g \circ \Psi$;  then $D U = \left((I + A)D \phi \right)\circ \Psi^{-1}$
where the matrix $A$ depends on $\p \tilde h$, and thus  satisfies
\begin{equation}\label{eqA}
\|A\|_{H^{2s -1/2}( \RR^3_- ) } < \infty,  \qquad \|A\|_{H^{s + 5/2}( \RR^3_- )  }\ll 1 .
\end{equation}
Consequently $\phi$ satisfies
\[ \begin{split}
&\mbox{trace}[(D + AD)^2]\phi =\Delta \phi +2 \operatorname{div} (B D \phi)  + c\cdot D\phi = \nabla \tilde g +  A \nabla \tilde g \qquad x \in \RR^3_-\\[.5 em]
&\phi( x,0)  = \phi_b (x) = f\circ \Psi(x),
\end{split} \]
where $c$ depends on $\p^2 \tilde h$, and $B$ is a linear combination of $A$ and $A^2$.
By letting $\Delta^{-1}_*$ be the inverse of  $\Delta$  on $\RR^3_-$,  we write
\[
 D \phi =D  (\phi_b)_\CH +D   \Delta^{-1}_*\nabla \tilde g  + D \Delta^{-1}_*(B \nabla \tilde g  -c\cdot D\phi)  - 2\nabla \Delta^{-1}_* \operatorname{div} BD \phi.
\]
By standard elliptic estimates on $\RR^3_-$, and Sobolev embedding  we have
\[
\begin{split}
&\|\nabla  \phi \|_{{L^2} ( \RR^3_- ) } 
\lesssim   \| \phi_b \|_{ \dot H^{\frac12} ( \RR^2) } + \|\tilde g\|_{L^2 ( \RR^3_- ) }
+ \|BD \phi\|_{L^2 ( \RR^3_- ) }  +  \|B \nabla \tilde g\|_{L^{\frac65} ( \RR^3_- ) }  +  \|c\cdot D\phi\|_{L^{\frac65} ( \RR^3_- ) }.\\
&\|\nabla  \phi \|_{H^{a-\frac12} ( \RR^3_- ) } 
\lesssim   \| \phi_b \|_{ H^{a} ( \RR^2) }
+ \|BD \phi\|_{H^{a-\frac12} ( \RR^3_- ) }+ \|\tilde g\|_{H^{a-\frac12} ( \RR^3_- ) } \\
&\phantom{\|\nabla  \phi \|_{H^{a-\frac12} ( \RR^3_- ) } 
\lesssim } +  \|B \nabla \tilde g\|_{H^{a-\frac32} ( \RR^3_- ) }   +  \|B \nabla \tilde g\|_{L^{\frac65} ( \RR^3_- ) } +  \|c\cdot D\phi\|_{H^{a-\frac32} ( \RR^3_- ) }   +  \|c\cdot D\phi\|_{L^{\frac65} ( \RR^3_- ) }.
\end{split}
\]
By setting $g =0$, or $f=0$, and using \eqref{eqA}, we conclude {\it 2.} or  {\it 3.},  for the stated ranges of $a$.
\endproof

The next proposition shows that
\[
\|u\|_{H^q(\SCB)}\sim \sum_{i=0}^q \|\CN^iu\|_{L^2(\SCB)}
\]
Recall that $\tilde  \p_i = \p_i +h_i \p_z$, and    that for  a vector field $\mathcal{X}= X\cdot \nabla$ on $T\SCB$, the commutator
$[\mathcal{X}, \CN]$ acting on $\varphi: \mathscr{B}\to \RR$  is given by:
\begin{equation}\label{comm-na}
[\mathcal{X}, \CN ] \varphi  = N \cdot \nabla \Delta^{-1}_0(2  {\rm div}((\nabla X_\CH)\nabla\varphi_\CH))+ (XN ) \cdot ( \nabla \varphi_\CH)   -N \cdot ( D X_\CH D\varphi _\CH )
 \end{equation}
Thus by proposition  \ref{elliptic}
\begin{equation}\label{eqnd}
\|[\tilde \p, \CN] \varphi\|_{L^2}\sim \|(D^2h)D\varphi\|_{L^2}
\end{equation}
\begin{proposition}\label{harmonic}
For any $\varphi:\SCB\to \RR$ and $3\le q\le  2s$ with $q\in\mathbb{N}_0/2$
\begin{equation}\label{hqnorm}\begin{cases}
\|\CN^{q}\varphi\|_{L^2}   \lesssim \sum_{i=0}^q \|\Lambda^iu\|_{L^2(\RR^2)}  +   \| \p h\|_{H^{q-1}} \| \nabla \varphi\|_{L^\infty}\\[.5em]
\|\Lambda^{q} \varphi\|_{L^2} \lesssim  \sum_{i=0}^q \|\CN^iu\|_{L^2(\SCB)}  +   \| \p h\|_{H^{q-1}} \| \nabla \varphi\|_{L^\infty}
\end{cases}
\end{equation}
\end{proposition}
\proof  From \cite{SZ06} we have
\begin{equation} \label{E:deltacn1}
(-\Delta_{\SCB} - \CN^2) \varphi = \kappa \CN(\varphi) +2N\cdot \nabla
\Delta_0^{-1}\operatorname{div} (D N_\CH  D \varphi_\CH) - \CN(N) \cdot (\CN(\varphi) N +
\nabla^\top \varphi).
\end{equation}
Thus the operator $\Delta_{\SCB} +  \CN^2$  is a first order operator with coefficients depending on second derivatives of $h$.  
Consequently  if $q$ is even \eqref{hqnorm}  are immediate.  If $q$ is odd then the equation \eqref{E:deltacn1}  and Lemma \ref{harmonic0} imply \eqref{hqnorm}.
For $q=k +\frac12$,  we use the commutator estimate \eqref{eqnd} to write
\[
\|\CN^{q}\varphi\|_{L^2}^2 = \int \CN^{k+1}\varphi\CN^k\varphi = \int D^k\varphi\CN D^k\varphi+ \, \text{lower order terms}.
\]
>From equation \eqref{eqh1/2} we obtain the desired result.
\endproof

\section{Properties of the Dirichlet-Neumann operator}

\label{appendixpotdno}

Recall that we denote $\mathcal{N}(h)$ or simply $\mathcal{N}$ for the Dirichlet-Neumann operator associated to the domain $\mathcal{D} = \{ z \leq h(x) \}$ with boundary 
$\mathscr{B} = \{ z = h(x) \}$.

\subsection{$L^\infty$ estimate}

Next we bound $\CN$ in $L^\8(\SCB) \cong L^\8(\RR^2)$  by allowing the loss of  small powers of $\Lambda = |\nabla|$ for small and large frequencies. 
\begin{proposition}\label{prop:elliptic}
There exists $\epsilon_0$ such that: if
$$
\|\p h \|_\infty + \| \nabla^3 h \|_\infty + \|\p h \|_2 + \| \nabla^3 h \|_2< \epsilon_0,
$$
then for any $\varphi$ defined on $\mathscr{B}$, and for $0<\sigma<1/2$,
\begin{align}\label{eq:1}
&\|\CN \varphi \|_{W^{2,\infty}(\RR^2)} \lesssim \|D \varphi\|_{W^{3,\infty}} +   \|\Lambda^{1-\sigma}\varphi\|_{L^\infty}.
\end{align}
\end{proposition}

\noindent
\begin{proof} \underline{Step 1: the double layer potential} Let $a = (x, z) \in \SD$ and $b = (y, h(y) ) \in \SCB$.
Then $\varphi_{\CH}$ can be represented by the double layer potential
\[
\varphi_{\CH} (a) = \int\limits_\SCB \mu (b)N
\cdot \nabla G (a -b) dS (b)
= -1/2\mu (b_0) +  \int\limits_\SCB (\mu (b)- \mu (b_0) ) N
\cdot \nabla G (a -b) dS (b)
\]
where $G(a-b) = \frac 1{4\pi} |a-b|^{-1}$
is the Newtonian potential and $b_0$ is an arbitrary point on $\SCB$. 
Define $K (x,y) = \sqrt{1+|\p h(y)|^2}N(b) \cdot \nabla G (b_0-b)$
for $b_0= (x, h(x) ) \in \SCB$ and $b = (y, h(y) ) \in \SCB$. In other words,
$$
K (x,y) = \frac{{- \nabla} h(y) \cdot (x-y) + h(x) -h (y) }{ 4\pi (|x-y|^2 + (h(x) -h(y))^2) ^{\frac 32}}.
$$
Then as $a\to b_0$, it follows by standard singular integral calculations~\cite{Folland} that
\begin{equation}
\label{eq:e2}
-1/2 \mu (x) + \int \mu (y) K (x,y) dy = \varphi (x).
\end{equation}

\bigskip
\noindent \underline{Step 2: estimating $\mu$} Notice that $K(x,y)$ has  the following properties:
\begin{itemize}
\item $\int K(x,y) \,dy = 0$
\item $|K(x,y)| \lesssim \epsilon_0 \min \left[ \frac{1}{|x-y|} \,,\, \frac{1}{|x-y|^2} \right]$
\item $|K(x,y) - K(x',y)| \lesssim \epsilon_0 \frac{|x-x'|}{|y-x|^3}$ if $|y-x|>>|x-x'|$
\item $\int |K(x,y)|\,dy \lesssim \epsilon_0$.
\end{itemize}
The first point above can be seen by a standard integration by parts; the second and third points follow from $\|\p h\|_\infty < \epsilon_0$; for the fourth one use in
addition that $\|\p h\|_2 < \epsilon_0$. 

Next we also denote $K$ for the operator with kernel $K(x,y)$. By points one to three above, it is standard to see that
$$
\|K\|_{\dot{\mathcal{C}}^\alpha \rightarrow \dot{\mathcal{C}}^\alpha} \lesssim \epsilon_0 \quad \mbox{if $0<\alpha<1$};
$$
and point four easily implies
$$
\|K\|_{L^\infty \rightarrow L^\infty} \lesssim \epsilon_0.
$$
Choosing $\epsilon_0$ small enough, it is possible to solve~(\ref{eq:e2}) by Neumann series to obtain
$$
\|\mu\|_{L^\infty} \lesssim \|\varphi\|_{L^\infty} \quad \mbox{and} \quad \|\mu\|_{\dot{\mathcal{C}}^\alpha} \lesssim \|\varphi\|_{\dot{\mathcal{C}}^\alpha}
$$
To obtain estimates on derivatives of $\mu$, write
\begin{equation*}
 -1/2\mu (x)
+ \int( \mu (y)-\mu(x)) K (x,y)\, dy
= \varphi (x)  .
\end{equation*}
and differentiate with respect to $x$ to get
$$
-1/2\partial_x \mu (x) + \int (\mu (y)-\mu(x)) \partial_x K (x,y)\, dy = \partial_x \varphi (x).
$$ 
Denoting $J(x,y) = (\partial_x + \partial_y)K(x,y)$, an integration by parts gives (writing $J$ for the operator with kernel $J(x,y)$)
\begin{equation}
\label{equn}
-1/2\partial_x \mu (x) - K \partial_x \mu + J \mu = \partial_x \varphi.
\end{equation}
Using that $\|\nabla^2 h\|_\infty + \|\nabla^2 h\|_2 < \epsilon_0$, one checks that $J(x,y)$ enjoys the properties of $K(x,y)$ listed above. Thus if $0<\alpha<1$,
\begin{equation}
\label{eqdeux}
\|J\|_{L^\infty \rightarrow L^\infty} + \|J\|_{\dot{\mathcal{C}}^\alpha \rightarrow \dot{\mathcal{C}}^\alpha} \lesssim \epsilon_0.
\end{equation}
It is furthermore easy to see that
$$
\|J\|_{\mathcal{C}^{1-\sigma} \rightarrow L^\infty} + \|J\|_{\dot{\mathcal{C}}^{1-\sigma} \rightarrow L^\infty} + \lesssim \epsilon_0.
$$
Now equation~(\ref{equn}) implies
$$
\partial_x \mu (x) = \left(\frac{1}{2} + K \right)^{-1} \left( J \mu - \partial_x \varphi \right),
$$
from which follows, in conjunction with the various bound on $K$ and $J$, if $0<\alpha<1$,
\begin{equation*}
\begin{split}
& \left\| \partial_x \mu \right\|_\infty \lesssim \| J \mu \|_\infty + \| \partial_x \varphi \|_\infty \lesssim \|\varphi\|_{\dot{\mathcal{C}}^{1-\sigma}} + 
\| \partial_x \varphi \|_\infty \\
& \left\| \partial_x \mu \right\|_{\dot{\mathcal{C}}^{1-\sigma}} \lesssim \| J \mu \|_{\dot{\mathcal{C}}^{1-\sigma}} + \| \partial_x \varphi \|_{\dot{\mathcal{C}}^{1-\sigma}}
\lesssim \|\varphi\|_{\dot{\mathcal{C}}^{1-\sigma}} + \| \partial_x \varphi \|_\infty. \\
\end{split}
\end{equation*}
Taking one more derivative and arguing similarly gives
$$
\| \partial_x^2 \mu \|_{\infty} \lesssim \| \partial_x^2 \varphi \|_\infty + \|\varphi\|_{\dot{\mathcal{C}}^{1-\sigma}}.
$$

\bigskip
\noindent
\underline{Step 3: estimating $\CN\varphi$}
Fix a point $b_0\in \SCB$ and  use normal coordinates in a neighborhood of $\SCB$  to  restrict $a$ near the boundary to the line  $a = b_0 + \nu N (b_0 )$. Thus
\[
N(b_0)\cdot \nabla \varphi_{\CH} (a)
= \int\limits_\SCB (\mu(b) -\mu(b_0))
 D^2 G(N(b), N(b_0))(a-b) dS (b) .
\]
For $|b-b_0|$ large and $\nu$ small  $|D^2G(a-b)|\lesssim |b-b_0|^{-3}$ and thus the above integral can be bounded by $\|\mu \|_{\dot{\mathcal{C}}^{ 1- \sigma}} $.   
For $|b-b_0|$ small we write
\[
N(b_0) = \theta(b,b_0)N(b) + \gamma(b,b_0) \tau \qquad \text{where} \quad  \tau \in T_{b}S.
\]

The term involving $\tau$ is integrable due to the vanishing  of $\gamma(b_0,b_0)$.   By repeating the argument that led to inequality \eqref{eqdeux} we obtain

\[
| \int\limits_\SCB (\mu(b) -\mu(b_0))D^2 G(N(b),\gamma \tau)(b_0-b) dS (b) | \lesssim
\|\mu\|_{\dot{C}^{1-\sigma}} \lesssim \|\varphi\|_{\dot{C}^{1-\sigma}}.
\]

Split the remaining term
\[
I = \int\limits_\SCB (\mu(b) -\mu(b_0))D^2 G(N(b), N(b))(\nu N(b_0) +b_0-b) dS (b)
\]
with the help of a cutoff function $\chi$ which localizes smoothly away from the ball of center 0 and radius one:
$$
I = \int \left( \chi(b-b_0) + [1 - \chi(b-b_0)] \right) \dots dS(b) = II + III.
$$
The term $III$ is easily dealt with; as for $II$, it is hypersingular as $\nu \to 0$. Use the identity
\[
0 = \De G
= \De_{S} G
+ \kappa N
\cdot \nabla G
+ D^2 G (N, N)
\]
for $\nu < 0$ to re-express $II$ as
\begin{equation*}
\begin{split}
II  = & \int\limits_{\mathscr{B}} \CD \left( \chi(b-b_0) (\mu(b)-\mu(b_0)) \right) \CD G(\nu N(b_0) +b_0-b) \\
& \qquad \qquad \qquad \qquad - \chi(b-b_0) (\mu(b) -\mu(b_0)) \kappa N(b) \cdot \nabla G(\nu N(b_0) +b_0-b) \, dS(b).
\end{split}
\end{equation*}
Letting now $\nu \to 0$, this can be bounded as before
\[
\| I \| \lesssim  \| \p \mu \|_{C^1}
+ \| \mu  \|_{C^{1-\sigma} } \lesssim \| \p \varphi \|_{C^{1}}
+ \| \varphi \|_{\dot {C}^{1-\sigma}  }.
\]
Thus
\[
\|  \CN \varphi \|_{L^\infty(S) }
\le \| \p \varphi \|_{C^1 }
+ \| \varphi \|_{\dot{C}^{1-\sigma} }.
\]
By repeating the above argument after applying
tangential derivatives to $N(b_0)\cdot\nabla \varphi$ we obtain
\[
\|  \CN \varphi \|_{W^{2,\infty}(S) }
\le \| \p \varphi \|_{C^3 }
+ \| \varphi \|_{\dot{C}^{1-\sigma} }.
\]
This proves inequality \eqref{eq:1}. 
\end{proof}

\subsection{Expansion into multilinear $L^2$ bounded operators}

\begin{prop}
\label{propexpansion}
The operator $\mathcal{N}(h)$ can be expanded into a series of multilinear operators
$$
\mathcal{N}(h) = \sum_{n \geq 0} M_n (h,\dots,h,\cdot) \quad \mbox{i.e.} \quad \mathcal{N}(h)f = \sum_{n \geq 0} M_n (h,\dots,h,f)
$$
where the operators $M_n$ are symmetric, $n-linear$ in $h$, and satisfy the estimate
\begin{equation} 
\label{multilinbounds}
\left\| M_n(h,\dots,h,f) \right\|_2 \leq C_*^n 
  \left\|\p h \right\|_\infty^n \|\nabla f\|_2.
\end{equation}
\end{prop}
\begin{proof} Though we could not find this explicit statement, this result is classical in harmonic analysis.
The idea is to expand the Dirichlet-Neumann operator into elementary operators akin to Calderon's commutators.

\noindent \underline{Step 1: The single-layer potential}
Let $\psi$ be the solution of the Dirichlet problem with boundary data $f$:
$$
\Delta \psi = 0 \quad \mbox{in $\mathcal{D}$} \quad \mbox{and} \quad \psi(x,h(x)) = f(x) \quad \mbox{on $\mathscr{B}$}.
$$
It can be represented via a single layer potential by
\begin{equation*}
\psi(x,z) = \int_{\mathbb{R}^2} \frac{\rho(y)}{\left(|x-y|^2+|h(y)-z|^2\right)^{1/2}} dy
\end{equation*}
(this is not quite the standard definition, where the weight $\sqrt{1+|\p h|^2}$ would appear, reflecting the volume element for $\mathscr{B}$).
In particular,
\begin{equation}
\label{slp}
f(x) = \int_{\mathbb{R}^2} \frac{\rho(y)}{\left(|x-y|^2+|h(x)-h(y)|^2\right)^{1/2}} dy.
\end{equation}
It is well-known (see Folland~\cite{Folland}) that 
\begin{equation}
\label{lim}
\begin{split}
\mathcal{N}(h)\psi (x) & = \operatorname{lim}_{z \rightarrow h(x)} N_x \cdot \nabla \psi(z,h(x)) \\
& = \frac{1}{2} \frac{\rho(x)}{\sqrt{1+|\p h|^2(x)}} + \int_{\mathbb{R}^2} \frac{h(x) - h(y) + \p h(x) \cdot (x-y)}{\left( |x-y|^2 + |h(x) - h(y)|^2\right)^{3/2}} \rho(y) \,dy .
\end{split}
\end{equation}

\bigskip
\noindent \underline{Step 2: expanding $\rho$}
Start by expanding~(\ref{slp}) in $h$:
$$
f(x) = \sum_{n=0}^\infty \beta_n \int_{\mathbb{R}^2} \frac{|h(x)-h(y)|^{2n}}{|x-y|^{2n+1}}\rho(y)\,dy,
$$
where $|\beta_n|\leq 1$. Apply $\Lambda$ to the above, noting that $\Lambda = \frac{D}{\Lambda}\cdot D = \sum_{j=1,2} R_j \partial_j$. This gives
\begin{equation*}
\begin{split}
\Lambda f(x) & = \rho(x) + \sum_{n=1}^\infty \sum_{j=1,2} R_j \int_{\mathbb{R}^2} \frac{(h(x)-h(y))^{2n-1}}{|x-y|^{2n+1}} \\
& \qquad \qquad \qquad \qquad \left[2n \partial_j h(x)-(2n+1)\frac{(x^j-y^j)}{|x-y|}\frac{(h(x)-h(y))}{|x-y|}\right] \rho(y)\,dy \\
& \overset{def}{=} \rho(x) + \sum_{n=1}^\infty \left[ P_n \rho \right] (x).
\end{split}
\end{equation*}
Inverting the above by Neumann series gives
\begin{equation}
\label{toureiffel}
\rho(x) = \sum_{k=0}^{\infty} (-1)^k \left[ \sum_{n=0}^\infty P_n \right]^k \Lambda f(x).
\end{equation}

\bigskip
\noindent \underline{Step 3: expanding $\mathcal{N}(h) f$} 
The right-hand side of~(\ref{lim}) can be expanded in $h$ to give
\begin{equation}
\label{champselysees}
\begin{split}
\left[ \mathcal{N}(h) \psi \right] (x) & =  \frac{1}{2} \frac{\rho(x)}{\sqrt{1+|\p h|^2(x)}} + 
\sum_{n\geq 0} \alpha_n \frac{h(x)-h(y) + \p h(x) \cdot (x-y)}{|x-y|^3} \frac{|h(x)-h(y)|^{2n}}{|x-y|^{2n}} \rho(y) \,dy, \\
& \overset{def}{=}  \frac{1}{2} \frac{\rho(x)}{\sqrt{1+|\p h|^2(x)}} + \sum_{n=0}^\infty \left[ Q_n \rho \right] (x)
\end{split}
\end{equation}
where $|\alpha_n| \leq n$.

\bigskip
\noindent \underline{Step 4: boundedness of the elementary operators and conclusion}
The elementary operators $P_n$ and $Q_n$ appearing respectively in~(\ref{toureiffel}) and~(\ref{champselysees}) can be estimated by the following result 
of Coifman, MacIntosh, and Meyer~\cite{CMIM}
\footnote{This result corresponds to Theorem III in~\cite{CMIM}, transfered on $\mathbb{R}^2$ by the method of rotations.}: the operator on $\mathbb{R}^2$ with kernel
$$
\frac{(h_1(x)-h_1(y))\dots (h_n(x)-h_n(y))}{|x-y|^{n+2}} 
$$
(and $n$ odd) has a norm on $L^2$ which is bounded by
$$
C(1+n^4) \|\p h_1 \|_\infty \dots \| \p h_n \|_\infty.
$$
Using this estimate in conjunction with~(\ref{toureiffel}) and~(\ref{champselysees}) gives the desired conclusion.
\end{proof}

\subsection{Symmetries}

\begin{lem}
\label{lemmasymmetries}
The Dirichlet-Neumann operator is invariant by the following symmetries:
\begin{itemize}
\item Translation: $G(h(\cdot+\delta)) \left[ f(\cdot + \delta) \right] = \left[ G(h) f \right] (\cdot + \delta)$ for $\delta \in \mathbb{R}^2$.
\item Rotation: $G(h (R_\theta \cdot) ) \left[ f(R_\theta \cdot) \right] = \left[ G(h) f \right] (R_\theta \cdot)$ with $\theta \in \mathbb{R}$, where $R_\theta$ is the rotation by an angle~$\theta$.
\item Dilation: $G \left( \frac{1}{\lambda} h(\lambda \cdot) \right) \left[ f(\lambda \cdot) \right] = \lambda \left[ G(h) f \right] (\lambda \cdot)$ with $\lambda > 0$
\end{itemize}
\end{lem}

Consider the expansion of Proposition~\ref{propexpansion}. Applying for instance the rotation symmetry, and using in conjunction the above lemma gives
$$
\left[ G(h) f \right] (R_\theta \cdot)  = \sum_{n \geq 0} M_n
 (h(R_\theta \cdot) ,\dots,h(R_\theta \cdot) , f(R_\theta \cdot) )
$$
Differentiating in the continuous parameter ($\theta$ in the above example, or $\lambda$ or $\delta$ if the other symmetries are used), we obtain that
$$
\Gamma \left[ G(h) f \right] = \sum_{n \geq 0} n M_n
 (\Gamma h,\dots,h,f) + M_n(h,\dots,h,\Gamma f) 
$$
for $\Gamma = \nabla, \Omega$; and
\begin{equation}
\label{leibniz}
\Gamma \left[ G(h) f \right] = \sum_{n \geq 0} n M_n (\Gamma h,\dots,h,f) - (n+1) 
 M_n(h,\dots,h,\Gamma f) + G(h) f.
\end{equation}
for $\Gamma = x\cdot \nabla$. This last formula remains valid for $\mathcal{S} = \frac{3}{2} t \partial_t + x \cdot \nabla$ if the functions under consideration depend on time. 

Analogous formulas can of course be obtained if more than one vector field $\Gamma$ is applied.

\section{Tools from linear and multilinear harmonic analysis}

\label{appendixbobo}

\subsection{Littlewood-Paley theory}

Consider $\theta$ a function supported in the annulus $\mathcal{C}(0,\frac{3}{4},\frac{8}{3})$ such that
$$
\mbox{for $\xi \neq 0$,}\quad\sum_{j \in \mathbb{Z}} \theta \left( \frac{\xi}{2^j} \right) = 1 .
$$
Also define
$$
\Theta(\xi) \overset{def}{=} \sum_{j <0} \theta \left( \frac{\xi}{2^j} \right) .
$$
Define then the Fourier multipliers
$$
P_j \overset{def}{=} \theta\left( \frac{D}{2^j} \right) \quad P_{<j} = \Theta \left( \frac{D}{2^j} \right)\quad P_{\geq j} = 1- \Theta\left( \frac{D}{2^j} \right)
$$
and similarly $P_{\leq j}$, $P_{>j}$.
This gives a homogeneous and an inhomogeneous decomposition of the identity (for instance, in $L^2$)
$$
\sum_{j \in \mathbb{Z}} P_j = \operatorname{Id} \quad\mbox{and}\quad P_{\leq 0} + \sum_{j> 0} P_j = \operatorname{Id}.
$$
All these operators are bounded on $L^p$ spaces:
$$
\mbox{if $1 \leq p \leq \infty$,}\quad \|P_j f \|_p \lesssim \|f\|_p \quad,\quad \|P_{<j} f \|_p \lesssim \|f\|_p\quad\mbox{and} \quad\|P_{>j} f \|_p \lesssim \|f\|_p.
$$
Furthermore, for $P_j f$, taking a derivative is essentially equivalent to multiplying by $2^j$:
\begin{equation}
\label{LPderivative}
\begin{split} 
& \mbox{if $1 \leq p \leq \infty$ and $\alpha \in \mathbb{R}$,}\quad \|\Lambda^\alpha P_j f \|_p \sim 2^{\alpha j} \|P_j f\|_p\\
& \mbox{if $1 \leq p \leq \infty$ and $\ell \in \mathbb{Z}$,} \quad \|\nabla^\ell P_j f \|_p \sim 2^{\ell j} \|P_j f\|_p.
\end{split}
\end{equation}
Also, we  recall Bernstein's lemma: if $1\leq q\leq p \leq \infty$,
\begin{equation}
\label{lemmadeltaj}
\|P_j f \|_p \leq 2^{2j\left( \frac{1}{q}-\frac{1}{p} \right)} \left\| P_j f \right\|_q\quad\mbox{and}\quad \left\| P_{<j} f \right\|_p \leq 2^{2j\left( \frac{1}{q}-\frac{1}{p} \right)} \left\| P_{<j} f \right\|_q .
\end{equation}

\subsection{Boundedness of bilinear operators}

\label{appendixbilin}

Recall the definition of the bilinear operator with symbol $m$:
$$
T_m(f,g)(x) \overset{def}{=} \int_{\mathbb{R}^2} e^{ix \xi} \widehat{f}(\eta) \widehat{g}(\xi-\eta) m(\xi,\eta)\, d\xi d\eta.
$$
These operators are called pseudo-products and were introduced by Coifman and Meyer~\cite{CM}.

\begin{df}
Let $\mathcal{M}^{\beta,c_1,c_2,c_3}$ denote the set of bilinear symbols $m(\xi,\eta)$ such that
\begin{itemize}
\item $m$ is homogeneous of degree $\beta$: $m(\lambda \xi,\lambda \eta) = \lambda^\beta  m(\xi,\eta)$.
\item $m$ is smooth away from $\{\xi=0\} \cup \{\eta = 0\} \cup \{ \xi-\eta=0 \}$.
\item  If $|\xi|\ll|\eta| \sim 1 $, $m(\xi,\eta)$ can be written under the form $|\xi|^{c_1} \mathcal{A}\left( |\xi|,\frac{\xi}{|\xi|},\eta \right)$, 
where $\mathcal{A}$ is smooth in its arguments. 
\item If $|\eta|\ll|\xi| \sim 1$, $m(\xi,\eta)$ can be written under the form $|\eta|^{c_2} \mathcal{A}\left( |\eta|,\frac{\eta}{|\eta|} ,\xi \right)$, 
where $\mathcal{A}$ is smooth in its arguments.
\item  If $|\xi-\eta|\ll|\eta| \sim 1$, $m(\xi,\eta)$ can be written under the form $|\xi-\eta|^{c_3} \mathcal{A}\left( |\xi-\eta|,\frac{\xi-\eta}{|\xi-\eta|},\eta \right)$, 
where $\mathcal{A}$ is smooth in its arguments.
\end{itemize}
\end{df}

\begin{prop} \label{Bili-prop}
Let $\psi(\xi,\eta)$ be a smooth function, supported on an annulus, and let $m \in \mathcal{M}^{\beta,c_1,c_2,c_3}$, with $c_1,c_2,c_3>0$. 
Finally set
$$ 
\mu = \psi \left( \frac{(\xi,\eta)}{2^j} \right) m(\xi,\eta).
$$
Then
\begin{equation}\label{bili-est}
\left\| T_{\mu}(f,g) \right\|_p \lesssim 2^{\beta  j} \|f\|_q \|g\|_r
\end{equation}
if $1 \leq p,q,r \leq \infty$ and $\frac{1}{q} + \frac{1}{r} = \frac{1}{p}$.
\end{prop}

\begin{proof}
By scaling, it suffices to treat the case $j = 0$, $\beta = 0$. Consider thus the symbol $\psi \mu$. It is supported on a compact set. First consider the region inside this
compact set where none of $\xi$, $\eta$, $\xi-\eta$ vanish. Then $\psi \mu$ is smooth, and the proposition is clear. We are left with the three regions where
$|\eta|\ll|\xi|$, $|\xi|\ll|\eta|$, and $|\xi-\eta|\ll|\eta|$. By a duality argument and symmetry, it suffices to treat one of these three cases, say $|\eta|\ll|\xi|$.
Then $m$ can be written
$$
|\eta|^c \mathcal{A}\left( |\eta|,\frac{\eta}{|\eta|} ,\xi \right).
$$
Expand $\mathcal{A}$ in Taylor series in its first argument, and in Fourier series in its second argument $\frac{\eta}{|\eta|} \in \mathbb{S}^1$
(we omit the necessary cut-off function to alleviate the notations)
$$
\mathcal{A}\left( |\eta|,\frac{\eta}{|\eta|} ,\xi \right) = \sum_{k\in\mathbb{Z}} \sum_{\ell=0}^M |\eta|^{c+\ell} e^{i k \frac{\eta}{|\eta|}} \mathcal{A}_{k\ell} (\xi) + R_M.
$$
(notice that the smoothness of $\mathcal{A}$ entails fast decay of the $\mathcal{A}_{k \ell}$ in $k,\ell$). 
Taking $M$ sufficiently large, $R_M$ is sufficiently smooth for
the theory of Coifman-Meyer to apply, and we are left with $\sum_{\ell=0}^M$. We might as well consider only the first summand, $\ell = 0$, the other being treated similarly.
Thus we are left with
$$
\sum_{k\in\mathbb{Z}} |\eta|^{c} e^{i k \frac{\eta}{|\eta|}} \mathcal{A}_{k0} (\xi).
$$
Multiply this symbol by $\psi(\xi,\eta)$, and call the result 
$\rho(\xi,\eta) \overset{def}{=} \psi(\xi,\eta) \sum_{k\in\mathbb{Z}} |\eta|^{c} e^{i k \frac{\eta}{|\eta|}} \mathcal{A}_{k0} (\xi)$. Then
$$
T_\rho (f,g) = \sum_{m<0} T_{\rho}(P_m f,g) =  
\sum_{m<0} \sum_{k\in\mathbb{Z}} T_{\psi(\xi,\eta)\mathcal{A}_{k0}(\xi)} (\Lambda^c e^{i k \frac{D}{\Lambda}} P_m f,P_1g).
$$
Now by standard linear theory there exists $C$ such that
$$
\left\| \Lambda^c e^{i k \frac{D}{\Lambda}} P_m f \right\|_{p} \lesssim 2^{mc} k^C \|f\|_p
$$
valid for any $p$ in $[1,\infty]$. Summing over $m$ is possible since $c>0$; and the fast decay of the $\mathcal{A}_{k0}$ makes the sum over $k$ converge. This finishes the
proof.
\end{proof}

A simple consequence of the previous Proposition is the following 

\begin{cor} \label{Bili-prop1}
Let $m \in \mathcal{M}^{\beta,c_1,c_2,c_3}$ and assume that $\sigma_2$, $\sigma_3$, $q$, $r$, $Q$, $R$ satisfy
$$
c_1>0, \quad \sigma_2 < c_2, \quad \sigma_3 < c_3, \quad \mbox{and}\quad \frac{1}{q} + \frac{1}{r} = \frac{1}{Q} + \frac{1}{R} = \frac{1}{p}.
$$
Then for any $\kappa > 0$,
$$
\left\| T_m(f,g) \right\|_{L^p} \lesssim \left\|f\right\|_{\dot{W}^{\sigma_2,q}} \left\| \left[\Lambda^\kappa + \Lambda^{-\kappa} \right] 
\Lambda^{\beta-\sigma_2} g\right\|_{L^r} + \left\| \left[\Lambda^\kappa + \Lambda^{-\kappa} \right] \Lambda^{\beta-\sigma_3} f\right\|_{L^Q} 
\left\|g\right\|_{\dot{W}^{\sigma_3,R}}.
$$
\end{cor}

\begin{proof} As in the proof of Proposition~\ref{Bili-prop}, it suffices to treat the case where $m$ is supported in the region where $|\eta|<<|\xi|$. Then
$\frac{m(\xi,\eta)}{|\eta|^{\sigma_2}} \in \mathcal{M}^{\beta - \sigma_2,c_1,c_2-\sigma_2,c_3}$. 
Keeping the notation $\psi$ defined above, and applying Proposition~\ref{Bili-prop} one gets
$$
\left\| T_{\frac{m(\xi,\eta)}{|\eta|^{\sigma_2}} \psi \left( \frac{(\xi,\eta)}{2^j} \right) } \left( |\eta|^{\sigma_2} f , g \right) \right\|_{L^p}
\lesssim 2^{(\beta - \sigma_2)j} \left\| |\eta|^{\sigma_2} f \right\|_q \left\|P_j g \right\|_r.
$$
Summing over $j$ gives the desired result.
\end{proof}

\end{document}